\documentclass[12pt]{amsart}
\usepackage{amsthm}
\usepackage{amsmath}
\usepackage[colorlinks]{hyperref}
\usepackage{epsfig}
\usepackage{accents}
\usepackage{geometry}
\usepackage{tikz}
\usepackage%[american]
{circuitikz}
\usetikzlibrary{arrows}
\usetikzlibrary{decorations.pathreplacing,arrows.meta}
\usepackage{amssymb,latexsym,cite,epsf,graphics}
\usepackage{graphicx,xcolor}
\usepackage{enumitem}
\usepackage{comment}

\usepackage{subcaption}
\usepackage[ruled,vlined,linesnumbered]{algorithm2e}
\definecolor{darkred}{rgb}{1,0,0}

\newtheorem{theorem}{Theorem}[section]

\newtheorem{lemma}[theorem]{Lemma}

\newtheorem{proposition}[theorem]{Proposition}

\newtheorem{corollary}[theorem]{Corollary}

\newtheorem*{theorem*}{Theorem}

\theoremstyle{remark}

\newcommand{\mscommnew}[1]{\begingroup%\color{teal}
#1\endgroup}
\newcommand{\shcomm}[1]{\begingroup%\color{green}
#1\endgroup}

\theoremstyle{definition}

\newtheorem{remark}[theorem]{Remark}
\newtheorem{example}[theorem]{Example}
\newtheorem{definition}[theorem]{Definition}
\newtheorem*{notation*}{Notation}

\title{Matrix-tree theorem\\ for cohomological electrical networks
}

\author{Pavlo Pylyavskyy}
\address{\hspace{-.3in} Department of Mathematics, University of Minnesota,
Minneapolis, MN 55414, USA}
\email{ppylyavs@umn.edu}

\author{Svetlana Shirokovskikh}
\address{\hspace{-.3in} HSE University, Moscow, Russia}
\email{sveta.17.10\,@\,yandex$\cdot $ru}

\author{Mikhail Skopenkov}
\address{\hspace{-.3in} King Abdullah University of Science and Technology}
\email{mikhail.skopenkov\,@\,gmail$\cdot $com}

\begin{document}

\thanks{The first author was partially supported by a grant from Simons Foundation International SFI-MPS-SFM-00011393 P.P.. The second author was supported by the Theoretical Physics and Mathematics Advancement Foundation ``BASIS''. The third author was supported by KAUST baseline.
}


\begin{abstract}
We introduce a new type of boundary condition for electrical networks by specifying a cohomology class on the underlying cell complex. As special and limiting cases, we recover the Dirichlet boundary condition, multiport condition, and prescribed voltage drops along nontrivial cycles on a surface. In this setting, we obtain a combinatorial formula for minors of the response matrix, generalizing the matrix-tree theorem, the Kenyon--Wilson formula, and a recently discovered formula for networks on surfaces due to Lam et al. To prove this formula, we develop a statistical-physics toolbox, including a new parafermionic observable for the uniform spanning tree model. This is the first parafermionic observable that is intrinsic, in the sense that it does not require an embedding of the graph in the plane or on a surface. Other ingredients are an intrinsic Temperley correspondence and source-synchronized networks, the latter providing a completely elementary reformulation of our results.

%\textbf{Keywords and phrases.} Electric network, cohomological network, parafermionic observable, network response, matrix-tree theorem

%05C82, 05C22, 94C05, 31C20, 35R02, 52C20
\end{abstract}

\ \vspace{-.1in}

\maketitle

\vspace{-.4in}

\setcounter{tocdepth}{1}
\tableofcontents

 \vspace{-.4in}

\section{Introduction}

Electrical networks are usually equipped with standard Dirichlet and Neumann boundary conditions in mathematical literature \cite{CurtisMorrow,Doyle-Snell-84,ZSU}, whereas many applications require non-standard ones. For instance, in multiport networks, one prescribes the voltage drop and sets the net current to zero within each port \cite{BIK,PSS}. For networks on surfaces, one prescribes voltage drops along topologically nontrivial loops, resulting in a discretization of Riemann surfaces \cite{BobenkoSkopenkov+2016+217+250}. Special boundary conditions also arise in electromagnetic circuits~\cite{Milton-Seppecher-10,Milton-Seppecher-10B}. In this work, we introduce a general framework for treating such boundary conditions intrinsically on arbitrary graphs and, more generally, cell complexes. See Figs.~\ref{fig:magnet} and~\ref{fig:boundary-conditions-examples}.

\begin{figure}[t]
\centering

% --- Картинка 1 ---
\begin{tikzpicture}[scale=0.8]%[scale=0.7]
    \coordinate (A) at (0,0);
    \coordinate (B) at (1,0);
    \coordinate (C) at (1,1);
    \coordinate (D) at (0,1);
    \coordinate (E) at (0,-1);
    \coordinate (F) at (1,-1);
    \coordinate (G) at (2.5,-2);
    \coordinate (H) at (2.5,-4);

    \fill[blue] (A)--(B)--(C)--(D)--cycle;
    \fill[red] (A)--(B)--(F)--(E)--cycle;
    \draw[->, thick] (1.5,1) -- (1.5,-1);
    \draw (0.5,-2) ellipse [x radius=2cm, y radius=0.7cm];
    \draw (0.5,-4) ellipse [x radius=2cm, y radius=0.7cm];

    \node at ($(A)!0.5!(C)$) {N};
    \node at ($(A)!0.5!(F)$) {S};

    \fill[color=black] (G) circle (2pt);
    \fill[color=black] (H) circle (2pt);
    \draw (G)--(H);

    \draw[->,thick] (-0.914,-2.7) arc[start angle=225, end angle=315, x radius=2cm, y radius=0.7cm];
    \draw[->,thick] (-0.914,-4.7) arc[start angle=225, end angle=315, x radius=2cm, y radius=0.7cm];
\end{tikzpicture}%
\hspace{0.02\textwidth}%
% --- Картинка 2 ---
\begin{circuitikz}[scale=0.8]%[scale=0.7]
    \draw (1,0) ellipse [x radius=1cm, y radius=2cm];
    \draw[-{Latex}] (1,-2) arc[start angle=270, end angle=315, x radius=1cm, y radius=2cm];
    \node at (1,2.4) {$e_1$};
    \draw (0,0) to[battery2,l={$\mathcal{E}$}] (0,0);
    \coordinate (H2) at (2,0);
    \fill[color=black] (H2) circle (2pt);

    \coordinate (G2) at (4,0);
    \fill[color=black] (G2) circle (2pt);
    \draw (5,0) ellipse [x radius=1cm, y radius=2cm];
    \draw[-{Latex}] (5,-2) arc[start angle=270, end angle=225, x radius=1cm, y radius=2cm];
    \draw (6,0) to[battery2,l={$\mathcal{E}$}] (6,0);
    \draw (2,0) -- (4,0) node[midway, above] {$e_2$};
    \node at (5,2.4) {$e_3$};

\end{circuitikz}%
\hspace{0.02\textwidth}%
% --- Картинка 3 ---
\begin{tikzpicture}[scale=0.8]%[scale=0.7]
\def\xr{2}  % горизонтальный радиус
\def\yr{0.5} % вертикальный радиус
\def\h{3}  % высота цилиндра

% верхнее основание — заливка светлее
\fill[blue!10] (0,0) ellipse [x radius=\xr, y radius=\yr];

% верхнее основание — контур
\draw (0,0) ellipse [x radius=\xr, y radius=\yr];

% передняя поверхность с заливкой
\fill[blue!40]
  (-\xr,-\h) arc[start angle=-180,end angle=0,x radius=\xr,y radius=\yr] % нижний эллипс
  -- (\xr,0) arc[start angle=0,end angle=-180,x radius=\xr,y radius=\yr] % верхний эллипс
  -- cycle;

% нижнее основание (видимая и невидимая часть)
\draw (-\xr,-\h) arc[start angle=180,end angle=360,x radius=\xr,y radius=\yr];
\draw[dashed] (-\xr,-\h) arc[start angle=180,end angle=0,x radius=\xr,y radius=\yr];

% боковые линии
\draw (-\xr,0) -- (-\xr,-\h);
\draw (\xr,0) -- (\xr,-\h);
\node at (0,1.5) {$\mathcal{E} \in H^1(S, \mathbb{R})$};
\end{tikzpicture}
\caption{A physical system modeled by an electrical circuit with non-standard boundary conditions. Left: A magnet quickly inserted into a circuit composed of two close identical loops induces equal source voltages in both loops. Middle: A model of this system is a
\emph{source-synchronized circuit}, that is, the one with the voltage sources forced to have the same voltage $\mathcal{E}$. %\mscomm{Add labels $e_1$, $e_2$, $e_3$ for the purpose of Example~\ref{ex:cylinder}!}
Right: A more conceptual model is a \emph{cohomological circuit} on a cylinder, with the source voltage represented as a cohomology class~$\mathcal{E}$. The source voltage on any cycle on the cylinder equals the evaluation of the cohomology class $\mathcal{E}$ on the cycle.
}
\label{fig:magnet}
\end{figure}
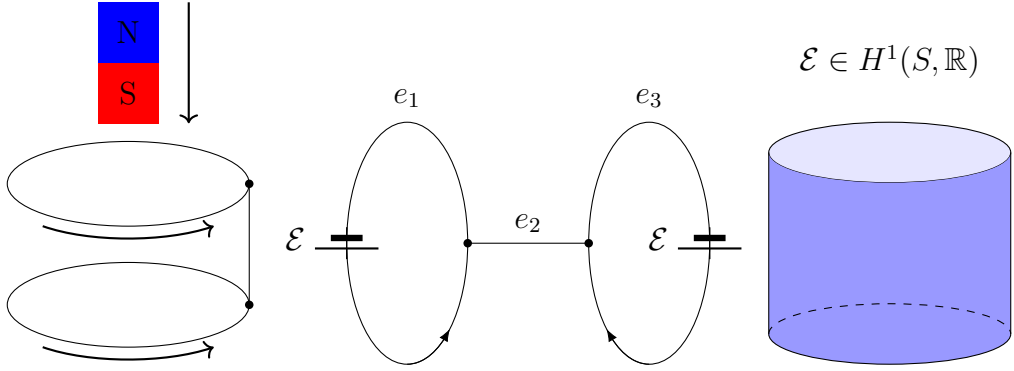

This framework traces back to the original work by Kirchhoff~\cite{Kirchhoff-47}. He conceptualized electrical networks as graphs whose edges represent network elements and vertices represent their junctions.
Notably, he allowed several edges with voltage sources, so that his setup did not reduce to the usual Dirichlet boundary conditions. We slightly generalize his setup by introducing \emph{synchronized sources}, that is, enforcing some of the sources to have the same voltage, and then put such boundary conditions into a more conceptual \emph{cohomological} form. See Sec.~\ref{sec-networks}.

This idea is motivated by the following example. Consider a cylinder with an electrical network embedded on its surface; see Fig.~\ref{fig:magnet}. If a magnet is quickly inserted into the cylinder, it induces a current in the network even in the absence of batteries. Here, the source voltage is not given by a potential difference between vertices. Instead, it manifests as a nonzero total voltage drop along non-contractible cycles generating the first homology of the cylinder. By contrast, the voltage drop along any contractible loop remains zero
(see Definition~\ref{def-surface-circuit} below). This is a natural example where the source voltage is encoded by a cohomology class rather than a function on vertices. We further extend the set-up to general cell complexes.

These \emph{cohomological} boundary conditions lead to rich theory, partially inspired by similar recent developments in hydrodynamics~\cite{Chern-etal-23}.
We obtain analogs of the existence and uniqueness theorem (Sec.~\ref{sec-networks}), formulae for currents, voltages, and response matrix (Secs.~\ref{sec-response} and~\ref{sec-voltages}).

The matrix-tree theorem is one of the main graph-theoretic results on electrical networks. Its origins go back to the same work of Kirchhoff \cite{Kirchhoff-47}. The theorem expresses the determinant of a so-called Kirchhoff matrix as a sum over spanning trees, resulting in a combinatorial formula for the currents in the network. See, e.g., \cite[Section~2]{PSS}. There are generalizations to other network elements \cite{Chen-72,robichaud1961graphes,seshu1961linear}, directed graphs \cite{chan1969introductory}, other boundary conditions \cite{PSS}, hypergraphs \cite{Masbaum-Vaintrob}, and higher dimensions \cite{Catanzaro-etal-15,Duval-etal-09}.
Chaiken \cite{chaiken1982} has proved a far-reaching extension, known as the all-minors matrix-tree theorem, see also~\cite{CK}.
For electrical networks, the response matrix is a natural relative of the Kirchhoff matrix. Curtis, Ingerman, and Morrow \cite{curtis1998,CurtisMorrow} found similar interpretations for a specific subset of minors of the response matrix, in the context of an inverse problem \cite{Milton-Seppecher-08,Skopenkov-15,Rote}. Then Kenyon and Wilson~\cite{KW} gave a formula for all minors. Recently, Lam et al.~\cite{Lam-25} have obtained analogous formulae for networks on surfaces; the proof was elegant but strongly tied to a surface embedding and very different from our methods.

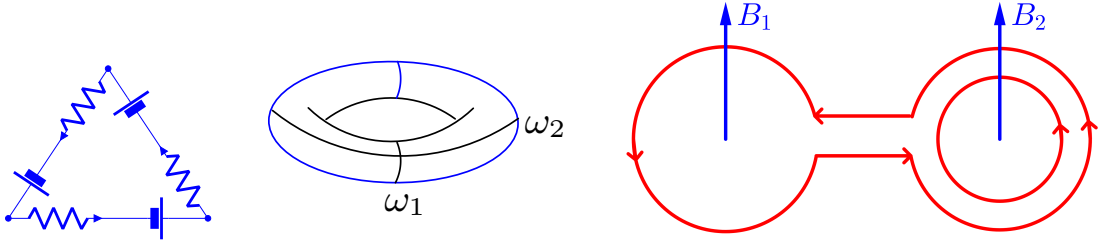
\begin{figure}[t]
    \centering
    \scalebox{0.66}{\begin{tikzpicture}
\draw[color=blue]
(0,0) to[R, thick, i={$\quad$}] (2,0)
(4,0) to[battery2, *-, thick] (2,0)
(4,0) to[R, thick, i={$\quad$}] (3,1.5)
(2,3) to[battery2, *-, thick] (3,1.5)
(2,3) to[R, thick, i={$\quad$}] (1,1.5)
(0,0) to[battery2, *-, thick] (1,1.5);
\end{tikzpicture}}
\quad
%%%%%%%%%%%
\scalebox{1.6}{\begin{tikzpicture}[line cap=round,line join=round,>=triangle 45,x=2.0cm,y=2.0cm]
\clip(-3.59,1.26) rectangle (-2.3,2.7);
\draw [color=blue, rotate around={-0.77:(-3.01,1.75)}] (-3.01,1.75) ellipse (1.03cm and 0.5cm);
\draw [shift={(-3.02,2.16)}] plot[domain=3.94:5.48,variable=\t]({1*0.49*cos(\t r)+0*0.49*sin(\t r)},{0*0.49*cos(\t r)+1*0.49*sin(\t r)});
\draw [shift={(-3.02,1.4)}] plot[domain=0.94:2.21,variable=\t]({1*0.45*cos(\t r)+0*0.45*sin(\t r)},{0*0.45*cos(\t r)+1*0.45*sin(\t r)});
\draw [shift={(-3.18,1.59)}] plot[domain=-0.47:0.39,variable=\t]({1*0.2*cos(\t r)+0*0.2*sin(\t r)},{0*0.2*cos(\t r)+1*0.2*sin(\t r)});
\draw [color=blue, shift={(-3.18,1.93)}] plot[domain=-0.42:0.36,variable=\t]({1*0.2*cos(\t r)+0*0.2*sin(\t r)},{0*0.2*cos(\t r)+1*0.2*sin(\t r)});
\draw [shift={(-2.98,2.45)}] plot[domain=4.03:5.33,variable=\t]({1*0.84*cos(\t r)+0*0.84*sin(\t r)},{0*0.84*cos(\t r)+1*0.84*sin(\t r)});
\begin{scriptsize}
\draw[color=black] (-2.96, 1.41) node {$\omega_1$};
\draw[color=black] (-2.38, 1.73) node {$\omega_2$};
%\draw[color=blue] (-2.88, 1.91) node {$e_1$};
%\draw[color=blue] (-3.3, 2.04) node {$e_2$};
\end{scriptsize}
\end{tikzpicture}
}
%%%%%%
% \quad\scalebox{1.6}{\begin{tikzpicture}[line cap=round,line join=round,>=triangle 45,x=2.0cm,y=2.0cm]
% \clip(-3.59,1.26) rectangle (-2.3,2.7);
% \draw [rotate around={-0.77:(-3.01,1.75)}] (-3.01,1.75) ellipse (1.03cm and 0.5cm);
% \draw [shift={(-3.02,2.16)}] plot[domain=3.94:5.48,variable=\t]({1*0.49*cos(\t r)+0*0.49*sin(\t r)},{0*0.49*cos(\t r)+1*0.49*sin(\t r)});
% \draw [shift={(-3.02,1.4)}] plot[domain=0.94:2.21,variable=\t]({1*0.45*cos(\t r)+0*0.45*sin(\t r)},{0*0.45*cos(\t r)+1*0.45*sin(\t r)});
% \draw [shift={(-3.18,1.6)}] plot[domain=-0.49:0.42,variable=\t]({1*0.2*cos(\t r)+0*0.2*sin(\t r)},{0*0.2*cos(\t r)+1*0.2*sin(\t r)});
% \draw [shift={(-2.98,2.45)}] plot[domain=4.03:5.33,variable=\t]({1*0.84*cos(\t r)+0*0.84*sin(\t r)},{0*0.84*cos(\t r)+1*0.84*sin(\t r)});
% \begin{scriptsize}
% \draw[color=black] (-2.96,1.41) node {\tiny  $\omega_1$};
% \draw[color=black] (-2.38,1.73) node {\tiny  $\omega_2$};
% % \draw[color=blue] (-2.96,1.41) node {$e_1$};
% % \draw[color=blue] (-2.38,1.73) node {$e_2$};
% \end{scriptsize}
% \end{tikzpicture}}
\quad\includegraphics[width=0.37\textwidth]{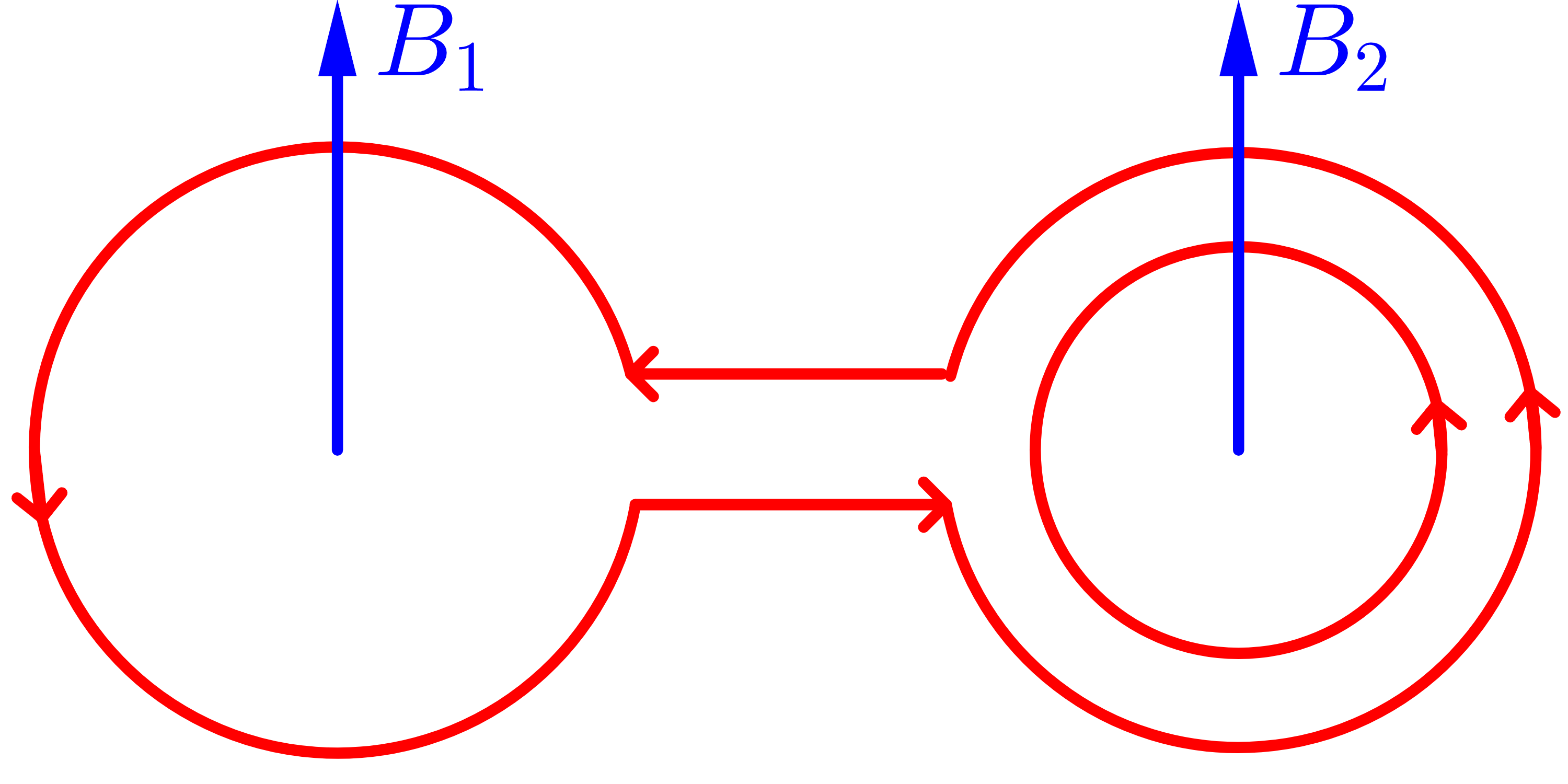}
    \caption{Electrical circuits with non-standard boundary conditions (from left to right): Kirchhoff's circuit with several batteries;
    %(essentially equivalent to a multiport network);
    a circuit on a surface with prescribed voltage drops along %topologically
    nontrivial cycles; % $\omega_1$ and $\omega_2$;
    electromagnetic circuit \cite{Milton-Seppecher-10}.} % by Milton and Seppecher.}
    \label{fig:boundary-conditions-examples}
\end{figure}

Our main results (Theorem~\ref{th-all-minor-matrix-tree} and Corollary~\ref{cor-all-minor-matrix-tree} in Sec.~\ref{sec-all-minor}) give combinatorial formulae for minors of the response matrices of source-synchronized and cohomological networks. While the classical matrix-tree theorems involve spanning trees and forests, our formulas are expressed in terms of spanning subgraphs that contain cycles, similarly to \cite{Kenyon-11,KL22}. The coefficients are given by determinants of certain Gram matrices, encoding the pairing between cycles in these subgraphs and the synchronized sources of the network. In the cohomological setting, this takes a particularly natural form: the source voltage is represented by a cohomology class, and the coefficients record its pairing with the homology classes of cycles in the contributing subgraphs.

For the proof of the main result, we use versions of the following tools from statistical physics:
\begin{itemize} [leftmargin=1.8em]
    \item
    \emph{Curtis--Morrow's connection decomposition} for the response matrix minors;
    \item
    \emph{Smirnov's parafermionic observable} for the uniform spanning tree model;
    \item \emph{Temperley's correspondence} between spanning trees and dimer configurations.
\end{itemize}

We view this toolbox as even a more important contribution than our main theorem itself.

\emph{Curtis--Morrow's connection decomposition} expresses minors of the response matrix in terms of sums over connections between sources. Originally, it was used to show that the signs of those minors characterize response matrices of planar electrical networks (with Dirichlet boundary conditions) \cite{CurtisMorrow} and was one of motivations for studying those minors. The decomposition is valid for arbitrary, not necessarily planar graphs. We generalize it to different boundary conditions (Lemma \ref{l-connections}), which requires additional ideas. See Sec.~\ref{sec-connection-decomposition}.

The resulting sums over connections are then interpreted as \emph{Smirnov's parafermionic observables} for the (weighted) uniform spanning tree model (Lemma~\ref{connection_F_and_L}), which are our main new tool. Such observables for various 2-dimensional models have made a revolution in statistical physics over the last 25 years \cite{Duminil-Copin-13}. Traditionally, they are related to 2-point correlation functions.
A (different) construction of such 2-point observable was also known for the uniform spanning tree model. Our observable is its multi-point generalization, which therefore interprets minors of the response matrix as multi-point correlation functions; cf.~\cite{KL22}. Such multi-point generalizations have been available for a very limited number of models \cite{chelkak2017revisiting,SU-22, SkoUst-24, KSS-25} (including the analogs
for \emph{Minkowskian} metric.) Unlike all previous approaches, our construction is intrinsic: it does not require embedding the graph in the plane or on a surface. See Sec.~\ref{sec-parafermionic}.

Finally, \emph{Temperley's correspondence} between spanning trees and dimer configurations translates the observable into the language of dimers and double-dimers. The original correspondence \cite{Temperley1974, KPW} linked spanning trees in a planar graph to dimer configurations on the union of the (subdivided) graph and its dual. It has been generalized to graphs on surfaces \cite{sun2016, ber2024}. Our variation links spanning trees to a pair of dimer configurations on the same (subdivided) graph (Lemma~\ref{lemma:bijection}), which is similar to the \emph{gliding} from \cite{KPW,Byun-Ciucu-24}, but no longer requires embedding in the plane or a surface. See Sec.~\ref{sec-Temperley}.
This allows to finish the proof of the main theorem in Sec.~\ref{sec-conclusion}.

To summarize, the present paper follows the same paradigm as recent papers by Khristoforov, Smirnov, et al.~\cite{KS-20,KSS-25}: for a known 2-point parafermionic observable, we find an equivalent construction that places it within the general unified framework, generalizing it to a multi-point one (in our case, also to higher dimensions) and leading to new applications. In Sec.~\ref{sec-program}, we propose a \emph{research program on such higher-dimensional $n$-point parafermionic observables.}

\section{Networks}
\label{sec-networks}

In this section, we introduce two related boundary conditions for electrical networks. The first (given in Secs.~\ref{ssec-electrical networks}--\ref{ssec-source-synchronized networks}) is a slight generalization of the one from the original Kirchhoff's work \cite{Kirchhoff-47} on electrical networks; it is elementary and convenient for our proofs. The second
(given in Secs.~\ref{ssec-networks on surfaces}--\ref{ssec-cohomological networks}) is more conceptual but less elementary; it gives the basis for the paper. Each boundary condition is first introduced in the simplest and most visual particular case (Secs.~\ref{ssec-electrical networks} and~\ref{ssec-networks on surfaces}) and then generalized in a concise but more abstract way (Secs.~\ref{ssec-source-synchronized networks} and~\ref{ssec-cohomological networks}).

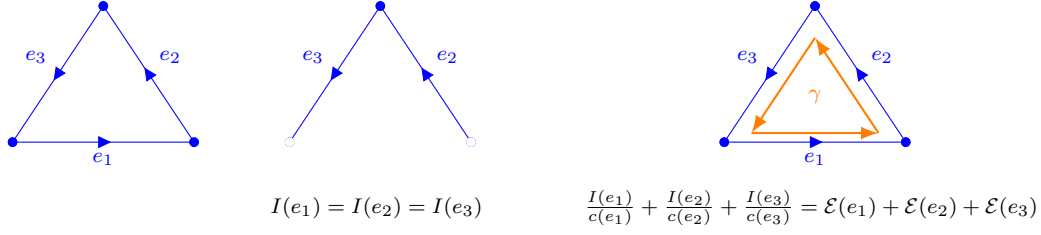
\begin{figure}
    \centering
\begin{tabular}{ccc}
\begin{tikzpicture}[scale=0.6]
  \draw[color=blue]
  (0,0) to[short, *-*, i_={\tiny $e_1$}] (4,0)
  (4,0) to[short, *-*, i_={\tiny $e_2$}] (2,3)
  (2,3) to[short, *-*, i_={\tiny $e_3$}] (0,0);
  \end{tikzpicture}
  &
\begin{tikzpicture}[scale=0.6]
\draw[color=blue]
%(0,0) to[short, *-*, i_={\tiny $e_1$}] (4,0)
(4,0) to[short, *-*, i_={\tiny $e_2$}] (2,3)
(2,3) to[short, *-*, i_={\tiny $e_3$}] (0,0);
\draw[color=white]
(0,0) to[short, *-*, i_={\tiny $e_1$}] (4,0);
%\node[text=blue] at (2,-1) {\tiny $I(e_1)=I(e_2)=I(e_3)$};
%{[anchor=west] (0,0) node[color=blue] {\tiny $I(e_2)=I(e_3)$}};
\end{tikzpicture}
&
\begin{tikzpicture}[scale=0.6]
\draw[color=blue]
(0,0) to[short, *-*, i_={\tiny $e_1$}] (4,0)
(4,0) to[short, *-*, i_={\tiny $e_2$}] (2,3)
(2,3) to[short, *-*, i_={\tiny $e_3$}] (0,0);
  \draw[-{Latex}, thick, draw=orange] (2,2.3) -- (0.6,0.2);
  \draw[-{Latex}, thick, draw=orange] (3.4,0.2) -- (2,2.3);
  \draw[-{Latex}, thick, draw=orange] (0.6,0.2) -- (3.4,0.2);
  % \draw[->, thick] (B) -- (C) node[midway, right] {$e$};
  % \draw[->, thick] (C) -- (D) node[midway, below left] {$d$};
  % \draw[->, thick] (D) -- (A) node[midway, left] {$a$};
        \node[text=orange] at (2,1) {\tiny $\gamma$};
        %\node[text=blue] at (2,-1.5) {\tiny $\frac{I(e_1)}{c(e_1)}-\mathcal{E}(e_1)\,+\frac{I(e_2)}{c(e_2)}-\mathcal{E}(e_2)\,+\frac{I(e_3)}{c(e_3)}-\mathcal{E}(e_3)=0$};
\end{tikzpicture}\\
 & {\tiny \quad\quad $I(e_1)=I(e_2)=I(e_3)$ \quad\quad}
 & {\tiny \quad $\frac{I(e_1)}{c(e_1)}+\frac{I(e_2)}{c(e_2)}+\frac{I(e_3)}{c(e_3)}=\mathcal{E}(e_1)+\mathcal{E}(e_2)+\mathcal{E}(e_3)$ \quad}
 %& {\tiny \quad $\frac{I(e_1)}{c(e_1)}-\mathcal{E}(e_1)\,+\frac{I(e_2)}{c(e_2)}-\mathcal{E}(e_2)\,+\frac{I(e_3)}{c(e_3)}-\mathcal{E}(e_3)=0$ \quad}
\end{tabular}
    \caption{An electrical circuit with three edges $e_1$, $e_2$, and $e_3$ (left), Kirchhoff's current law (middle), and Kirchhoff's voltage law (right). A cycle $\gamma$ is in orange.} %\mscomm{Rearrange as a 3x2 table.}}
    \label{fig:electrical-network}
\end{figure}

\subsection{Electrical networks}
\label{ssec-electrical networks}

We begin by recalling Kirchhoff's original definition of electrical circuits, restated in modern terminology. See Fig.~\ref{fig:boundary-conditions-examples} to the left.

An electrical circuit is viewed as a graph whose edges conduct electric current and %also 
carry some voltage sources, such as batteries.
The presence of sources leads to natural edge directions (from the negative poles of batteries to the positive ones). Thus, in what follows, we only consider directed graphs, and an \emph{edge} always means a directed edge of the graph. This is a slight abuse of common convention for circuits, but is typical for discrete field theory~\cite{Skopenkov+2023}. We allow loops and multiple edges in the graphs.
By a \emph{simple cycle} in a graph, we mean a simple directed cycle in a graph obtained by reversing the direction of some of the edges of the graph.

\begin{definition}\label{def-electrical} (See Fig.~\ref{fig:electrical-network})
An \emph{electrical circuit} is a finite %weakly
connected directed graph with two numbers, $c(e)>0$ (\emph{conductance}) and $\mathcal{E}({e})\in\mathbb{R}$ (\emph{source voltage}), assigned to each %oriented
edge $e$.

The \emph{current} %in the circuit
is the real %-valued
function $I$ on the set of %oriented
edges determined by the following axioms:
\begin{itemize}
    \item[(I)]\textit{The Kirchhoff current law.} For each vertex $v$, we have  (see Fig.~\ref{fig:electrical-network}, middle)
    $$
    \sum_{e\textrm{ \shcomm{starting at} }v}I(e)-\sum_{e\textrm{ \shcomm{ending at} }v}I(e)=0,
    $$
    where the sums are over all %oriented
    edges $e$ \shcomm{starting at} $v$ and \shcomm{ending at} $v$, respectively.
    Hereafter, an empty sum (respectively, product) is defined to be $0$ (respectively, $1$).
    \item[(V)]\textit{The Kirchhoff voltage law.} For each
    simple cycle $\gamma$, we have (see Fig.~\ref{fig:electrical-network}, right)
    $$
    \sum_{e\textrm{ oriented along }\gamma}\frac{I(e)}{c(e)}-\sum_{e\textrm{ oriented opposite to }\gamma}\frac{I(e)}{c(e)}=\sum_{e\textrm{ \shcomm{oriented along} }\gamma}\mathcal{E}(e)-\sum_{e\textrm{ \shcomm{oriented opposite to} }\gamma}\mathcal{E}(e),
    $$
    where the sums are over all the %oriented
    edges $e$ contained in $\gamma$ that are oriented along $\gamma$ and opposite to $\gamma$, respectively.
\end{itemize}
\end{definition}

The function $I$ satisfying (I) and (V) exists and is unique; see, e.g., Theorem~\ref{th-existence-uniqueness-synchronized} below.

\begin{remark} \label{rem-total-current}
The left side of the equality in
law~(I) is called the \emph{total (outgoing) current at the vertex $v$} and is denoted by $I_v$. Thus, law~(I) means that $I_v=0$ for every vertex~$v$.

    The value ${I(e)}/{c(e)}-\mathcal{E}(e)$ equals the \emph{voltage drop}
    along the edge $e$ (this is \emph{Ohm's law}), and law~(V) means that there are well-defined \emph{voltages} at the vertices (see Sec.~\ref{sec-voltages}). The source voltage $\mathcal{E}(e)$ is alternatively called \textit{electromotive force}.
\end{remark}

\begin{notation*}
   For a subgraph $G$ of the circuit, $c(G):=\prod_{e\subset G}c(e)$ is the product over edges~$e$.
\end{notation*}

\begin{example} \label{ex-triangle} For the circuit in Fig.~\ref{fig:electrical-network}, we easily find
\[
\begin{aligned}
I(e_1)=I(e_2)=I(e_3)
=
\frac{
c(e_1)c(e_2)c(e_3)
\left(\mathcal{E}(e_1)+\mathcal{E}(e_2)+\mathcal{E}(e_3)\right)
}{
c(e_1)c(e_2)+c(e_2)c(e_3)+c(e_3)c(e_1)
}=
\frac{
c\left(
\begin{tikzpicture}[scale=0.07]
\draw[color=blue]
(0,0) to[short, *-*] (4,0)
(0,0) to[short, *-*] (2,3)
(2,3) to[short, *-*] (4,0);
\end{tikzpicture}
\right)
\left(\mathcal{E}(e_1)+\mathcal{E}(e_2)+\mathcal{E}(e_3)\right)
}{
c\left(
\begin{tikzpicture}[scale=0.07]
\draw[color=blue]
(0,0) to[short, *-*] (4,0)
(0,0) to[short, *-*] (2,3);
\end{tikzpicture}
\right)
+
c\left(
\begin{tikzpicture}[scale=0.07]
\draw[color=blue]
(0,0) to[short, *-*] (2,3)
(2,3) to[short, *-*] (4,0);
\end{tikzpicture}
\right)
+
c\left(
\begin{tikzpicture}[scale=0.07]
\draw[color=blue]
(0,0) to[short, *-*] (4,0)
(2,3) to[short, *-*] (4,0);
\end{tikzpicture}
\right)
}.
\end{aligned}
\]
\end{example}

We are interested in currents as functions of source voltages for fixed edge conductances. The resulting linear map
$$L\colon\mathbb{R}^m\to \mathbb{R}^m, \qquad (\mathcal{E}(e_1),\dots,\mathcal{E}(e_m))\mapsto (I(e_1),\dots,I(e_m)),$$
where $e_1,\dots,e_m$ are all edges of the circuit,
is called the \emph{response} of the circuit.
For instance, in Example~\ref{ex-triangle}, the response map has %matrix is
the $3\times 3$ matrix with all the entries $\tfrac{c(e_1)c(e_2)c(e_3)}{c(e_1)c(e_2)+c(e_2)c(e_3)+c(e_3)c(e_1)}$.

In what follows, we distinguish between circuits and networks: informally, a network is a circuit without specified source voltages. So, it is more accurate to speak of a network response rather than a circuit response, because the map $L$ itself does not depend on the source voltages.

\subsection{Source-synchronized networks}
\label{ssec-source-synchronized networks}

Now we slightly generalize Kirchhoff's definition by \emph{synchronizing} some of the sources, that is, by enforcing them to have the same voltage, as in Fig.~\ref{fig:magnet}, and also by setting some of the source voltages to zero.

On the one hand, Kirchhoff's electrical circuits fit into this new setup as a particular case where each synchronized source contains a single edge. On the other hand, source-synchronized circuits can be viewed as a special case of Kirchhoff's circuits when some of the source voltages coincide. However, we refrain from the latter point of view to emphasize %the additional structure resulting from
synchronization.

Note that the edges with vanishing source voltages no longer have any natural direction. However, it is still convenient to direct them. The choice of direction is arbitrary and does not affect any quantities of interest %, such as the network response
(up to sign in some cases).

\begin{definition}
A \textit{source-synchronized network} is a finite connected directed graph with a %positive
number $c(e)>0$ (\emph{conductance}) assigned to each %oriented
edge $e$ and with \(b\) disjoint sets \(\omega^1, \dots, \omega^b\) of %oriented
edges (\textit{synchronized sources}). % whose union does not contain an edge with both orientations.
A \textit{source-synchronized circuit} is a source-synchronized network with $b$ real numbers $\mathcal{E}_{1},\dots,\mathcal{E}_{b}$
(\emph{synchronized source voltages}) assigned to the %synchronized
sources $\omega^1, \dots, \omega^b$.

The \emph{source voltage} is the
real-valued function $\mathcal{E}$ on the set of all %oriented
edges $e$ given by %the formula
\begin{equation*} %\label{eq-E-synchronized}
\mathcal{E}(e):=
\begin{cases}
    \mathcal{E}_j, &\text{if }e\in\omega^j\text{ for some }j,\\
    %-\mathcal{E}_k, &\text{if }\overleftarrow e\in\omega_k\text{ for some }k,\\
    0, &\text{otherwise}.
\end{cases}
\end{equation*}

The \emph{current} $I$ is defined by the same Kirchhoff laws (I) and (V) as in Definition~\ref{def-electrical}.

The \emph{current through the synchronized source $\omega^j$} is
$I_{j} := I_{\omega^j}:= \sum_{e\in \omega^j} I(e).$

The \emph{response} of the source-synchronized network is the linear map $$L\colon \mathbb{R}^b \to \mathbb{R}^b, \qquad (\mathcal{E}_{1},\dots,\mathcal{E}_{b}) \mapsto (I_1, \dots, I_b).$$
Denote by $L$ the matrix of this map as well, and by $L_i^j$ the entry in $i$-th row and $j$-th column. For %a matrix \(M\), and
any \(P,Q\subset\{1,\dots,b\}\), denote by \(L_P^Q\) the submatrix of \(L\) formed by taking the rows indexed by \(P\) and the columns indexed by \(Q\).
We omit the index $Q$ if $Q=\{1,\dots,b\}$.
\end{definition}

\begin{theorem}\label{th-existence-uniqueness-synchronized}
For any source-synchronized circuit, there exists a unique real %-valued
function $I$ on the set of %oriented
edges satisfying the Kirchhoff laws (I) and (V).
\end{theorem}

\begin{proof} \emph{Uniqueness.} By linearity, it suffices to prove that $\mathcal{E}_1=\dots=\mathcal{E}_b=0$ implies $I=0$. Assume the converse. For each edge $e$ with $I(e)<0$, if any, revert the direction of $e$ and the sign of $I(e)$; this does not affect %the Kirchhoff
laws (I) and (V), as $\mathcal{E}(e)=0$. Then there is an edge $e_1$ with $I(e_1)>0$. Start growing a path from $e_1$. By law (I), there is an edge $e_2$ starting at the endpoint of $e_1$ such that $I(e_2)>0$. Continue this path until we reach a vertex already visited. We get a simple cycle with a positive current through each edge, contradicting law (V). Thus, $I=0$.

\emph{Existence.} Let the network have $n$ vertices and $m$ edges. Fix a spanning tree $T$. For edge $e\notin T$, there is a unique (up to direction reversal) simple cycle in $T\cup e$ called a \emph{basis simple cycle}.
Note that it suffices to impose law (I) for any $n-1$ vertices $v$ and law (V) for $m-n+1$ basis simple cycles $\gamma$.
This gives a system of $m$ linear equations on $m$ unknown currents through all edges. By the uniqueness part, the system has a unique solution for $\mathcal{E}_1=\dots=\mathcal{E}_b=0$. Hence, it is nondegenerate and has a solution for any $\mathcal{E}_1,\dots,\mathcal{E}_b$.
\end{proof}

% \begin{figure}
% \begin{tikzpicture}[line cap=round,line join=round,>=triangle 45,x=2.0cm,y=2.0cm]
% \clip(-3.8,1.25) rectangle (-2.73,2.69);
% \draw [->] (-3.5,1.5) -- (-3.5,2);
% \draw [->] (-3,1.5) -- (-3,2);
% \draw [color=blue,->] (-3.5,1.5) -- (-3,1.5);
% \draw [color=blue,->] (-3.5,2) -- (-3,2);
% \draw [color=blue,->] (-3.25,1.5) -- (-3.25,2);
% %\draw (-3.5,1.5)-- (-3,2);
% \begin{scriptsize}
% \draw[color=black] (-3.59,1.74) node {$\mathcal{E}$};
% \draw[color=black] (-2.92,1.75) node {$\mathcal{E}$};
% \draw[color=blue] (-3.21,1.38) node {$e_1$};
% \draw[color=blue] (-3.25,2.06) node {$e_3$};
% \draw[color=blue] (-3.4,1.8) node {$e_2$};
% \draw[color=blue] (-3.35,1.6) node {$G$};
% \end{scriptsize}
% \end{tikzpicture}
%     \caption{An electrical circuit on the cylinder.}
%     \label{fig:electrical-network-cylinder}
% \end{figure}

\begin{example}
\label{ex:cylinder}
Let us find the currents in the synchronized circuit %shown
in Fig.~\ref{fig:magnet}, middle. %Introduce notation as in Figure~\ref{fig:magnet}.

The synchronized source is $\omega^1=\{e_1,e_3\}$ with the source voltage $\mathcal E$. We direct $e_2$ arbitrarily.

By Kirchhoff's current law, %at one of the two vertices,
we get %vertex $1$:
%\[
$I(e_1) - I(e_1) + I(e_2) = 0$ % \quad \Rightarrow \quad
so that $I(e_2) = 0.$
%\]

By Kirchhoff's voltage law along the loops, we get
%\[
$I(e_1) = c(e_1)\,\mathcal{E}$ and $ %\qquad
I(e_3) = c(e_3)\,\mathcal{E}.$
%\]
%
%Thus, the currents in the network are
%\[
%I(e_1) = c(e_1)\,\mathcal{E}, \qquad I(e_2) = 0, \qquad I(e_3) =  c(e_3)\,%\mathcal{E}.
%\]
%\mscommnew{Compute the response matrix as well}

So the response matrix is the $1\times 1$ matrix
%\[
$L = \begin{pmatrix}
c(e_1)+c(e_3)
\end{pmatrix}.$
%\]
\end{example}

\begin{example} \label{ex-electric-via-source-synchronized}
Consider a classical electrical network with the Dirichlet boundary condition. %In other words, a conductance is assigned to each edge, and
This means that a voltage $U_i$ is assigned to each of a subset of vertices called {\it boundary vertices} $v_i$, where $i = 1, \ldots, b$ (see \cite[Section~2]{PSS}). There are no sources on edges, i.e., $\mathcal E(e) = 0$ for every edge $e$. In this case, the current flows according to the usual Kirchhoff laws. This case, however, can be simulated in our setting as follows.

Introduce an additional vertex $u$, and join it with each boundary vertex $v_i$ by an edge $e_i$
%, and connect each of the boundary vertices $v_i$ to $u$ by edge $e_i$
%\shcomm{directed from $u$ to $v_i$}
of conductance $c(e_i)$, which we eventually tend to infinity.
% and connect each of the boundary vertices $v_i$ to $u$ by edge $e_i$ \shcomm{directed from $u$ to $v_i$} of conductance $c(e_i)$, which we eventually tend to infinity.
Set $\omega^i:=\{e_i\}$ and $\mathcal E_i:=U_i$. % be equal to the voltage originally assigned to vertex $v_i$.

Let us show that in the limit $c(e_i)\to +\infty$, the currents in the resulting source-synchronized circuit reproduce the ones in the original network. We construct an approximation $I(e)$ of the former and verify laws~(I) and~(V). For %each
an edge $e$ of the original network, set $I(e)$ to the original current. For a new edge $e_i$, set $I(e_i)$ to the total current at %leaving
$v_i$ in the original network.

Now, law~(I) holds at all non-boundary vertices %of the source-synchronized circuit
as it did for the original network; %Dirichlet boundary conditions;
it holds for boundary vertices due to our choice of $I(e_i)$; it holds for $u$ because the sum of total
currents %out of
at all boundary vertices %for initial Dirichlet boundary conditions
in the original network was $0$.

Law~(V) holds for any simple cycle not passing through $u$ as it did before. For a simple cycle $\gamma$ passing through $u$, let $v_i$ and $v_j$ be boundary vertices through which it passes. Let $\gamma$ be oriented along $e_i$.
By the properties of networks with the Dirichlet boundary condition, we get
$$\sum_{e \text{ oriented along } \gamma} \frac{I(e)}{c(e)}
-\sum_{e \text{ oriented opposite to } \gamma} \frac{I(e)}{c(e)}
=U_i - U_j+\frac{I(e_i)}{c(e_i)}-\frac{I(e_j)}{c(e_j)}
\to \mathcal E(e_i) - \mathcal E(e_j)$$
as $c(e_i),c(e_j)\to+\infty$.
Thus, %the voltage
law~(V) holds in the limit as all $c(e_i)\to +\infty$.

By simple linear algebra, the currents in the source-synchronized circuit are rational functions in edge conductances and source voltages and have a finite limit as all $c(e_i)\to +\infty$ (see e.g. Proposition~\ref{prop-current} below). Then they tend to $I(e)$, %as all $c(e_i)\to +\infty$,
and thus
reproduce the original currents.
\end{example}

\begin{remark}
A multiport or superport circuit (see \cite[Sections~3--4]{PSS} for the definitions and the notation) can also be realized as a limiting case of a source-synchronized circuit as follows. For every non-root boundary vertex $j$, add an edge $e_j$ from $root(j)$ to $j$, a synchronized source $\omega^j=\{e_j\}$ with the synchronized source voltage $\mathcal E_j=\Delta U_{j,root(j)}$, and let the conductance $c(e_j)$ tend to $+\infty$. In this limit, the resulting source-synchronized circuit reproduces %exactly the voltages and
the currents in the original %superport
circuit. T.~Newton (private communication) has recently found an elegant formula for all minors of the response matrix of a superport circuit.
\end{remark}

\subsection{Networks on surfaces}
\label{ssec-networks on surfaces}

We now turn to another type of boundary condition that arises naturally in networks on surfaces.

It is instructive to revisit the Kirchhoff voltage law~(V) and the expression on the right side. We observe that the current $I$ depends only on the values of this expression for all simple cycles $\gamma$, rather than the source voltages of the individual edges. For instance, if all edges containing a particular vertex are directed towards this vertex, %and are not contained in any synchronized sources,
then adding the same value to the source voltages at those edges %a new synchronized source consisting of these edges
does not affect the current $I$. This operation is known as adding the \emph{coboundary} of the vertex. In other words, %adding a coboundary
the current depends only on the \emph{cohomology class} of source voltages, that is, on the source voltages up to adding coboundaries.

For a circuit on a surface, we set the expression on the right side of the Kirchhoff voltage law~(V) to zero for all trivial cycles $\gamma$ on the surface, that is, for all contractible %null-homologous
simple cycles~$\gamma$. Then the current $I$ depends only on the values for a few basis simple cycles.

It is more geometric to use basis cycles on the surface that are \emph{transversal} to the edges rather than the cycles in the network itself.
Recall that the \emph{crossing number} $\alpha\cdot\beta$ of two transversal regular oriented curves $\alpha$ and $\beta$ on a smooth oriented surface is the sum of signs of their intersection points. The \emph{sign} of an intersection point is $+1$ if the basis formed by the tangent vectors to $\alpha$ and $\beta$ at the point is positively oriented, and $-1$, otherwise. Geometrically,
if at a crossing $\beta$ approaches $\alpha$ from the right, the sign is $+1$, if from the left, it is $-1$; see Fig.~\ref{fig:crossing}.

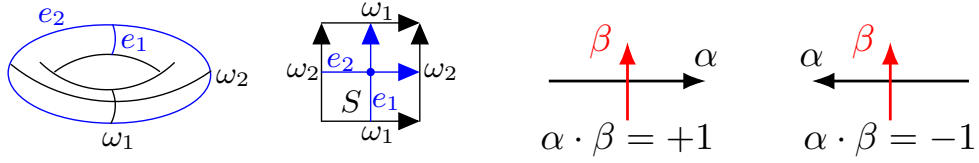
\begin{figure}
\scalebox{1.3}{
    \centering
    \quad\begin{tikzpicture}[line cap=round,line join=round,>=triangle 45,x=2.0cm,y=2.0cm]
\clip(-3.59,1.26) rectangle (-2.3,2.7);
\draw [color=blue, rotate around={-0.77:(-3.01,1.75)}] (-3.01,1.75) ellipse (1.03cm and 0.5cm);
\draw [shift={(-3.02,2.16)}] plot[domain=3.94:5.48,variable=\t]({1*0.49*cos(\t r)+0*0.49*sin(\t r)},{0*0.49*cos(\t r)+1*0.49*sin(\t r)});
\draw [shift={(-3.02,1.4)}] plot[domain=0.94:2.21,variable=\t]({1*0.45*cos(\t r)+0*0.45*sin(\t r)},{0*0.45*cos(\t r)+1*0.45*sin(\t r)});
\draw [shift={(-3.18,1.59)}] plot[domain=-0.47:0.39,variable=\t]({1*0.2*cos(\t r)+0*0.2*sin(\t r)},{0*0.2*cos(\t r)+1*0.2*sin(\t r)});
\draw [color=blue, shift={(-3.18,1.93)}] plot[domain=-0.42:0.36,variable=\t]({1*0.2*cos(\t r)+0*0.2*sin(\t r)},{0*0.2*cos(\t r)+1*0.2*sin(\t r)});
\draw [shift={(-2.98,2.45)}] plot[domain=4.03:5.33,variable=\t]({1*0.84*cos(\t r)+0*0.84*sin(\t r)},{0*0.84*cos(\t r)+1*0.84*sin(\t r)});
\begin{scriptsize}
\draw[color=black] (-2.96, 1.41) node {$\omega_1$};
\draw[color=black] (-2.38, 1.73) node {$\omega_2$};
\draw[color=blue] (-2.88, 1.91) node {$e_1$};
\draw[color=blue] (-3.3, 2.04) node {$e_2$};
\end{scriptsize}
\end{tikzpicture}
\begin{tikzpicture}[line cap=round,line join=round,>=triangle 45,x=2.0cm,y=2.0cm]
\clip(-3.8,1.25) rectangle (-2.73,2.69);
\draw [->] (-3.5,1.5) -- (-3.5,2);
\draw [->] (-3,1.5) -- (-3,2);
\draw [->] (-3.5,1.5) -- (-3,1.5);
\draw [->] (-3.5,2) -- (-3,2);
\draw [color=blue,->] (-3.25,1.5) -- (-3.25,2);
\draw [color=blue,->] (-3.5,1.75) -- (-3,1.75);
\fill [color=blue] (-3.25,1.75) circle (1.2pt);
%\draw (-3.5,1.5)-- (-3,2);
\begin{scriptsize}
\draw[color=black] (-3.59,1.75) node {$\omega_2$};
\draw[color=black] (-2.89,1.75) node {$\omega_2$};
\draw[color=black] (-3.21,1.43) node {$\omega_1$};
\draw[color=black] (-3.21,2.06) node {$\omega_1$};
\draw[color=blue] (-3.4,1.8) node {$e_2$};
\draw[color=blue] (-3.17,1.6) node {$e_1$};
\draw[color=black] (-3.35,1.6) node {$S$};
\end{scriptsize}
\end{tikzpicture}
    \quad
    \begin{tikzpicture}[>=Stealth, thick,scale=0.4]
  \draw[-{Latex}] (-2,0) -- (2,0) node[above] {\small $\alpha$};    % по оси x
  \draw[-{Latex}, red] (0,-1) -- (0,1) node[left] {\small $\beta$};     % по оси y
  %\fill (-2,0) circle (2pt) node[below left] {$k$};
  %\fill (2,0) circle (2pt) node[below right] {$l$};
  \node at (0,-1.5) {\small $\alpha\cdot \beta=+1$};
\end{tikzpicture}
\quad
\begin{tikzpicture}[>=Stealth, thick,scale=0.4]
  \draw[-{Latex}] (2,0) -- (-2,0) node[above] {\small $\alpha$};    % по оси x
  \draw[-{Latex}, red] (0,-1) -- (0,1) node[left] {\small $\beta$};     % по оси y
  %\fill (-2,0) circle (2pt) node[below left] {$k$};
  %\fill (2,0) circle (2pt) node[below right] {$l$};
  \node at (0,-1.5) {\small $\alpha\cdot \beta=-1$};
\end{tikzpicture}}
    \caption{A circuit on a surface (left) and the sign of a crossing (right). The opposite sides of the square are identified. See Example~\ref{ex-torus-network}.} %\mscomm{$\alpha_1,\beta_1\to\omega_1,\omega_2$. Show the vertex!}}
    \label{fig:crossing}
\end{figure}

\begin{definition} \label{def-surface-circuit}
(See Fig.~\ref{fig:crossing}, left.)
Let $S$ be a connected closed smooth oriented surface equipped with simple closed regular oriented curves $\omega_1,\dots,\omega_{2g}$
%$\alpha_1,\dots,\alpha_g$ and $\beta_1,\dots,\beta_g$
forming a basis of cycles (i.e., a basis of the homology group $H_1(S;\mathbb{Z})$ viewed as a free $\mathbb{Z}$-module).
A \emph{circuit on the surface} is a finite %weakly
  connected embedded directed graph on %embedded into
  the surface~$S$, whose edges are regular curves that intersect the basis cycles %$\alpha_1,\dots,\alpha_g$ and $\beta_1,\dots,\beta_g$
  transversely, with a number $c(e)>0$ %(\redit{conductance})
  assigned to each %oriented
  edge $e$ and real numbers $\mathcal{E}_1,\dots,\mathcal{E}_{2g}$ (\emph{voltage drops across the cycles})
  assigned to the basis cycles in order.
  %$\alpha_1,\dots,\alpha_g$ and $\beta_1,\dots,\beta_g$.

  The \emph{current} %in the circuit
is the real-valued
function $I$ on the set of %oriented
edges given by (I) and
\begin{itemize}%[<+->]
    \item[(V$^*$)] \emph{The Kirchhoff voltage law:} for each
    simple cycle $\gamma$ in the graph, %we have
    $$
    %\hspace{-0.4cm}
    %\sum_e\langle \gamma,e\rangle \frac{I(e)}{c(e)}=
    \sum_{e\textrm{ oriented along }\gamma}\frac{I(e)}{c(e)}-\sum_{e\textrm{ oriented opposite to }\gamma}\frac{I(e)}{c(e)}=
    \sum_j(\gamma\cdot\omega_j)\mathcal{E}_{j}.
    %\sum_k(\gamma\cdot\alpha_k)\mathcal{E}_{k} +\sum_k(\gamma\cdot\beta_k)\mathcal{E}_{g+k}.
    $$
    %    $\sum_{\vec e}R_eI(\vec e)=\sum_{\vec e}\mathcal{E}(\vec e)$,
    %where the sums are over all %the oriented
    %edges $e$ contained in $\gamma$.
\end{itemize}

The \emph{currents across the basis cycles} are $I_j:=\sum_e(e\cdot \omega_j) I(e)$.
    %$I_k:=\sum_e(e\cdot \alpha_k) I(e)$ and $I_{k+g}:=\sum_e(e\cdot \beta_k) I(e)$.

The \emph{response} is the map $L\colon (\mathcal{E}_1,\dots,\mathcal{E}_{2g})\mapsto (I_1,\dots,I_{2g})$.
\end{definition}

The %function $I$ satisfying (I) and (V$^*$)
current $I$ exists and is unique; the proof is the same as for
Theorem~\ref{th-existence-uniqueness-synchronized}.

% \begin{example} For the network in Figure~\ref{fig:crossing}, we find \mscomm{more details!}
% \begin{align*}
%         e_1\cdot \alpha_1&=-1 & e_1\cdot \beta_1&=0 & I(e_1)&=-c(e_1)\mathcal{E}_1 & I_1&=c(e_1)\mathcal{E}_1 \\
%         e_2\cdot \alpha_1&=0 & e_2\cdot \beta_1&=1 & I(e_2)&=c(e_2)\mathcal{E}_2 & I_2&=c(e_2)\mathcal{E}_2 ;
% \end{align*}
% $$L=\begin{pmatrix}
%         c(e_1) & 0     \\
%         0      & c(e_2)
%     \end{pmatrix}.$$
% \end{example}

\begin{example}\label{ex-torus-network}
Consider the network in Fig.~\ref{fig:crossing} with edges $e_1,e_2$ of conductances $c(e_1),c(e_2)$ crossing the basis cycles $\omega_1, \omega_2$ that follow the cuts, so that
\[
e_1 \cdot \omega_1 = -1, \quad e_1 \cdot \omega_2 = 0, \qquad
e_2 \cdot \omega_1 = 0, \quad e_2 \cdot \omega_2 = 1,
\]
with voltage drops $\mathcal{E}_1,\mathcal{E}_2$ across cycles $\omega_1,\omega_2$.

The edges $e_1$ and $e_2$ are the cycles in the network; by the voltage law (V$^*$) for them,
%\[
%\sum_e ( \gamma \cdot e ) \frac{I(e)}{c(e)} = \sum_k (\gamma \cdot \omega_k) \mathcal{E}_k,
%\]
we get
\[
I(e_1) = -c(e_1)\mathcal{E}_1, \qquad I(e_2) = c(e_2)\mathcal{E}_2.
\]
The currents across the cycles are
\[
I_1 = \sum_e (e \cdot \omega_1) I(e) = c(e_1)\mathcal{E}_1, \qquad
I_2 = \sum_e (e \cdot \omega_2) I(e) = c(e_2)\mathcal{E}_2.
\]
Then the response matrix is
\[
L = \begin{pmatrix}
c(e_1) & 0 \\
0 & c(e_2)
\end{pmatrix}.
\]

If we pick a different basis $\omega_1' = \omega_1 + 2\omega_2$, $\omega_2' = \omega_2$ of the homology group $H_1(S;\mathbb{Z})$, then
\[
e_1 \cdot \omega_1' = -1, \quad e_1 \cdot \omega_2' = 0, \qquad
e_2 \cdot \omega_1' = 2, \quad e_2 \cdot \omega_2' = 1.
\]
%and we correspondingly set $\mathcal{E}_1' = \mathcal{E}_1$, $\mathcal{E}_2' = \mathcal{E}_2 - 2\mathcal{E}_1$ to obtain the same currents $I(e_1) = -c(e_1)\mathcal{E}_1$ and $I(e_2) = c(e_2)\mathcal{E}_2$.
In terms of voltage drops $\mathcal{E}_1'$ and $\mathcal{E}_2'$ across $\omega_1'$ and $\omega_2'$, the currents are $I(e_1) = -c(e_1)\mathcal{E}_1'$ and $I(e_2) = c(e_2)(2\mathcal{E}_1'+\mathcal{E}_2')$.
The response matrix becomes
\[
L' = \begin{pmatrix}
c(e_1) + 4c(e_2) & 2c(e_2) \\
2c(e_2) & c(e_2)
\end{pmatrix}
= A L A^\mathrm{T}, \qquad A = \begin{pmatrix} 1 & 2 \\ 0 & 1\end{pmatrix},
\]
where $A$ is the matrix of the change of basis. %; in particular, $\det L' = \det L$.
The entries of $L'$ involve integer coefficients greater than $1$; a less trivial example of this phenomenon was provided by W.Y.~Lam et al.~\cite[Sec.~9]{Lam-25}.
\end{example}

In general, if we switch from one choice of homology basis $\omega=(\omega_1,\dots,\omega_{2g})^\mathrm{T}\subset H_{1}(S;\mathbb{Z})$ to a different one $\omega'=(\omega'_1,\dots,\omega'_{2g})^\mathrm{T}$
and transform the vector of voltage drops $\mathcal{E}=(\mathcal{E}_1,\dots,\mathcal{E}_{2g})^\mathrm{T}$ suitably, we get the same currents. Indeed, let
the linear change be described by a matrix~$A$, i.e., $A \omega = \omega'$.
If $\mathcal{E}' = (A^\mathrm{T})^{-1} \mathcal{E}$, then
%since in this case
$(\gamma \cdot \omega')^\mathrm{T} \mathcal{E}' = (\gamma \cdot \omega)^\mathrm{T} A^\mathrm{T} (A^\mathrm{T})^{-1} \mathcal{E} = (\gamma \cdot \omega)^\mathrm{T}  \mathcal{E}$ for any simple cycle $\gamma$, and the contribution on the right side of (V$^*$) is the same.

Thus, the same system can be defined using any homology basis. It is often natural to choose curves $\omega_1,\dots,\omega_{2g}$ forming a symplectic basis of $H_{1}(S;\mathbb{Z})$ so that cutting the surface along them results in a disk; such choices are always possible. In the latter case, the circuit is viewed as a \emph{discrete Riemann surface}. The complex numbers $A_k:=\mathcal{E}_{k+g}-iI_{k}$ and $B_k:=\mathcal{E}_{g}+iI_{k+g}$, where $k=1,\dots,g$, are then called \emph{A-} and \emph{B-periods} respectively. A-periods determine B-periods. The \emph{(discrete) period matrix} %$\Pi$
is a $g\times g$ matrix whose $i$-th column is the B-period of the circuit with A-periods $A_j=\delta^i_j$. See \cite{Lam-25} for a comprehensive theory.
%\mscomm{The choices of basis. The case when cutting results in a disk, zippers.}

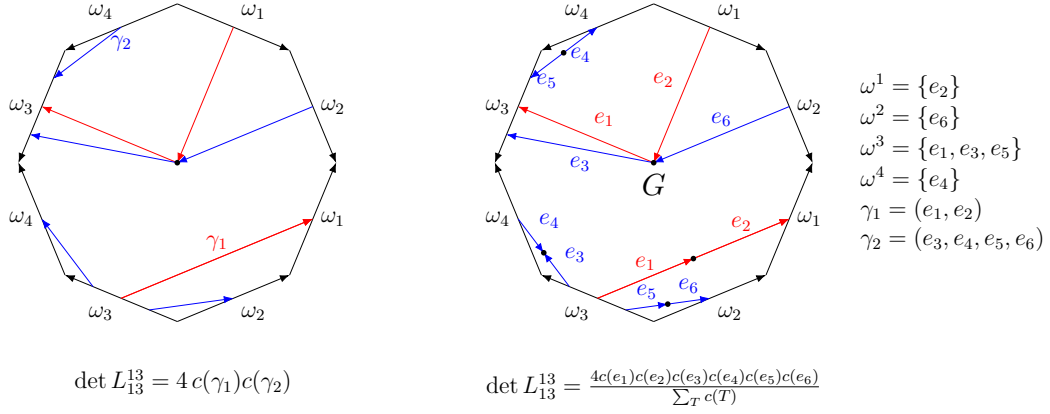
\begin{figure}
\scalebox{0.75}{\begin{tikzpicture}[scale=1.4,>=stealth,->]
\begin{scope}
  \foreach \i in {0,...,7} {
    \coordinate (P\i) at (90-45*\i:2);
  }
  % Подписи рёбер
  \path (P0) -- (P1) node[midway,above right] {$\omega_1$};
  \path (P1) -- (P2) node[midway, right] {$\omega_2$};
  \path (P2) -- (P3) node[midway, right] {$\omega_1$};
  \path (P3) -- (P4) node[midway,below right] {$\omega_2$};
  \path (P4) -- (P5) node[midway,below left] {$\omega_3$};
  \path (P5) -- (P6) node[midway, left] {$\omega_4$};
  \path (P6) -- (P7) node[midway, left] {$\omega_3$};
  \path (P7) -- (P0) node[midway,above left] {$\omega_4$};
  % Стрелки (ориентация)
  \draw[-{Latex}] (P0) -- (P1);
  \draw[-{Latex}] (P1) -- (P2);
  \draw[-{Latex}] (P3) -- (P2);
  \draw[-{Latex}] (P4) -- (P3);
  \draw[-{Latex}] (P4) -- (P5);
  \draw[-{Latex}] (P5) -- (P6);
  \draw[-{Latex}] (P7) -- (P6);
  \draw[-{Latex}] (P0) -- (P7);

  \draw[{Latex}-,red] (0,0) -- ($ (P0)!0.5!(P1) $);

  \coordinate (alpha1opp) at ($ (P2)!0.5!(P3) $);
  \coordinate (gamma1) at ($ (P4)!0.5!(P5) $);
  \coordinate (delta1) at ($ (P7)!0.5!(P0) $);
  \coordinate (delta1opp) at ($ (P5)!0.5!(P6) $);
  \coordinate (gamma1opp) at ($ (P6)!0.5!(P7) $);
  \coordinate (beta1opp) at ($ (P3)!0.5!(P4) $);
    \coordinate (gamma1_1) at ($ (P4)!0.25!(P5) $);
    \coordinate (gamma1_2) at ($ (P4)!0.75!(P5) $);
    \coordinate (beta1) at ($ (P1)!0.5!(P2) $);

    \draw[{Latex}-,red] (alpha1opp) -- node[midway,above]{$\gamma_1$} (gamma1)
    coordinate[midway] (midpoint);
\draw[{Latex}-,red] (alpha1opp) -- (gamma1);

\draw[{Latex}-,red] (gamma1opp) -- (0,0);

\draw[{Latex}-,blue] (0,0) -- (beta1);

\coordinate (mid1) at ($(beta1opp)!0.5!(gamma1_1)$);
\draw[{Latex}-,blue] (beta1opp) -- (gamma1_1);

\coordinate (gamma1opp_1) at ($ (P6)!0.75!(P7) $);
\coordinate (gamma1opp_2) at ($ (P6)!0.25!(P7) $);

\draw[{Latex}-,blue] (gamma1opp_1) -- (delta1) node[below]{$\gamma_2$};

\draw[{Latex}-,blue] (delta1opp) -- (gamma1_2);

\draw[{Latex}-,blue] (gamma1opp_2) -- (0,0);
\fill (0,0) circle (1pt);
\node[anchor=north west, align=left] at (-1.4,-2.5) {%(-2.1,-2.8) {%\scriptsize
$
\det L_{13}^{13}=4\,c(\gamma_1)c(\gamma_2)
%L=
%\begin{pmatrix}
%c(\gamma_1) & 0 & c(\gamma_1) & 0 \\
%0 & c(\gamma_2) & 2c(\gamma_2) & c(\gamma_2) \\
%c(\gamma_1) & 2c(\gamma_2) & c(\gamma_1)+4c(\gamma_2) & 2c(\gamma_2) \\
%0 & c(\gamma_2) & 2c(\gamma_2) & c(\gamma_2)
%\end{pmatrix}
$
};
\end{scope}
\begin{scope}[xshift=6cm]
  % Восьмиугольник на окружности радиуса 2
  \foreach \i in {0,...,7} {
    \coordinate (P\i) at (90-45*\i:2);
  }
  % Подписи рёбер
  \path (P0) -- (P1) node[midway,above right] {$\omega_1$};
  \path (P1) -- (P2) node[midway, right] {$\omega_2$};
  \path (P2) -- (P3) node[midway, right] {$\omega_1$};
  \path (P3) -- (P4) node[midway,below right] {$\omega_2$};
  \path (P4) -- (P5) node[midway,below left] {$\omega_3$};
  \path (P5) -- (P6) node[midway, left] {$\omega_4$};
  \path (P6) -- (P7) node[midway, left] {$\omega_3$};
  \path (P7) -- (P0) node[midway,above left] {$\omega_4$};
  % Стрелки (ориентация)
  \draw[-{Latex}] (P0) -- (P1);
  \draw[-{Latex}] (P1) -- (P2);
  \draw[-{Latex}] (P3) -- (P2);
  \draw[-{Latex}] (P4) -- (P3);
  \draw[-{Latex}] (P4) -- (P5);
  \draw[-{Latex}] (P5) -- (P6);
  \draw[-{Latex}] (P7) -- (P6);
  \draw[-{Latex}] (P0) -- (P7);

\draw[{Latex}-,red] (0,0) -- ($ (P0)!0.5!(P1) $) node[midway,above left] {$e_2$};

  \coordinate (alpha1opp) at ($ (P2)!0.5!(P3) $);
  \coordinate (gamma1) at ($ (P4)!0.5!(P5) $);
  \coordinate (delta1) at ($ (P7)!0.5!(P0) $);
  \coordinate (delta1opp) at ($ (P5)!0.5!(P6) $);
  \coordinate (gamma1opp) at ($ (P6)!0.5!(P7) $);
  \coordinate (beta1opp) at ($ (P3)!0.5!(P4) $);
    \coordinate (gamma1_1) at ($ (P4)!0.25!(P5) $);
    \coordinate (gamma1_2) at ($ (P4)!0.75!(P5) $);
    \coordinate (beta1) at ($ (P1)!0.5!(P2) $);

\draw[{Latex}-,red] (alpha1opp) -- node[midway,above] {} (gamma1)
    coordinate[midway] (midpoint);
\draw[{Latex}-,red] (alpha1opp) -- (midpoint) node[midway,above] {$e_2$};
\draw[{Latex}-,red] (midpoint) -- (gamma1) node[midway,above] {$e_1$};

\draw[{Latex}-,red] (gamma1opp) -- (0,0) node[midway,above right] {$e_1$};

\draw[{Latex}-,blue] (0,0) -- (beta1) node[midway,above] {$e_6$};

\coordinate (mid1) at ($(beta1opp)!0.5!(gamma1_1)$);
\draw[{Latex}-,blue] (beta1opp) -- (mid1) node[midway,above] {$e_6$};
\draw[{Latex}-,blue] (mid1) -- (gamma1_1) node[midway,above] {$e_5$};

\fill (mid1) circle (1pt); % node[midway,below] {\small $G$};

\node[midway,below] at (mid1) {\small $G$};

\coordinate (gamma1opp_1) at ($ (P6)!0.75!(P7) $);
\coordinate (gamma1opp_2) at ($ (P6)!0.25!(P7) $);

\coordinate (mid2) at ($(gamma1opp_1)!0.5!(delta1)$);
\draw[{Latex}-,blue] (gamma1opp_1) -- (mid2) node[midway,below] {$e_5$};
\draw[-{Latex},blue] (mid2) -- (delta1) node[midway,below] {$e_4$};
\fill (mid2) circle (1pt);

\coordinate (mid) at ($(delta1opp)!0.5!(gamma1_2)$);
\draw[-{Latex},blue] (delta1opp) -- (mid) node[midway,above right] {$e_4$};
\draw[{Latex}-,blue] (mid) -- (gamma1_2) node[midway,above right] {$e_3$};
\fill (mid) circle (1pt);

\draw[{Latex}-,blue] (gamma1opp_2) -- (0,0) node[midway,below] {$e_3$};
\fill (0,0) circle (1pt);
\fill (midpoint) circle (1pt);
\node[anchor=west] at (2.5,0) {
  \shortstack[l]{
    $\omega^1 = \{e_2\}$\\
    $\omega^2 = \{e_6\}$\\
    $\omega^3 = \{e_1, e_3, e_5\}$\\
    $\omega^4 = \{e_4\}$\\
    $\gamma_1 = (e_1, e_2)$\\
    $\gamma_2 = (e_3, e_4, e_5, e_6)$
  }
};
\node[anchor=north] at (0,-2.5) { %(0,-2.8) {%\scriptsize
$\det L_{13}^{13}=\frac{4c(e_1)c(e_2)c(e_3)c(e_4)c(e_5)c(e_6)}{\sum_T c(T)}$
};
\end{scope}
\end{tikzpicture}}
    \caption{A network on a double torus (left) and its refinement (right). The sides of the octagon with the same labels are identified. Two original edges $\gamma_1$ and $\gamma_2$ (red and blue, forming a basis $\Gamma$ of simple cycles of the graph) are subdivided into
    several ones, each intersecting the octagon boundary once.
    %edges $e_1,e_2$ and $e_3,e_4,e_5,e_6$, respectively.
    The resulting network can be viewed as source-synchronized network $G$ with the synchronized sources $\omega^1,\dots,\omega^4$ listed to the right. %The graph $G$ has multiplicity greater than~$1$ in the minor $L^{13}_{13}$ of the response matrix.
    See Examples~\ref{ex-double-torus}, \ref{ex-sufaces-via-synchronized}, and~\ref{ex-network-on-double-torus}.}
    \label{fig:network-on-double-torus}
\end{figure}

\begin{example} \label{ex-double-torus}
    Consider the network in Fig.~\ref{fig:network-on-double-torus} to the left with edges $\gamma_1,\gamma_2$ of conductances $c(\gamma_1),c(\gamma_2)$ lying on a double torus with basis cycles $\omega_1,\dots,\omega_4$. Its response matrix is
    $$
L=
\begin{pmatrix}
c(\gamma_1) & 0 & c(\gamma_1) & 0 \\
0 & c(\gamma_2) & 2c(\gamma_2) & -c(\gamma_2) \\
c(\gamma_1) & 2c(\gamma_2) & c(\gamma_1)+4c(\gamma_2) & -2c(\gamma_2) \\
0 & -c(\gamma_2) & -2c(\gamma_2) & c(\gamma_2)
\end{pmatrix}.
$$
\end{example}

In this example, some edges intersect more than one of the basis cycles $\omega_j$ %or perhaps even
%and some
or intersect the same $\omega_j$ more than once. In such situations,
we introduce %let us consider
a refined circuit, where we break single edges into several edges connected in series, so that in the result no edge intersects more than one $\omega_j$, and no edge intersects the same $\omega_j$ more than once. See Fig.~\ref{fig:network-on-double-torus} to the right. For each edge $e$ of the original circuit that got broken into $e_1, \ldots, e_{\ell}$, choose conductances for the resulting edges %\emph{subdivision edges}
so that $\sum_{k=1}^{\ell}\frac{1}{c(e_k)} = \frac{1}{c(e)}$.
% \begin{equation}\label{eq-ex-sufaces-via-synchronized}
% \frac{1}{c(e)} = \frac{1}{c(e_1)} + \dotsc + \frac{1}{c(e_{\ell})}.
% \end{equation}
Orient each %subdivision
edge $e_k$ %of the circuit
so that $e_k \cdot \omega_j = +1$ for the unique $\omega_j$ that intersects $e_k$.
%where $\omega_j$ is the only basis cycle that intersects $e_k$.
The resulting circuit is called a {\it {refined circuit on the surface}}. The current in the refined circuit is the same as in the original one, in the sense that $I(e_k) = I(e)$ for $k=1,\dots,\ell$, because $\sum_{k=1}^{\ell}{I(e_k)}/{c(e_k)} = {I(e)}/{c(e)}$ and hence law (V$^*$) is preserved.

\begin{example} \label{ex-sufaces-via-synchronized} (See Fig.~\ref{fig:network-on-double-torus}.)
%Let us show that
Any circuit on a surface can be simulated by a source-synchronized circuit.
Indeed, we refine the circuit on the surface and create synchronized sources as follows: for each $\omega_j$, unify edges $e_k$ with $e_k \cdot \omega_j = +1$ into a synchronized source, with synchronized source voltage $\mathcal{E}_j$ equal to the voltage drop across the original cycle~$\omega_j$.
Clearly, then laws (V) and (V$^*$) are equivalent, laws~(I) are identical, hence the currents are the same.
\end{example}

We conclude this subsection with an informal physical remark. The ratio $I(e)/c(e)$ in the Kirchhoff voltage law~(V$^*$) has the meaning of \emph{voltage drop} $V(e)$ along the edge $e$.
In our setup, we have a linear relation $V(e)=I(e)/c(e)$ (known as \emph{Ohm's law}), but in general nonlinear networks, $V(e)$ can be any function $V(I(e))$ of $I(e)$,
%This function $V\colon \mathbb{R}\to\mathbb{R}$ is
known as \emph{voltage-current characteristic} of the edge $e$. We can also view the voltage-current characteristic as a functional $V(I)$ that takes the real-valued function $I$ on the set of edges to the real-valued function $V$ on the set of edges. In our linear setup, this functional just divides the value of $I(e)$ by $c(e)$. This linear functional can be viewed as the \emph{discrete Hodge star}. % on the discrete Riemann surface.
It motivates the definition in the next subsection.

\subsection{Cohomological networks}
\label{ssec-cohomological networks}

The definitions from Secs.~\ref{ssec-electrical networks} and~\ref{ssec-networks on surfaces} can be reformulated intrinsically in terms of chains, cochains, and cohomology, eliminating the need for embeddings, crossing numbers, and auxiliary zipper constructions.

This cohomological formulation is more conceptual and slightly more general. Let us recall some standard notions (for %more %comprehensive elementary introduction,
details, see \cite{Skopenkov+2023}, and for general theory, see \cite{fomenko2016homotopical}).

Our setup involves a higher-dimensional generalization of a graph, called a (finite) \emph{cell complex}.
It is a topological space defined by induction as follows.
A \emph{$0$-dimensional cell complex} $S_0$ is a finite discrete set. A \emph{$k$-dimensional cell complex} $S_k$ is obtained from a $(k-1)$-dimensional cell complex $S_{k-1}$ by attaching finitely many $k$-dimensional topological disks $D^k_i$ by continuous maps $\phi_i\colon \partial D_i^k\to S_{k-1}$. In the following definitions, to minimize technicalities, we assume that %the maps
$\phi_i$ are injective for $k=2$, and postpone the case of non-injective maps till Remark~\ref{rem-arbitrary-phi}.

The interiors $%\mathring{D}^k_i:=
{D}^k_i\setminus \partial{D}^k_i$ of the disks $D^k_i$
are called \emph{$k$-dimensional cells} or \emph{$k$-cells}. Each cell is an embedded open disk in the cell complex. An \emph{oriented cell complex} $S$ is a cell complex equipped with a fixed decomposition into cells and a fixed orientation of each cell. In what follows, we consider an oriented cell complex $S$, and a \emph{cell} always means an oriented cell.

A \emph{$k$-chain} is a real-valued function $I$ on the set of $k$-cells. (Traditionally, it is viewed as a formal linear combination $\sum_e I(e)e$ of $k$-cells $e$ with real coefficients $I(e)$. Sometimes, we switch to this traditional notation $\sum_e I(e)e$ for chains. In what follows, we focus on real coefficients; the construction for integer coefficients is completely analogous.) Denote by $C_k(S;\mathbb{R})$ the real vector space of $k$-chains.
Introduce the \emph{boundary operator}
$\partial\colon C_1(S;\mathbb{R}) \to C_{0}(S;\mathbb{R})$, $I\mapsto \partial I$, by
%by the formula
\[
[\partial I](v)=\sum_{e\textrm{ ending at }v}I(e)-\sum_{e\textrm{ starting at }v}I(e)
\]
for each vertex ($0$-cell) $v$, where the sums are over edges ($1$-cells) $e$.
A chain $I \in C_1(S;\mathbb{R})$ satisfying $\partial I = 0$ is called a \emph{cycle}.
Denote by $Z_1(S;\mathbb{R})=\ker \partial\subset C_1(S;\mathbb{R})$
the space of cycles.

In our setting, a current $I$ assigns a real number to each %oriented
edge, hence it is naturally a
$1$-chain. Kirchhoff's current law~(I) is exactly the condition $\partial I=0$, so a current is a $1$-cycle. %This justifies our non-traditional notation for chains.

Introduce the \emph{boundary operator} $\partial\colon C_2(S;\mathbb{R}) \to C_{1}(S;\mathbb{R})$, $I\mapsto \partial I$, by
$$
    [\partial I](e) =
    \sum_{f:\partial f\textrm{ oriented along }e}I(f)-\sum_{f:\partial f\textrm{ oriented opposite to }e}I(f)
$$
    for each $1$-cell $e$, where the sums are over all the $2$-cells $f$
    whose boundary $\partial f$ contains $e$ and is oriented along $e$ and opposite to $e$, respectively. (Here, we essentially rely on the injectivity of the %attaching
    maps $\phi_i$ to guarantee that
    %used the assumption that attaching maps $\phi_i$ are injective, so that
    the fixed orientation of $f$ determines an %well-defined
    orientation of $\partial f$.)

Introduce the \emph{first homology group} as the quotient
\[
H_1(S;\mathbb{R}) = Z_1(S;\mathbb{R}) / \partial C_{2}(S;\mathbb{R}).
\]

In our setting, the \emph{homology class} of the current $I$, that is, its equivalence class in $%[I]\in
H_1(S;\mathbb{R})$, is a measure of the current ``along nontrivial cycles in $S$''.

We also need a dual perspective. Traditionally, the space $C^k(S;\mathbb{R})$ of \emph{cochains} is viewed as the dual space of $C_k(S;\mathbb{R})$ of chains. However, the latter has a distinguished basis and can hence be identified with its dual. Therefore, we define a \emph{$k$-cochain} $U$ as a real-valued function on the set of $k$-cells again and denote by $C^k(S;\mathbb{R})$ the real linear space of such functions.
%(Traditionally, it is viewed as the dual space of $C_k(S;\mathbb{R})$, but in our setup, the latter comes with a distinguished basis and hence can be identified with the dual. We, however, keep different notations )

The \emph{coboundary operator}
$
\delta\colon C^0(S;\mathbb{R}) \to C^{1}(S;\mathbb{R})
$, $U\mapsto \delta U$,
is defined by
$$
[\delta U](e)=U(v_j)-U(v_i),
$$
where $e$ is an edge starting at a vertex $v_i$ and ending at a vertex $v_j$.
%, and $v_i$ and $v_j$ are its starting point and endpoint, respectively.

The \emph{coboundary operator}
$
\delta\colon C^1(S;\mathbb{R}) \to C^{2}(S;\mathbb{R})
$,
$\mathcal{E}\mapsto\delta \mathcal{E}$,
is defined by
$$
[\delta \mathcal{E}](f)=
    \sum_{e\textrm{ oriented along }\partial f}\mathcal{E}(e)-\sum_{e\textrm{ oriented opposite to }\partial f}\mathcal{E}(e),
$$
where $f$ is a $2$-cell, and we sum over $1$-cells $e$  oriented along and opposite to $\partial f$, respectively.
Here, an element of $C^1(S;\mathbb{R})$ is denoted by $\mathcal{E}$ for similarity to Kirchhoff's voltage law~(V).
%This is similar to the expression in Kirchhoff's voltage law~(V).

A cochain $\mathcal{E}\in C_1(S;\mathbb{R})$ satisfying $\delta \mathcal{E} = 0$ is called a \emph{cocycle}.
Denote by $Z^1(S;\mathbb{R})=\ker \delta\subset C^1(S;\mathbb{R})$
the space of cocycles. Introduce the \emph{first cohomology group} as the quotient
\[
H^1(S;\mathbb{R}) = Z^1(S;\mathbb{R}) / \delta C^{0}(S;\mathbb{R}).
\]

One can see that it is the dual space of the homology group $H_1(S;\mathbb{R})$; in particular, the two groups are non-canonically isomorphic (beware, this is no longer the case for integer coefficients).

Similarly to Sec.~\ref{ssec-networks on surfaces}, it is natural to set the expression on the right side of Kirchhoff's voltage law (V) to zero, if the simple cycle $\gamma$ is the boundary of a $2$-cell. Then the source voltages $\mathcal{E}(e)$ form a cocycle. We have already observed %above
that the current $I$ depends only on the cohomology class of this cocycle. Thus, we conceptualize the source voltage as a cohomology class.

\begin{definition} \label{def-homological}
    A \textit{cohomological (electrical) network} is a finite oriented %connected 1- or 2-dimensional
    cell complex $S$ with a number $c(e)>0$ (\textit{conductance}) assigned to each edge~$e$.
    A \textit{cohomological circuit} is a cohomological network with a distinguished cohomology class $\mathcal{E}\in H^1(S;\mathbb{R})$
    (\textit{source voltage}).

    The \emph{voltage-current characteristic} %of the cohomological network
    is the linear map $V\colon C_1(S;\mathbb{R})\to C^1(S;\mathbb{R})$
    given by the formula $[V(I)](e)=I(e)/c(e)$
    for all $I\in C_1(S;\mathbb{R})$ and all edges $e$ (or, in the traditional notation,
    $\left[V\left(\sum_{e'}I(e')e'\right)\right](e)=I(e)/c(e)$, where $e$ and $e'$ are edges and $\sum_{e'}I(e')e'$ is a $1$-chain).

    The \emph{current} in the circuit is the chain $I\in C_1(S;\mathbb{R})$ determined by the following axioms:
\begin{itemize}
    \item[(I)]\textit{The Kirchhoff current law.} $I$ is a cycle.
    \item[(V)]\textit{The Kirchhoff voltage law.} $V(I)$ is a cocycle and its cohomology class is $\mathcal{E}$.
\end{itemize}
The \emph{response} of the cohomological network is the map $L\colon H^1(S;\mathbb{R})\to H_1(S;\mathbb{R})$ that takes the source voltage $\mathcal{E}$ to the homology class of the current $I$.
\end{definition}

\begin{theorem} \label{th-existence-uniqueness-cohomological}
For any cohomological circuit, there is a unique %$1$-
chain~$I$ satisfying %conditions
(I) and~(V).
\end{theorem}

%This theorem has
The proof is literally the same %proof
as of Theorem~\ref{th-existence-uniqueness-synchronized}, because law (V) is equivalent to $V(I)$ and $\mathcal{E}$ having the same evaluation on any basis simple cycle in the union of 0- and 1-cells of $S$.

\begin{remark} Cells of dimension $>2$ do not affect the first (co)homology of $S$, hence the current $I$, and are therefore irrelevant. However, we still allow higher-dimensional cells; sometimes, this visualizes and simplifies the computation of the first (co)homology.
\end{remark}

\begin{example} %Let us show that
Any %Kirchhoff’s
electric circuit can be simulated by a %$1$-dimensional
cohomological circuit on the same directed graph.
Indeed, we can view the collection of source voltages $\mathcal{E}(e)$ as a %1-cochain (actually, a
1-cocycle, and introduce the source voltage ${\mathcal{E}}\in H^1(S;\mathbb{R})$ as the cohomology class of this cocycle.
By definition, the current in the %resulting
cohomological circuit is the same as in the %original
electric circuit.
%(more accurately, the currents $\widetilde{I}$ and $I$ in the circuits are related via $\widetilde{I}=\sum_e I(e)e$).
\end{example}

\begin{example} Consider a classical electrical network embedded into a disc with the Dirichlet boundary condition, that is, prescribed voltages $U_i$ at the vertices $v_i$ on the boundary of the disc (cf.~Example~\ref{ex-electric-via-source-synchronized}). This case can also be simulated by a cohomological circuit.

Indeed, introduce an additional vertex $u$ and join it with each boundary vertex $v_i$ by an edge $e_i$
%, and connect each of the boundary vertices $v_i$ to $u$ by edge $e_i$
\shcomm{(directed from $u$ to $v_i$)} of conductance $c(e_i)$, which we eventually tend to infinity.
Let~$S$ be the union of the additional edges with the cell decomposition of the disc formed by the original network. Introduce a 1-cocycle ${\mathcal{E}}\in C^1(S;\mathbb{R})$ by the formula
$\mathcal{E}(e)=U_i$ if $e=e_i$, and $\mathcal{E}(e)=0$ otherwise.
%$\langle{\mathcal{E}},e\rangle=U(v_i)$ if $e=e_i$, and $\langle{\mathcal{E}},e\rangle=0$ otherwise.
Let the source voltage ${\mathcal{E}}\in H^1(S;\mathbb{R})$ be the cohomology class of this cocycle.
Then the current in the %resulting
cohomological circuit tends to the one in the original network as $c(e_i)\to +\infty$.

Analogously, we simulate a network embedded into a \emph{ball} (and forming the 1-skeleton of a cell-decomposition of the ball) with the Dirichlet boundary condition. This example demonstrates how $3$-dimensional cells simplify the understanding of the cohomology of $S$.
\end{example}

\begin{example} \label{ex-surface-2-cohomological} Any circuit on a surface $S$ can be simulated by a %$1$-dimensional
cohomological circuit. Here, we add several edges $e_i$ to make each component of the complement to the circuit on $S$ homeomorphic to an open disc. %The conductance $c(e_i)$ is eventually tended to zero.
We set $\mathcal{E}:=\mathcal{E}_1\omega^1+\dots+\mathcal{E}_{2g}\omega^{2g}$, where $\omega^1,\dots,\omega^{2g}\in H^1(S;\mathbb{Z})$ are Poincar\`e dual to the classes of $\omega_1,\dots,\omega_{2g}$ in $H_1(S;\mathbb{Z})$. The current in the resulting cohomological circuit tends to the one in the original circuit as $c(e_i)\searrow 0$.
\end{example}

%We need the following notion. %construction.
Let $\omega^1,\dots,\omega^{b}\in C^1(S;\mathbb{Z})$ be 1-cocycles whose cohomology classes form a basis of $H^1(S;\mathbb{Z})$, which is always a free $\mathbb{Z}$-module.
Write $\mathcal E=\mathcal E_1\omega^1+\dots +\mathcal E_b\omega^b$ for some $\mathcal E_1,\dots,\mathcal E_b\in\mathbb{R}$.
A \emph{refined cohomological circuit} is obtained by breaking each edge $e$ into edges $e_1,\dots,e_\ell$ connected in series and modifying $\omega^1,\dots,\omega^{b}$ %accordingly without changing their cohomology classes,
so that $\sum_{k=1}^{\ell}\omega^j(e_k) = \omega^j(e)$, $\omega^j(e_k)\in\{0,\pm 1\}$ and $\omega^i(e_k)\omega^j(e_k)=0$ for all $i\ne j$ and all edges~$e_k$.
Each edge $e_k$ is then oriented so that $\omega^j(e_k)\in\{0,+1\}$, and the conductances $c(e_k)$ are chosen so that $\sum_{k=1}^{\ell}\frac{1}{c(e_k)} = \frac{1}{c(e)}$. %~\eqref{eq-ex-sufaces-via-synchronized} holds.
The current in the refined circuit is the same as in the original one, in the sense that $I(e_k) = I(e)$ for $k=1,\dots,\ell$.

\begin{example} \label{ex:reduc}
Any cohomological circuit can be simulated by a source-synchronized circuit (cf.~Example~\ref{ex-sufaces-via-synchronized}).
Indeed, %fix a basis $\{\omega^j\}$, %$\omega^1,\dots,\omega^{b}\in C^1(S;\mathbb{Z})$ of $H^1(S;\mathbb{Z})$,
%write $\mathcal E=\mathcal E_1\omega^1+\dots +\mathcal E_b\omega^b$,
refine the circuit, % (see the definition before the proof of Theorem~\ref{th-existence-uniqueness-cohomological}),
and for each \(j=1,\dots,b\), group all subdivision edges $e_k$ with $\omega^j(e_k)=+1$ into one synchronized source with the synchronized source voltage~$\mathcal E_j$. In the resulting source-synchronized circuit, the source voltage %~\eqref{eq-E-synchronized}
can be written as $\mathcal{E}(e)=\sum_{j=1}^b\omega^j(e)\,\mathcal E_j$. Viewed as a $1$-cochain, it coincides with %has the same cohomology class as the cochain
$\mathcal E=\mathcal E_1\omega^1+\dots +\mathcal E_b\omega^b$. Then for any simple cycle $\gamma$, the sum $\sum_{e\subset\gamma}\mathcal{E}(e)$ equals the evaluation of the cohomology class of $\mathcal E$ on~$\gamma$. Thus, the source-synchronized law~(V) and the cohomological law~(V) are equivalent. Laws~(I) are identical.
%Under this construction,
Hence, the currents in the circuits are the same. %in natural bijection.
\end{example}

% \begin{example} \label{ex:reduc}
% Any cohomological circuit can be simulated by a source-synchronized circuit (cf.~Example~\ref{ex-sufaces-via-synchronized}). Indeed, apply the refinement construction used in the proof of Theorem~\ref{th-existence-uniqueness-cohomological}, and then group all sub-edges with the same nonzero cocycle label into synchronized sources.
% The resulting source-synchronized circuit has the same current as the original cohomological circuit.
% \end{example}

% \begin{remark}
% It is easy to see that the source-synchronized circuit constructed in
% Example~\ref{ex:reduc} has the same current as the original cohomological circuit.
% \end{remark}

\begin{remark}\label{rem-arbitrary-phi}
    Without the assumption of the injectivity of the attaching maps $\phi_i$, the definition of the two-dimensional (co)boundary operator should be modified as follows.
    This time, a $1$-cell $e$ can contribute to the boundary of a $2$-cell $f=\mathrm{Int}\,D^2_i$
    with a nontrivial multiplicity (also known as the incidence number).
    Namely, let $D^2_i$ be attached to the $1$-dimensional cell subcomplex $S_1\subset S$ by a map $\phi_i\colon\partial D^2_i\to S_1$. In the image $\phi_i(\partial D^2_i)$, contract all the edges except for~$e$. The degree of the resulting map %$\partial D^2_i\to e/\partial e$
    between two topological circles $\partial D^2_i$ and $e/\partial e$ is the desired \emph{incidence number} $[f:e]$. It is well-defined (if $e\subset \partial f$), because both $f$ and $e$ are equipped with a fixed orientation. The operators $\partial \colon C_2(S;\mathbb{R}) \to C_{1}(S;\mathbb{R})$, $I\mapsto \partial I$, and
    $\delta \colon C_1(S;\mathbb{R}) \to C_{2}(S;\mathbb{R})$,
    $\mathcal{E}\mapsto\delta \mathcal{E}$, are then defined by
\[
[\partial I](e)=\sum_{f:\partial f\supset e} [f:e]I(f)
\qquad\text{and}\qquad
[\delta \mathcal{E}](f)=\sum_{e\subset \partial f} [f:e]\mathcal{E}(e),
\]
where $e$ and $f$ denote $1$- and $2$-cells, respectively. The remaining definitions remain identical.
\end{remark}

\section{Currents and response}
\label{sec-response}

In this section, we prove combinatorial formulae for the currents and response.

\subsection{Source-synchronized networks}
\label{ssec-current-source-synchronized-networks}

We start with a combinatorial formula for the current in a source-synchronized circuit.
%Let us introduce notation
For %an oriented path
a simple cycle or simple path $\gamma$ %in the circuit
and an %(oriented)
edge $e$, denote
\begin{equation}\label{eq-bracket-cycle-edge}
\langle \gamma,e\rangle:=\begin{cases}
    +1, &\text{if $e$ is oriented along }\gamma,\\
    -1, &\text{if $e$ is oriented opposite to }\gamma,\\
    0,  &\text{if $e$ is not contained in }\gamma.
\end{cases}
\end{equation}

For a simple cycle or simple path $\gamma$ and a synchronized source $\omega^i$, define
\[
\langle \gamma,\omega^i\rangle := \sum_{e \in \omega^i} \langle \gamma,e\rangle.
\]

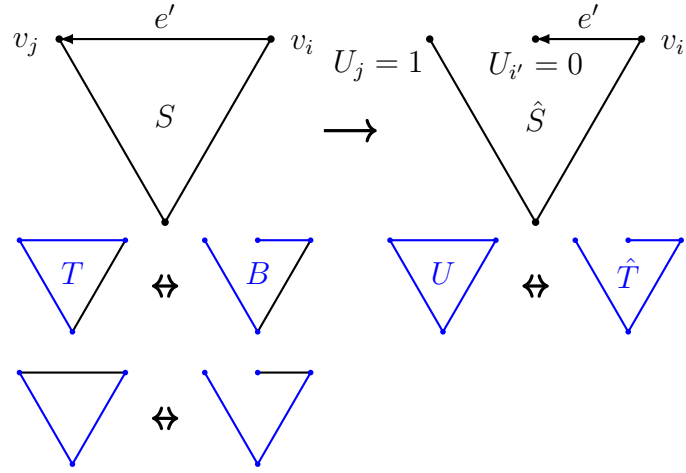
\begin{figure}[h!]
\begin{tikzpicture}[scale=0.7, thick]

% =========================================================
%                    ВЕРХНЯЯ ЧАСТЬ
% =========================================================

% ---------- граф S ----------
\begin{scope}

\coordinate (A) at (-2,0);
\coordinate (B) at ( 2,0);
\coordinate (C) at ( 0,-3.46);

% рёбра
\draw (A) -- (C);
\draw[latex-] (A) to node[above=0pt] {$e'$} (B);
\draw (B) -- (C);

% вершины
\fill (A) circle (2pt);
\fill (B) circle (2pt);
\fill (C) circle (2pt);

% подписи верхних вершин
\node[left=3pt, yshift=-2pt]  at (A) {$v_j$};
\node[right=3pt, yshift=-2pt] at (B) {$v_i$};

% название графа
\node at (0,-1.45) {$S$};

\end{scope}

% ---------- стрелка S -> \hat S ----------
\draw[->, very thick]
    (3,-1.73) -- (4,-1.73);

% ---------- граф \hat S ----------
\begin{scope}[xshift=7cm]

\coordinate (A) at (-2,0);
\coordinate (M) at ( 0,0);
\coordinate (B) at ( 2,0);
\coordinate (C) at ( 0,-3.46);

% рёбра
\draw (A) -- (C);
\draw[latex-] (M) to node[above=0pt] {$e'$} (B);
\draw (B) -- (C);

% вершины
\fill (A) circle (2pt);
\fill (M) circle (2pt);
\fill (B) circle (2pt);
\fill (C) circle (2pt);

\node[left=-2pt, yshift=-10pt] at (A) {$U_j=1$};
\node[yshift=-10pt]             at (M) {$U_{i'}=0$};
\node[right=3pt, yshift=-2pt] at (B) {$v_i$};

% название графа
\node at (0,-1.45) {$\hat S$};

\end{scope}

% =========================================================
%                    НИЖНЯЯ ЧАСТЬ
% =========================================================

\begin{scope}[yshift=-6.3cm, scale=0.5]

% =========================================================
%                    НИЖНИЙ РЯД
% =========================================================

% ---------- первый маленький граф ----------
\begin{scope}[xshift=-3.5cm]

\coordinate (A) at (-2,0);
\coordinate (B) at ( 2,0);
\coordinate (C) at ( 0,-3.46);

\draw[blue] (A) -- (C);
\draw       (A) -- (B);
\draw[blue] (B) -- (C);

\fill[blue] (A) circle (3pt);
\fill[blue] (B) circle (3pt);
\fill[blue] (C) circle (3pt);

% стрелка
\draw[<->, very thick]
    (3,-1.73) -- (4,-1.73);

\end{scope}

% ---------- второй маленький граф ----------
\begin{scope}[xshift=3.5cm]

\coordinate (A) at (-2,0);
\coordinate (M) at ( 0,0);
\coordinate (B) at ( 2,0);
\coordinate (C) at ( 0,-3.46);

\draw[blue] (A) -- (C);
\draw       (M) -- (B);
\draw[blue] (B) -- (C);

\fill[blue] (A) circle (3pt);
\fill[blue] (M) circle (3pt);
\fill[blue] (B) circle (3pt);
\fill[blue] (C) circle (3pt);

\end{scope}

% =========================================================
%                    ВЕРХНИЙ РЯД
% =========================================================

% ---------- T ----------
\begin{scope}[yshift=5cm,xshift=-3.5cm]

\coordinate (A) at (-2,0);
\coordinate (B) at ( 2,0);
\coordinate (C) at ( 0,-3.46);

\draw[blue] (A) -- (C);
\draw[blue] (A) -- (B);
\draw       (B) -- (C);

\fill[blue] (A) circle (3pt);
\fill[blue] (B) circle (3pt);
\fill[blue] (C) circle (3pt);

% подпись внутри
\node[blue] at (0,-1.25) {$T$};

% стрелка
\draw[<->, very thick]
    (3,-1.73) -- (4,-1.73);

\end{scope}

% ---------- B ----------
\begin{scope}[yshift=5cm,xshift=3.5cm]

\coordinate (A) at (-2,0);
\coordinate (M) at ( 0,0);
\coordinate (B) at ( 2,0);
\coordinate (C) at ( 0,-3.46);

\draw[blue] (A) -- (C);
\draw[blue] (M) -- (B);
\draw       (B) -- (C);

\fill[blue] (A) circle (3pt);
\fill[blue] (M) circle (3pt);
\fill[blue] (B) circle (3pt);
\fill[blue] (C) circle (3pt);

% подпись внутри
\node[blue] at (0,-1.25) {$B$};

\end{scope}

% ---------- U ----------
\begin{scope}[yshift=5cm,xshift=10.5cm]

\coordinate (A) at (-2,0);
\coordinate (B) at ( 2,0);
\coordinate (C) at ( 0,-3.46);

\draw[blue] (A) -- (C);
\draw[blue] (A) -- (B);
\draw[blue] (B) -- (C);

\fill[blue] (A) circle (3pt);
\fill[blue] (B) circle (3pt);
\fill[blue] (C) circle (3pt);

% подпись внутри
\node[blue] at (0,-1.25) {$U$};

% стрелка
\draw[<->, very thick]
    (3,-1.73) -- (4,-1.73);

\end{scope}

% ---------- \hat T ----------
\begin{scope}[yshift=5cm,xshift=17.5cm]

\coordinate (A) at (-2,0);
\coordinate (M) at ( 0,0);
\coordinate (B) at ( 2,0);
\coordinate (C) at ( 0,-3.46);

\draw[blue] (A) -- (C);
\draw[blue] (M) -- (B);
\draw[blue] (B) -- (C);

\fill[blue] (A) circle (3pt);
\fill[blue] (M) circle (3pt);
\fill[blue] (B) circle (3pt);
\fill[blue] (C) circle (3pt);

% подпись внутри
\node[blue] at (0,-1.25) {$\hat T$};

\end{scope}

\end{scope}

\end{tikzpicture}
\caption{Top: A source-synchronized circuit $S$ (left) gives rise to an ordinary electric circuit $\hat S$ (%with the Dirichlet boundary conditions;
right) by introducing a new vertex $v_{i'}$ and rerouting the edge $e'$ from $v_j$ to $v_{i'}$. Bottom left: The spanning trees $T$ of $S$ (blue) are in bijection with spanning bitrees $B$ of $\hat S$ (blue). Bottom right: %bijection between spanning connected subgraphs
The unicycles $U$ of $S$ (blue) %with a unique cycle and
are in bijection spanning trees $\hat T$ of $\hat S$ (blue). See the proof of Proposition~\ref{prop-current}.}
\label{fig-reroute}
\end{figure}

\begin{proposition}\label{prop-current}
The current through an edge $e$ of a source-synchronized circuit equals
\begin{equation}\label{eq-prop-current}
    I(e)=\frac{\sum_{U,e'} c(U)\langle \gamma(U),e\rangle\langle \gamma(U),e'\rangle\mathcal{E}(e')}{\sum_T c(T)},
\end{equation}
where the sum in the numerator is over all spanning connected subgraphs $U$ having a unique (up to orientation reversal) simple cycle $\gamma(U)$ and over all edges~$e'$, and the sum in the denominator is over %all
spanning trees~$T$.
\end{proposition}

Here, the numerator is well-defined, that is, it does not depend on the orientation of ~$\gamma(U)$.

\begin{proof}
By linearity, it suffices to consider the case where $\mathcal E(e')=1$ for some fixed edge $e'$ and $\mathcal E({e})=0$ for each %edge
${e}\neq e'$. In this case, the source-synchronized circuit $S$ gives rise to an ordinary
electric circuit $\hat S$ as follows; see Fig.~\ref{fig-reroute}. Let $v_i$ and $v_j$ be the starting %point
and the endpoint of $e'$. %, respectively.
%, where $e$ is an edge from $v_j$ and $v_i$.
Introduce a new vertex $v_{i'}$ and reroute the endpoint of $e'$ from $v_j$ to $v_{i'}$. %The resulting circuit is
We get an %ordinary
electric
circuit $\hat S$  with two boundary vertices $v_{i'}$ and $v_j$ and prescribed voltages $U_{i'}=0$ and $U_j=1$.

It is well-known~\cite[Proposition~2.7]{PSS} that %in this setting
the current through an oriented edge $e$ in $\hat S$ equals %the resulting circuit is given by
\begin{equation}\label{eq-current-tree}
    I(e)=\frac{\sum_{T^+} c(T^+)-\sum_{T^-} c(T^-)}{\sum_{B} c(B)},
\end{equation}
where the sum in the denominator is over all spanning bitrees (spanning forests with exactly two components) $B$ such that $v_{i'}$ and $v_j$ are in different components, and the sums in the numerator are over all spanning trees $T^+$ and $T^-$ such that the oriented path from $v_j$ to $v_{i'}$ in the tree passes the edge $e$ in the direction of $e$ and opposite to the direction of $e$, respectively.

There is a natural bijection between spanning trees $T$ of $S$ and spanning bitrees $B$ of $\hat S$ separating $v_{i'}$ and $v_j$. Indeed, if $T$ is a spanning tree of $S$, then rerouting the edge $e'$ produces a spanning bitree $B$ of $\hat S$ with two components, one containing $v_{i'}$ and the other containing $v_j$. If $T$ does not contain $e'$, then one component consists of the single vertex~$v_{i'}$. Conversely, given such a bitree $B$, rerouting the endpoint of $e'$ from $v_{i'}$ back to $v_j$ reconstructs a spanning tree $T$ of $S$. Therefore,
$
\sum_B c(B)=\sum_T c(T).
$

Similarly, each spanning tree $\hat T$ of $\hat S$ corresponds, by rerouting the endpoint of $e'$ from $v_{i'}$ back to $v_j$, to a spanning connected subgraph $U$ of $S$ with a unique cycle $\gamma(U)$. Moreover, %the sign of
the contribution of $\hat T$ in the formula above is exactly
$
c(U)\langle\gamma(U),e\rangle\langle\gamma(U),e'\rangle.
$

Summing over all $\hat T$ gives the numerator, and the denominator is $\sum_B c(B)=\sum_T c(T)$.
\end{proof}

The spanning connected subgraphs $U$ having a unique simple cycle $\gamma$ (up to orientation reversal) are referred to as \emph{unicycles} or \emph{cycle-rooted spanning trees}. This is a particular case of Kenyon's \emph{cycle-rooted spanning forests}, or $CRSF$-s \cite{Kenyon-11} with one component.

%%%%%%%%%%%%%%%%%
%\mscomm{Adopt the following examples to the new notation.}
%
%   \begin{figure}[ht]
%%    \begin{center}
%\scalebox{0.8}{\input{kir1.pstex_t}}
%%    \end{center}
%    \caption{}
%    \label{fig:kir1}
%\end{figure}
%
%\begin{example}
%In the example in Figure \ref{fig:kir1} surface $S$ is a cylinder, network $N$ has two boundary vertices and two internal vertices.
%We have three zippers carrying voltages $w_1$, $w_2$ and $w_3$, respectively.
%Solving the linear system of flow conditions and generalized Kirchoff law, we find
%$$i_a = \frac{C_a C_b C_c w_1+C_a C_b C_d w_1-C_a C_b C_d w_2+C_a C_b C_d w_3+C_a C_c C_d w_1}{C_aC_b+C_aC_d+C_bC_c+C_bC_d+C_cC_d},$$
%$$i_b = \frac{-C_aC_bC_dw_1+C_aC_bC_dw_2-C_aC_bC_dw_3+C_bC_cC_dw_2-C_bC_cC_dw_3}{C_aC_b+C_aC_d+C_bC_c+C_bC_d+C_cC_d},$$
%$$i_c = \frac{C_aC_bC_cw_1+C_aC_cC_dw_1+C_bC_cC_dw_2-C_bC_cC_dw_3}{C_aC_b+C_aC_d+C_bC_c+C_bC_d+C_cC_d},$$
%$$i_d = \frac{C_aC_bC_dw_1-C_aC_bC_dw_2+C_aC_bC_dw_3-C_bC_cC_dw_2+C_bC_cC_dw_3}{C_aC_b+C_aC_d+C_bC_c+C_bC_d+C_cC_d}.$$
%\end{example}
%
%  \begin{figure}[ht]
%%    \begin{center}
%\scalebox{0.8}{\input{kir2.pstex_t}}
%%    \end{center}
%    \caption{}
%    \label{fig:kir2}
%\end{figure}
%
%\begin{example} \label{ex:2}
% Example.
%\end{example}
%%%%%%%%%%%%%%%%%%%%%%%%%%

\begin{proposition}\label{prop-response} The response matrix of a source-synchronized network has the entries
$$L_{i}^{j}=\frac{\sum_{U} c(U)\langle \gamma(U),\omega^i\rangle\langle \gamma(U),\omega^j\rangle}{\sum_T c(T)},$$
where the sum in the numerator is over all spanning connected subgraphs $U$ having a unique (up to orientation reversal) simple cycle $\gamma(U)$, and the sum in the denominator is over %all
spanning trees~$T$. In particular, $L$ is symmetric and non-negatively definite.
%its diagonal entries are non-negative.
\end{proposition}

\begin{proof}
By definition, $L_{i}^{j}$ is the current through $\omega^i$ when the synchronized source voltage at $\omega^j$ is $1$,
%a unit voltage is applied to $\omega^j$,
and all other source voltages vanish. %are zero.
Summing the edge currents given by Proposition~\ref{prop-current} over $e\in \omega^i$, we get the %desired
formula for $L_{i}^{j}$. Symmetry and non-negative definiteness %non-negativity of the diagonal entries
follow. %immediately.
\end{proof}

\subsection{Cohomological networks}

Let us state analogous results for cohomological circuits.

Identify a simple %(oriented)
cycle $\gamma$ in a cohomological circuit $S$ with the corresponding $1$-cycle %chain
$$
\gamma = \sum_e \langle \gamma,e\rangle\, e,
$$
where $\langle \gamma,e\rangle \in \{0,\pm1\}$ records if the edge $e$
appears in $\gamma$ and with what orientation; see~\eqref{eq-bracket-cycle-edge}.

Introduce the natural pairing between chains and cochains:
\[
\langle \cdot , \cdot \rangle : C^1(S;\mathbb{R}) \times C_1(S;\mathbb{R}) \to \mathbb{R}, \qquad \langle \mathcal{E}, I \rangle=\sum_e \mathcal{E}(e)I(e),
\]
where the sum is over all $1$-cells $e$. It descends to a %natural
pairing between homology and cohomology
\[
\langle \cdot , \cdot \rangle : H^1(S;\mathbb{R}) \times H_1(S;\mathbb{R}) \to \mathbb{R}.
\]

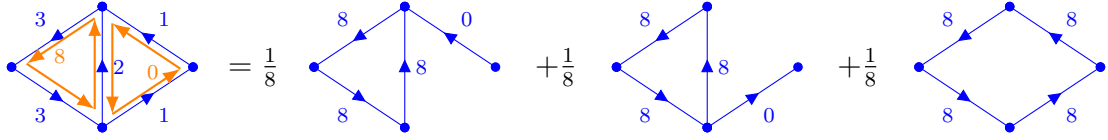
\begin{figure}[h]
 \center
  \begin{tabular}{c}
\begin{tikzpicture}[scale=0.4]
\coordinate (A) at (0,0);
\coordinate (B) at (-3,2);
\coordinate (C) at (-3,-2);
\coordinate (D) at (-6,0);

\coordinate (G) at ($ (A)!1/3!(B) !1/3!(C) $);
\coordinate (G1) at ($ (B)!1/3!(C) !1/3!(D) $);

\def\lambda{0.75}

\draw[color=blue]
(A) to[short, *-*, i_={\tiny $1$}] (B)
(C) to[short, *-*, i_={\tiny $2$}] (B)
(C) to[short, *-*, i_={\tiny $1$}] (A)
(B) to[short, *-*, i_={\tiny $3$}] (D)
(D) to[short, *-*, i_={\tiny $3$}] (C);

\coordinate (a) at ($(G)!\lambda!(A)$);
\coordinate (b) at ($(G)!\lambda!(B)$);
\coordinate (c) at ($(G)!\lambda!(C)$);

\coordinate (b1) at ($(G1)!\lambda!(B)$);
\coordinate (d1) at ($(G1)!\lambda!(D)$);
\coordinate (c1) at ($(G1)!\lambda!(C)$);

\draw[-{Latex}, thick, draw=orange] (a) -- (b);
\draw[-{Latex}, thick, draw=orange] (b) -- (c);
\draw[-{Latex}, thick, draw=orange] (c) -- (a);

\draw[-{Latex}, thick, draw=orange] (b1) -- (d1);
\draw[-{Latex}, thick, draw=orange] (d1) -- (c1);
\draw[-{Latex}, thick, draw=orange] (c1) -- (b1);
\node[text=orange] at ($(G)!0.2!(A)$) {\tiny $0$};
\node[text=orange] at ($(G1)!0.2!(D)$) {\tiny $8$};
\node at (2,0) {$=\frac{1}{8}$};

\begin{scope}[xshift=10cm]
\coordinate (A) at (0,0);
\coordinate (B) at (-3,2);
\coordinate (C) at (-3,-2);
\coordinate (D) at (-6,0);

\coordinate (G) at ($ (A)!1/3!(B) !1/3!(C) $);
\coordinate (G1) at ($ (B)!1/3!(C) !1/3!(D) $);

\def\lambda{0.75}

\draw[color=blue]
(A) to[short, *-*, i_={\tiny $0$}] (B)
(C) to[short, *-*, i_={\tiny $8$}] (B)
(B) to[short, *-*, i_={\tiny $8$}] (D)
(D) to[short, *-*, i_={\tiny $8$}] (C);

\coordinate (a) at ($(G)!\lambda!(A)$);
\coordinate (b) at ($(G)!\lambda!(B)$);
\coordinate (c) at ($(G)!\lambda!(C)$);

\coordinate (b1) at ($(G1)!\lambda!(B)$);
\coordinate (d1) at ($(G1)!\lambda!(D)$);
\coordinate (c1) at ($(G1)!\lambda!(C)$);

\end{scope}

\node at (12,0) {+$\frac{1}{8}$};
\begin{scope}[xshift=20cm]
\coordinate (A) at (0,0);
\coordinate (B) at (-3,2);
\coordinate (C) at (-3,-2);
\coordinate (D) at (-6,0);

\coordinate (G) at ($ (A)!1/3!(B) !1/3!(C) $);
\coordinate (G1) at ($ (B)!1/3!(C) !1/3!(D) $);

\def\lambda{0.75}

\draw[color=blue]
(C) to[short, *-*, i_={\tiny $8$}] (B)
(C) to[short, %dashed,
*-*, i_={\tiny $0$}] (A)
(B) to[short, *-*, i_={\tiny $8$}] (D)
(D) to[short, *-*, i_={\tiny $8$}] (C);

\coordinate (a) at ($(G)!\lambda!(A)$);
\coordinate (b) at ($(G)!\lambda!(B)$);
\coordinate (c) at ($(G)!\lambda!(C)$);

\coordinate (b1) at ($(G1)!\lambda!(B)$);
\coordinate (d1) at ($(G1)!\lambda!(D)$);
\coordinate (c1) at ($(G1)!\lambda!(C)$);

\end{scope}

\node at (22,0) {+$\frac{1}{8}$};

\begin{scope}[xshift=30cm]
\coordinate (A) at (0,0);
\coordinate (B) at (-3,2);
\coordinate (C) at (-3,-2);
\coordinate (D) at (-6,0);

\coordinate (G) at ($ (A)!1/3!(B) !1/3!(C) $);
\coordinate (G1) at ($ (B)!1/3!(C) !1/3!(D) $);

\def\lambda{0.75}

\draw[color=blue]
(A) to[short, *-*, i_={\tiny $8$}] (B)
(C) to[short, *-*, i_={\tiny $8$}] (A)
(B) to[short, *-*, i_={\tiny $8$}] (D)
(D) to[short, *-*, i_={\tiny $8$}] (C);

\coordinate (a) at ($(G)!\lambda!(A)$);
\coordinate (b) at ($(G)!\lambda!(B)$);
\coordinate (c) at ($(G)!\lambda!(C)$);

\coordinate (b1) at ($(G1)!\lambda!(B)$);
\coordinate (d1) at ($(G1)!\lambda!(D)$);
\coordinate (c1) at ($(G1)!\lambda!(C)$);

\end{scope}

\end{tikzpicture}
\end{tabular}
    \caption{Left: a cohomological circuit $S$ consisting of $5$ edges of conductance~$1$. Two basis cycles are depicted in orange, and the source voltage $\mathcal{E}\in H^1(S;\mathbb{R})$ is uniquely determined by its values on the cycles, also shown in orange. The current is a formal linear combination of edges with the coefficients shown in blue. Right: The current equals the sum over all possible unicycles $U$ with the coefficients $\langle\mathcal{E},\gamma(U)\rangle$, divided by $\sum_T c(T) = 8$. Only the unicycles $U$ with non-zero coefficients $\langle\mathcal{E},\gamma(U)\rangle$ are shown. See Example~\ref{ex-theta} and Corollary~\ref{cor-cohomological-current}. }
    \label{fig:current-cohomological-network}
\end{figure}

\begin{corollary} \label{cor-cohomological-current} (See Fig.~\ref{fig:current-cohomological-network}) %Under the notation of Proposition~\ref{prop-current},
    For any cohomological circuit, we have
    %$$I=\frac{\sum_{U,e}c(U)\langle\mathcal{E},\gamma(U)\rangle \langle \gamma(U),e\rangle e}{\sum_T c(T)},$$
$$I=\frac{\sum_{U} c(U)\,\langle \mathcal E,\gamma(U)\rangle\,\gamma(U)}{\sum_T c(T)},$$
   %where we use the notation from the previous proposition.
    %and $\imath$ is the inclusion of $U$ into the circuit.
    where the sum in the numerator is over all spanning connected subgraphs $U$ having a unique (up to orientation reversal) simple cycle $\gamma(U)$, %and all edges $e$,
    and the sum in the denominator is over all
spanning trees~$T$.
\end{corollary}

\begin{proof}
We first show that the identity is invariant under circuit refinement (see the definition after Example~\ref{ex-surface-2-cohomological}; cf.~Examples~\ref{ex-sufaces-via-synchronized} and~\ref{ex:reduc}), and then prove it for a convenient refinement.

Subdivide an edge $e$ into $e_1,\dots,e_\ell$ with
$$
\frac{1}{c(e)}=\sum_{i=1}^{\ell}\frac{1}{c(e_i)}.
$$
A spanning tree or unicycle containing $e$ refines to one containing all of
$e_1,\dots,e_\ell$, multiplying its weight by $\prod_i c(e_i)/c(e)$. A spanning tree or unicycle not containing $e$ gives rise to $\ell$ ones obtained by omitting exactly one of the edges $e_i$, and the total weight of these $\ell$ subgraphs is
$$
\prod_i c(e_i)\sum_i\frac{1}{c(e_i)}=\frac{\prod_i c(e_i)}{c(e)}.
$$
Thus, after refinement, every contribution to the numerator and to the denominator is multiplied by
the same factor $\prod_i c(e_i)/c(e)$, while $\langle \mathcal E,\gamma(U)\rangle$ is
unchanged, where we identify each cycle in the original graph with a refined cycle in the refined graph. Hence, the quotient is unchanged. Clearly, the current $I$ is also unchanged (cf.~Example~\ref{ex-sufaces-via-synchronized}).

By Example \ref{ex:reduc}, %after finitely many refinements,
a cohomological circuit can be realized as a source-synchronized circuit after suitable refinement; % $\widetilde S$; %by the preceding paragraph
it suffices to prove the identity for the latter circuit.
Multiplying the identity from Proposition \ref{prop-current} by the $1$-chain $e$ and summing over all edges $e$, %and using $\sum_{e_k}\langle\gamma(U),e_k\rangle\,e_k=\gamma(U)$,
we get
$$
I=\frac{\sum_{U} c(U)\Big(\sum_{e'}\langle\gamma(U),e'\rangle \mathcal E(e')\Big)
\Big(\sum_{e}\langle\gamma(U),e\rangle\,e\Big)
%\gamma(U)
}
{\sum_T c(T)},
$$
where $\mathcal E(e')$ is the source voltage in the source-synchronized circuit, viewed as a 1-cocycle
% value on $e'$ of the cocycle %$\mathcal E_1\omega^1+\dots +\mathcal E_b\omega^b$
%constructed in~Example \ref{ex:reduc},
representing the cohomology class $\mathcal E$. Since
% $\sum_e\langle\gamma(U),e\rangle \mathcal E(e)=\langle \mathcal E,\gamma(U)\rangle$ and $\sum_{e_k}\langle\gamma(U),e_k\rangle\,e_k=\gamma(U)$,
$\sum_{e'}\langle\gamma(U),e'\rangle \mathcal E(e')=\langle \mathcal E,\gamma(U)\rangle$ and $\sum_{e}\langle\gamma(U),e\rangle\, e=\gamma(U)$,
the desired formula follows.
\end{proof}

\begin{example}  \label{ex-theta}
Consider the example in Fig. \ref{fig:current-cohomological-network}. Here, all edges have conductance $1$, and the source voltage $\mathcal{E}$ is the linear function that assigns value $8$ to the left orange cycle and value $0$ to the right one. To find the resulting current $I$, %on each edge,
we sum over all possible unicycles $U$, and divide by $\sum_T c(T) = 8$. The only three unicycles $U$ with non-zero pairing $\langle\mathcal{E},\gamma(U)\rangle$ are shown, together with their contributions for each edge.  Note that one of the unicycles is the sum of the two original basis cycles, shown in orange.
\end{example}

For a connected spanning subgraph $U\subset S$, the inclusion
$\imath\colon U\to S$ induces the pullback
$\imath^*\colon H^1(S;\mathbb Z)\to H^1(U;\mathbb Z)$. A choice of a basis
$\Gamma=(\gamma_1,\dots,\gamma_k)$ of $H_1(U;\mathbb Z)$ gives rise to the
(non-canonical) standard scalar product
$H^1(U;\mathbb Z)\times H^1(U;\mathbb Z)\to\mathbb Z$ given by
$$
\langle\alpha,\beta\rangle_{U,\Gamma} := \sum_{i=1}^{k}\alpha(\gamma_i)\beta(\gamma_i).
$$
Equivalently, the basis of $H^1(U;\mathbb Z)$ dual to $\Gamma$ is orthonormal. If
$U$ is a unicycle, then $\Gamma$ is unique up to sign, hence
$\langle\alpha,\beta\rangle_{U,\Gamma}=:\langle\alpha,\beta\rangle$ is uniquely
determined by $U$.

%Let $S$ be a cohomological network and let $\omega^1,\dots,\omega^b$ be a basis
Let $\omega^1,\dots,\omega^{b}\in C^1(S;\mathbb{Z})$ be cocycles whose cohomology classes form a basis of $H^1(S;\mathbb Z)$. Let $I_i\in C_1(S;\mathbb R)$ be the current in the
cohomological circuit with $\mathcal E=\omega^i$. Then, in the basis $\omega^1,\dots,\omega^{b}$ and the dual basis of $H_1(S;\mathbb R)$, the response has the matrix %with the entries
$L_i^j=\langle\omega^j,I_i\rangle$.
%The expression is well-defined, that is, it does not depend on the choice of a %particular
%cochain representing the cohomology class $\omega^j$, because $I_i$ is a cycle by law~(I).
%
Pairing the formula in Corollary~\ref{cor-cohomological-current} with the classes $\omega^j$ gives the following result.

\begin{corollary}
    For any cohomological network, we have $L_{i}^{j}={\sum_{U}c(U)\langle\imath^*\omega^i,\imath^*\omega^j}\rangle/{\sum_T c(T)}$, where we use the notation from Corollary \ref{cor-cohomological-current}, and $\imath$ is the inclusion of $U$ into the network.
\end{corollary}

\begin{proof}
%Let $I_i$ be the current in the cohomological circuit with source voltage $\mathcal E=\omega^i$.
Pairing the formula in Corollary \ref{cor-cohomological-current}  with $\omega^j$ and using
bilinearity of $\langle\,\cdot\,,\,\cdot\,\rangle$, we get
%$$
%I_i=\frac{\sum_U c(U)\,\langle\omega^i,\gamma(U)\rangle\,\gamma(U)}{\sum_T c(T)},
%$$
%the sum being over all %spanning
%unicycles $U$.
%Pairing with $\omega^j$ and using
%bilinearity of $\langle\,\cdot\,,\,\cdot\,\rangle$, we get
$$
L^j_i=\langle\omega^j,I_i\rangle
=\frac{\sum_U c(U)\,\langle\omega^i,\gamma(U)\rangle\,
\langle\omega^j,\gamma(U)\rangle}{\sum_T c(T)}.
$$
For a unicycle $U$, the group $H_1(U;\mathbb Z)$ has rank one, with generator
represented by $\gamma(U)$ uniquely up to sign, so the standard scalar product on
$H^1(U;\mathbb Z)$ gives
$
\langle\imath^*\omega^i,\imath^*\omega^j\rangle
=\langle\omega^i,\gamma(U)\rangle\,\langle\omega^j,\gamma(U)\rangle ,
$
independently of the orientation of $\gamma(U)$. Substituting yields the desired formula.
%$$
%L^j_i=\frac{\sum_U c(U)\,\langle\imath^*\omega^i,\imath^*\omega^j\rangle}{\sum_T c(T)},
%$$
%as desired.
\end{proof}

\section{Voltages}
\label{sec-voltages}

\subsection{Source-synchronized networks}

In an electrical circuit, the fundamental quantities are the vertex voltages and the edge currents. Ohm’s law relates them by expressing each edge current as the product of the edge conductance and the voltage drop on it. In practice, it is more efficient to work with %vertex
voltages rather than currents, because this leads to fewer variables.
The Kirchhoff matrix %. This
is the matrix of the linear operator that maps the vector of voltages to the vector of currents at all vertices.
The response matrix is certain Schur's complement.

In a source-synchronized circuit, it is instructive to augment the usual vertex variables by variables corresponding to the synchronized sources. The Kirchhoff matrix then extends naturally to incorporate these additional degrees of freedom, encoding the contribution of each edge and each synchronized source to the overall current-balance relations.

Let %a source-synchronized
the circuit have $n$ vertices $v_1,\dots,v_n$ and $b$ synchronized sources $\omega^1,\dots,\omega^b$. %, with conductance $c(e)>0$ of each edge $e$.
For each vertex $v_k$, we define a real number $U_k$,  %\in \mathbb{R}$,
called the \emph{voltage} (or \emph{potential}) at $v_k$, such that for every edge $e$ from vertex a $v_i$ to vertex a $v_j$, the voltage drop along $e$ satisfies \emph{Ohm's law}
\begin{itemize}
\item[(C)]
$
\Delta U(e):=\shcomm{U_i - U_j = \frac{I(e)}{c(e)}-\mathcal{E}(e)}.
$
\end{itemize}

Clearly, by law (V), such numbers $U_k$ exist and are unique up to adding a constant.

Conversely, for any $\mathcal{E}_1,\dots,\mathcal{E}_b,\, U_1,\dots,U_n $, the function $I(e)=c(e)\bigl(U_i-U_j+\mathcal E(e)\bigr)$ given by law~(C) automatically satisfies law~(V) but not necessarily law~(I).

The \emph{Kirchhoff matrix} $\hat K$ is the matrix of the resulting linear operator
\[
(\mathcal{E}_1,\dots,\mathcal{E}_b,\, U_1,\dots,U_n)
\mapsto
(I_{\omega^1},\dots,I_{\omega^b},\, I_{v_1},\dots,I_{v_n})
\]
determined by Ohm's law~(C) (without assuming laws~(I) and~(V)).
Equivalently, it is the matrix of the \emph{energy dissipation form} (or the \emph{Dirichlet
energy})
$$
\sum_{i=1}^n U_iI_{v_i}+\sum_{j=1}^b \mathcal{E}_j I_{\omega^j}
=
\sum_e (\Delta U(e)+\mathcal E(e))I(e)
=
\sum_e c(e)(\Delta U(e)+\mathcal E(e))^2,
$$
where the last two sums are over all edges $e$, viewed as a quadratic form in the variables $\mathcal E_1,\dots,\mathcal E_b,U_1,\dots,U_n$.
%Here we use the convention that the matrix of a quadratic form $x^T Mx$ is $M$.
It is a $(b+n)\times(b+n)$ matrix %of size $(b+n)\times(b+n)$
with both rows and columns labeled by $\omega^1,\dots,\omega^b, v_1,\dots,v_n$ in %this
order, and the entries
\[
\begin{aligned}
\hat K_{v_i}^{v_j}
&=
\begin{cases}
-\displaystyle \sum_{e\ni v_i,v_j%e\,:\, v_i,v_j \in e
} c(e), &\text{if } v_i \ne v_j, \\
\displaystyle \sum_{e\ni v_i%e\,:\, v_i \in e
,\; e \text{ is not a loop}} c(e), &\text{if } v_i = v_j;
\end{cases}
\qquad
\hat K_{\omega^i}^{\omega^j}
=
\begin{cases}
0, \vphantom{\displaystyle \sum_{e\in \omega^i} c(e)} &\text{if } \omega^i\ne\omega^j, %\mscomm{(?)},
\\
\displaystyle \sum_{e\in \omega^i} c(e), &\text{if } \omega^i=\omega^j;
\end{cases}
\\
\hat K_{v_i}^{\omega^j} &= \hat K_{\omega^j}^{v_i}
= \sum_{e \in \omega^j,\; e \text{ starts at } v_i} c(e) - \sum_{e \in \omega^j,\; e \text{ ends at } v_i} c(e).
\end{aligned}
\]

% \begin{example}
% Let the graph have vertices $v_1,v_2,v_3=v_n$ and directed edges
% $e,f:v_1\to v_2$, $h:v_1\to v_3$, and $g:v_2\to v_3$.
% Let $\omega=\omega^1=\{e\}$ and take $P=Q=\{1\}$. In the labels
% $(\omega,\delta^1,\delta^2)$, the corresponding augmented Kirchhoff minor is
% $$
% K=
% \begin{pmatrix}
% c_e & c_e & -c_e\\
% c_e & c_e+c_f+c_h & -c_e-c_f\\
% -c_e & -c_e-c_f & c_e+c_f+c_g
% \end{pmatrix}.
% $$
% The contribution of the edge $e$ is
% $$
% K(e)=c_e
% \begin{pmatrix}
% 1&1&-1\\
% 1&1&-1\\
% -1&-1&1
% \end{pmatrix},
% $$
% which has rank one. Therefore every minor of $K(e)$ of size at least two vanishes.

% For example, consider the generalized diagonal
% $(\omega\to\delta^1,e),\quad
% (\delta^1\to\delta^2,f),\quad
% (\delta^2\to\omega,e)$.
% It is paired with
% $(\omega\to\omega,e),\quad
% (\delta^1\to\delta^2,f),\quad
% (\delta^2\to\delta^1,e)$. \mscommnew{[too informal notation]}
% The products of matrix entries are the same, but the two permutations have opposite signs.
% Equivalently, after fixing the middle factor $(\delta^1\to\delta^2,f)$, their sum is this fixed factor times the minor
% $K_{\omega}^{\delta^1}(e)K_{\delta^2}^{\omega}(e)
% -
% K_{\omega}^{\omega}(e)K_{\delta^2}^{\delta^1}(e)=0$.
% \end{example}

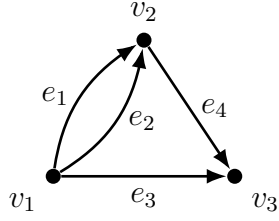
\begin{figure}
\begin{tikzpicture}[
    scale=0.6,
    vertex/.style={circle, fill=black, inner sep=2pt},
    edge/.style={-{Latex}, line width=1pt}
]
    % Вершины
    \node[vertex, label=below left:$v_1$]  (v1) at (0,0) {};
    \node[vertex, label=above:$v_2$]       (v2) at (2,3) {};
    \node[vertex, label=below right:$v_3$] (v3) at (4,0) {};

    % Параллельные рёбра
    \draw[edge]
        (v1) to[bend left=25]
        node[left] {$e_1$}
        (v2);

    \draw[edge]
        (v1) to[bend right=25]
        node[right] {$e_2$}
        (v2);

    % Нижнее ребро
    \draw[edge]
        (v1) --
        node[below] {$e_3$}
        (v3);

    % Правое ребро
    \draw[edge]
        (v2) --
        node[right] {$e_4$}
        (v3);
\end{tikzpicture}
\caption{The source-synchronized network of Example~\ref{ex-rank-one}.
%: three vertices, four directed
%edges, and one synchronized source $\omega=\{e_1\}$.
}%Optionally, we can show the connections corresponding to two terms that cancel]}}
\label{fig:ex-semi-connection}
\end{figure}

\begin{example}\label{ex-rank-one}
Consider a source-synchronized network with vertices $v_1,v_2,v_3$, directed edges $e_1,e_2,e_3,e_4$, and a single synchronized source $\omega^1=\{e_1\}$;
see Fig.~\ref{fig:ex-semi-connection}.
Its Kirchhoff matrix is
$$
\hat K=
\begin{pmatrix}
c(e_1) & c(e_1) & -c(e_1) & 0\\
c(e_1) & c(e_1)+c(e_2)+c(e_3) & -c(e_1)-c(e_2) & -c(e_3)\\
-c(e_1) & -c(e_1)-c(e_2) & c(e_1)+c(e_2)+c(e_4) & -c(e_4)\\
0 & -c(e_3) & -c(e_4) & c(e_3)+c(e_4)
\end{pmatrix}.
$$
%%%%%%%
% \shcomm{$$
% K=
% \begin{pmatrix}
% c(e_1) & c(e_1) & -c(e_1)\\
% c(e_1) & c(e_1)+c(e_2)+c(e_3) & -c(e_1)-c(e_2)\\
% -c(e_1) & -c(e_1)-c(e_2) & c(e_1)+c(e_2)+c(e_4)
% \end{pmatrix}.
% $$
The contribution of the edge $e_1$ is
$$
%K(e_1,e_1,e_1)=
c(e_1)
\begin{pmatrix}
1&1&-1& 0\\
1&1&-1& 0\\
-1&-1&1& 0\\
0&0&0& 0
\end{pmatrix},
$$
%%%%%
%$$
%K(e_1,e_1,e_1)=c(e_1)
%\begin{pmatrix}
% 1&1&-1\\
% 1&1&-1\\
% -1&-1&1
% \end{pmatrix},
% $$}
which has rank one. So, in the latter matrix, %$K(e_1,e_1,e_1)$
every minor of size at least two vanishes. %This will lead to additional cancellations in the minor expansions in Sec.~\ref{sec-connection-decomposition} compared to \cite[Lemma~3.12]{CurtisMorrow}.
%This shows additional cancellations in~\eqref{eq-connections} compared to determinant expansion as in \cite[Lemma~3.12]{CurtisMorrow}.
\end{example}

We now reformulate this definition in an equivalent and more conceptual way.

A \emph{trivial source} $\delta^i$ (or \emph{minus elementary coboundary}) associated with a vertex $v_i$ %, denoted by $\delta^i$,
is the set of edges starting at $v_i$ after reversing the edges that end at $v_i$.
For an edge $e$, % incident to $v_i$,
we then
write
\begin{align*}
\langle \delta^i, e \rangle &:=
\begin{cases}
1, & \text{if $e$ originally started at $v_i$ and $e$ is not a loop},\\
-1, & \text{if $e$ originally ended at $v_i$ and $e$ is not a loop},\\
0, & \text{if $e$ does not contain $v_i$ or $e$ is a loop};%\text{otherwise};
\end{cases}\\
\langle \omega^i, e \rangle &:=
\begin{cases}
1, & \text{if } e \in \omega^i,\\
0, & \text{if } e \notin \omega^i. %\text{otherwise}.
\end{cases}
\end{align*}

We now identify each vertex label $v_i$ with its trivial source $\delta^i$.
In other words, the rows and columns previously labeled by vertices $v_1,\dots,v_n$ are henceforth labeled by the trivial sources $\delta^1,\dots,\delta^n$.
With this identification, the Kirchhoff matrix admits the unified expression
\[
\hat K_{\lambda}^{\rho} = \sum_{e} \langle \lambda, e \rangle \langle \rho, e \rangle c(e),
\]
where $\lambda$ and $\rho$ run over all sources \(\omega^i\) and trivial sources \(\delta^i\). %This explain the rank-one phenomenon in Example~\ref{ex-rank-one}.

The \emph{augmented Kirchhoff matrix} $K$ is the $(b+n-1)\times(b+n-1)$ matrix obtained from $\hat K$
by deleting the last row and the last column (corresponding to $\delta^n$).

Denote by \( \mathrm{\hat{V}} \) the set of all vertices in the network, and let  \( \mathrm{ V} \) denote the set of vertices excluding the last vertex $v_n$.
Then denote by \( K_{\mathrm{V}}^{\mathrm{V}} \) the submatrix of \( K \) formed by the rows and columns labeled by all trivial sources \(\delta^i\) excluding $\delta^n$.

Let
$
M = \begin{pmatrix}
A & B \\
C & D
\end{pmatrix}
$
be a block matrix, where \(A\) and \(D\) are square matrices. If \(D\) is invertible, the \emph{Schur complement} of the block \(D\) in the matrix \(M\) is defined as $A - B D^{-1} C$.

\begin{lemma} \label{schur_complement}
    The response matrix is %given by
    the Schur complement of the minor $K_{\mathrm{V}}^{\mathrm{V}}$ in %the augmented Kirchhoff matrix~
    $K$.
\end{lemma}

The proof is essentially the same as the proof for ordinary electrical networks.

\begin{proof}
% Let us derive a relation between the voltages and the current.
% Denote %$I_{\omega^k}:= I_k$,
% $I_{\delta^k}:=I_{v_k}$, $\Delta U_{\omega^k}:= \mathcal E_k$,
% and $\Delta U_{\delta^k}:=U_k$. Then %the definition from
% Remark~\ref{rem-total-current} and Ohm's law~(C) can be rewritten as $I_\rho=\sum_e\langle\rho,e\rangle I(e)$ and $I(e)=c(e)\sum_{\rho}\langle\rho,e\rangle\,\Delta U_{\rho}$, respectively, where $\rho$ is an arbitrary source (synchronized or trivial), and $e$ is an arbitrary edge.
% %the sums are over all sources $\rho$ (synchronized or trivial).
% %%%%%%%%
% %For any source $\rho$ (synchronized or trivial) write
% %$I_\rho=\sum_e\langle\rho,e\rangle I(e)$, and set $\Delta U_{\omega^k}:= \mathcal E_k$,
% %$\Delta U_{\delta^k}:=U_k$.  By Ohm's law~%in the form
% %(C), we have $I(e)=c(e)\sum_{\rho'}\langle\rho',e\rangle\,\Delta U_{\rho'}$, where the sum is over all sources $\rho'$.
% Thus
% $$
% I_\rho=\sum_{\rho'}\Delta U_{\rho'}\sum_e\langle\rho',e\rangle\langle\rho,e\rangle c(e)
% =\sum_{\rho'}\hat K^{\rho}_{\rho'}\,\Delta U_{\rho'}.
% $$
% %
% In what follows, set $U_n=0$, as we can always add a constant to all the voltages. Then
% $$
% I_\rho=\sum_{\rho'}K^{\rho}_{\rho'}\,\Delta U_{\rho'},\qquad\text{for }\quad\rho\ne\delta^n.
% $$
%
% %Ordering the synchronized sources before the trivial ones, the full
% %%Since rows and columns labeled by $(\omega^1,\dots,\omega^b, \delta^1,\dots,\delta^n)$ in %this
% %order,
The augmented Kirchhoff matrix %appearing here
has the block form
$
%\hat
K=\begin{pmatrix} A & %\hat
B\\ %\hat
B^{\mathrm{T}} & %\hat
D\end{pmatrix},
$
where $A$ is the $b\times b$ block labeled by synchronized sources and $%\hat
D=K_{\mathrm{V}}^{\mathrm{V}}$ is the
$(n-1)\times (n-1)$ block labeled by trivial sources excluding $\delta^n$.
Set $U_n=0$, as we can always add a constant to all the voltages. Set $I(e):=c(e)\bigl(U_i-U_j+\mathcal E(e)\bigr)$.
With the notation $\Delta U_\omega:=(\mathcal E_1,\dots,\mathcal E_b)^{\mathrm T}$, $\Delta U_\delta:=(U_1,\dots,U_{n-1})^{\mathrm T}$,
$I_\omega:=(I_{\omega^1},\dots,I_{\omega^b})^{\mathrm T}$, and
$I_\delta:=(I_{v_1},\dots,I_{v_{n-1}})^{\mathrm T}$, we have
$$
\begin{pmatrix} A & %\hat
B\\ %\hat
B^{\mathrm T} & %\hat
D\end{pmatrix}
\begin{pmatrix}\Delta U_\omega\\ \Delta U_\delta\end{pmatrix}
=\begin{pmatrix} I_\omega\\ I_\delta\end{pmatrix}.
$$
The currents $I(e)=c(e)\bigl(U_i-U_j+\mathcal E(e)\bigr)$ automatically satisfy law~(V), and law (I) is equivalent to $I_\delta=B^{\mathrm T} \Delta U_\omega+D\Delta U_\delta=0$.
The sub-matrix $D$ is invertible by Theorem~\ref{th-existence-uniqueness-synchronized}.
%because $\det D=\det K_{\mathrm{V}}^{\mathrm{V}}=\sum_T c(T)>0$ is the sum over spanning trees by the matrix-tree theorem.
%%%%%%
%The block $\hat D$ is the weighted graph Laplacian of the network, hence singular: its kernel consists of the constant potentials $U_1=\dots=U_n$, because voltages are defined only up to a common additive constant. This does not change $i$ or $\Delta U_f$, since the columns of $\hat B$ sum to zero. We may therefore normalize $U_n=0$; deleting the row and column of $\delta^n$ replaces $\hat D$ by the reduced Laplacian $D=K^V_V$, which is invertible (indeed $\det D=\sum_T c(T)>0$ by the matrix-tree theorem), and
%$\hat B$ by $B$. With $\Delta U_g=(U_1,\dots,U_{n-1})^T$ the system becomes
%$$
%\begin{bmatrix} A & B\\ B^{T} & D\end{bmatrix}
%\begin{bmatrix}\Delta U_f\\ \Delta U_g\end{bmatrix}
%=\begin{bmatrix} i\\ 0\end{bmatrix}.
%$$
%The second block row then gives
Thus $\Delta U_\delta=-D^{-1}B^{\mathrm T} \Delta U_\omega$. Substituting into
the first block row gives $I_\omega=(A-BD^{-1}B^{\mathrm T} )\Delta U_\omega$. Hence the response
$\Delta U_\omega\mapsto I_\omega$ equals the Schur complement $A-BD^{-1}B^{\mathrm T} $ of
$D=K_{\mathrm{V}}^{\mathrm{V}}$ in $K$.
\end{proof}

\subsection{Cohomological networks}

In a cohomological circuit, the voltage is more naturally viewed as a multi-valued function, whose value may change after performing a topologically nontrivial loop. Thus, the voltage drop $\Delta U$ is conceptualized as a 1-cocycle that is not necessarily a coboundary of a 0-cochain $U$.
%Therefore, we work with the edge variables $\Delta U$ instead of vertex variables $U$.
Furthermore, since the value $\mathcal{E}(e)$ is
no longer defined, it is absorbed into $\Delta U$ in Ohm's law so that the latter becomes $\Delta U(e)=\shcomm{I(e)/c(e)=[V(I)](e)}$.

This is formalized as follows. For an arbitrary chain $I\in C_1(S;\mathbb{R})$, we define the \emph{voltage drop} as $\Delta U:=V(I)\in C^1(S;\mathbb{R})$.
If $I$ satisfies law~(V), then $\Delta U\in Z^1(S;\mathbb{R})$ is a cocycle. The \emph{energy dissipation form} (or the \emph{Dirichlet energy}) is the quadratic form on $Z^1(S;\mathbb{R})$ given by
$$
\langle \Delta U,I\rangle=\langle \Delta U,V^{-1}(\Delta U)\rangle=\sum_e c(e)(\Delta U(e))^2,
$$
%for each $\Delta U\in Z^1(S;\mathbb{R})$
where the summation is over all the edges $e$. %The \emph{Kirchhoff matrix} is the matrix of the energy dissipation form in \mscomm{TBC}

%To introduce the Kirchhoff matrix,
We construct a special basis of $Z^1(S;\mathbb{Z})$ as follows.
First, take cocycles $\omega^1,\dots,\omega^b\in Z^1(S;\mathbb{Z})$ whose cohomology classes form a basis of $H^1(S;\mathbb{Z})$. Second, for each vertex \(v_i\), define the  \emph{minus elementary coboundary} \(\delta^i\in Z^1(S;\mathbb{Z})\) by the following formula, for each edge $e$ of $S$:
$$
\delta^i(e):=
\begin{cases}
\shcomm{1}, & \text{if $e$ starts at $v_i$ and $e$ is not a loop},\\
\shcomm{-1}, & \text{if $e$ ends at $v_i$ and $e$ is not a loop},\\
0, & \text{if $e$ does not contain $v_i$ or $e$ is a loop}.
\end{cases}
$$
These $b+n$ cocycles span $Z^1(S;\mathbb{Z})$ and satisfy the single relation $\delta^1+\dots+\delta^n=0$, so that $\omega^1,\dots,\omega^b,\delta^1,\dots,\delta^{n-1}$ is a basis.

The \emph{Kirchhoff matrix} $\hat K$ is then defined as the matrix of the energy dissipation form
(or, equivalently, of the operator
$Z^1(S;\mathbb{R})\hookrightarrow  C^1(S;\mathbb{R})\overset{V^{-1}}{\to} C_1(S;\mathbb{R})\to C_1(S;\mathbb{R})/\partial C_2(S;\mathbb{R})$)
%$V^{-1}\colon Z^1(S;\mathbb{R})\to C_1(S;\mathbb{R})/\partial C_2(S;\mathbb{R})$)
with respect to the cocycles $\omega^1,\dots,\omega^b,\delta^1,\dots,\delta^n$.
It is the \((n+b) \times (n+b)\) matrix with the entries $\hat K_{\lambda}^{\rho} = \langle \lambda,V^{-1}(\rho)\rangle $, %whose rows and columns are labeled with
where $\lambda$ and $\rho$ run over both %basis cocycles
$\omega^1,\dots,\omega^b$ and %elementary coboundaries
$\delta^1,\dots,\delta^n$.
The \emph{augmented Kirchoff matrix} $K$ is obtained from $\hat K$ by removing the row and the column corresponding to the minus elementary coboundary $\delta^n$, so that $K$ is the matrix of the energy dissipation form in the above basis. % constructed above.
Analogously to the source-synchronized case, we get the following result.

\begin{lemma} \label{l-schur_complement-var}
    The response matrix of a cohomological circuit is the Schur complement of the minor $K_{\mathrm{V}}^{\mathrm{V}}$ in the augmented Kirchhoff matrix~$K$.
\end{lemma}

Note that the Kirchhoff matrix depends on the choice of cocycles $\omega^1,\dots,\omega^b$, whereas the response matrix depends only on their cohomology classes.
In the particular case of networks on surfaces, an analog of Lemma~\ref{l-schur_complement-var} was proved by a different method in \cite{Lam-25}.

%\newpage

\section{All-minors matrix-tree theorem}
\label{sec-all-minor}

In this section, we state combinatorial formulae for all minors of the response matrices of source-synchronized and cohomological networks. %. Consider a homological network
%with $b$ synchronized %voltage
%sources $\omega^{1},\dots,\omega^{b}$.
These formulae are proved in Secs.~\ref{sec-connection-decomposition}--\ref{sec-conclusion}.

\subsection{Source-synchronized networks}

Consider a source-synchronized network with $b$ synchronized %voltage
sources $\omega^{1},\dots,\omega^{b}$.
In what follows, $P,Q\subset\{1,\ldots,b\}$ are two subsets with the same number
of elements (denoted by $k$), and their elements $p_1<\dots<p_k$ and $q_1<\dots<q_k$ are always listed in increasing
order, %$P=\{p_1<\dots<p_k\}$ and $Q=\{q_1<\dots<q_k\}$,
so that $L_P^Q$ is a submatrix of~$L$. %in the usual sense.
Our goal is to express the minor
$\det L_P^Q$ as a weighted sum over certain subgraphs of the network.
These %graphs contributing to the sum
are so-called valid subgraphs, which we are going to define now. See Fig.~\ref{fig:valid-subgraphs}.

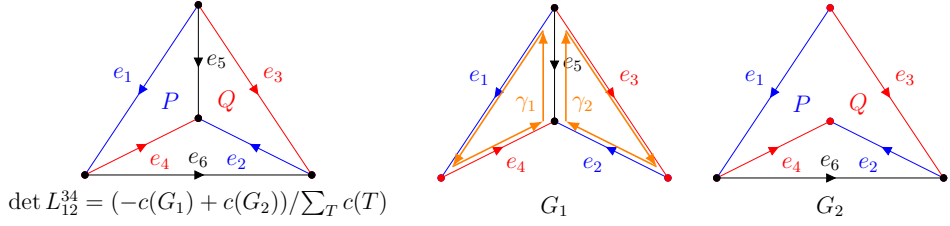
\begin{figure}
    \centering
    \scalebox{0.75}{\begin{tikzpicture}
\draw[color=blue]
(2,3) to[short, *-*, i_=$e_1$] (0,0)
(4,0) to[short, *-*, i=$e_2$] (2,1);
\draw[color=red]
(2,3) to[short, *-*, i=$e_3$] (4,0)
(0,0) to[short, *-*, i_=$e_4$] (2,1);
\draw
(2,3) to[short, *-*, i=$e_5$] (2,1)
(0,0) to[short, *-*, i=$e_6$] (4,0);
\node[text=blue] at (1.5,1.3) {$P$};
\node[text=red] at (2.5,1.3) {$Q$};
\node at (2,-0.5) {$\det L_{12}^{34}={(-c(G_1)+c(G_2))}/{\sum_T c(T)}$};
\end{tikzpicture}}
\quad
\scalebox{0.75}{\begin{tikzpicture}
\draw[color=blue]
(2,3) to[short, *-*, i_=$e_1$] (0,0)
(4,0) to[short, *-*, i=$e_2$] (2,1);
\draw[color=red]
(2,3) to[short, *-*, i=$e_3$] (4,0)
(0,0) to[short, *-*, i_=$e_4$] (2,1);
\draw
(2,3) to[short, *-*, i=$e_5$] (2,1);
%(0,0) to[short, *-*, i=$e_6$] (4,0);
\node[text=orange] at (1.5,1.3) {$\gamma_1$};
\node[text=orange] at (2.5,1.3) {$\gamma_2$};
\node at (2,-0.5) {$G_1$};
 \draw[-{Latex}, thick, draw=orange] (1.8,2.6) -- (0.2,0.2);
 \draw[-{Latex}, thick, draw=orange] (0.2,0.2) -- (1.8,1);
 \draw[-{Latex}, thick, draw=orange] (1.8,1) -- (1.8,2.6);
 \draw[-{Latex}, thick, draw=orange] (2.2,2.6) -- (3.8,0.2);
 \draw[-{Latex}, thick, draw=orange] (3.8,0.2) -- (2.2,1);
 \draw[-{Latex}, thick, draw=orange] (2.2,1) -- (2.2,2.6);
\end{tikzpicture}}
\quad
\scalebox{0.75}{\begin{tikzpicture}
\draw[color=blue]
(2,3) to[short, *-*, i_=$e_1$] (0,0)
(4,0) to[short, *-*, i=$e_2$] (2,1);
\draw[color=red]
(2,3) to[short, *-*, i=$e_3$] (4,0)
(0,0) to[short, *-*, i_=$e_4$] (2,1);
\draw
%(2,3) to[short, *-*, i=$e_5$] (2,1)
(0,0) to[short, *-*, i=$e_6$] (4,0);
\node[text=blue] at (1.5,1.3) {$P$};
\node[text=red] at (2.5,1.3) {$Q$};
\node at (2,-0.5) {$G_2$};
\end{tikzpicture}}
    \caption{An electrical circuit (left) %, where each edge contains an individual battery,
    and two subgraphs $G_1$ and $G_2$ that are valid relative to both $P=\{1,2\}$ and $Q=\{3,4\}$ separately (middle and right). The basis simple cycles $\gamma_1$ and $\gamma_2$ in $G_1$ are shown
    in orange. See Example~\ref{ex-all-minor-matrix-tree}.}
    \label{fig:valid-subgraphs}
\end{figure}

A set \( E \) of edges %in a network
is \textit{valid relative to} \( P \) if \(|E| = |P|\) and
all the synchronized sources $\omega^{p_1}, \ldots, \omega^{p_k}$ intersect \( E \).
%the set of synchronized sources intersecting \( E \) is \(\{\omega^{p_1}, \ldots, \omega^{p_k}\}\).
%
A subgraph \( G \) %of the network
is \emph{valid relative to} \( P \) if it contains a valid set of edges relative to \( P \) such that the subgraph obtained by deleting these edges from \( G \) is a spanning tree.

Informally, a subgraph is valid relative to $P$ if it contains a choice of $|P|$ edges, one contained in each source $\omega^p$, $p\in P$, such that removing these edges leaves a spanning tree.

Take a subgraph $G$ valid relative to both $P$ and $Q$ separately.
Given a basis of simple cycles $\Gamma=(\gamma_1, \ldots, \gamma_{|P|})$ in $G$ %\mscomm{$u_i\to \gamma_i$!}
(i.e., an ordered set of simple %oriented
cycles whose homology classes form a basis of $H_1(G;\mathbb{Z})$), %\mscomm{$U\to \Gamma$!}
introduce the $b\times b$ \emph{Gram matrix} $M=M(G,\Gamma)$ with the entries
$$M_{i}^{j} = \sum_{l=1}^{|P|} \langle \gamma_l , \omega^i \rangle \langle \gamma_l , \omega^j \rangle.$$
The \emph{multiplicity} of $G$ in the minor $L_P^Q$ is
$$
\mathrm{mult}_P^Q(G) := \det M_P^Q. %\left(A(G,U)_P^Q\right).
$$ We shall see that the multiplicity does not depend on the choice of the basis $\Gamma$.

\begin{theorem} \label{th-all-minor-matrix-tree}
    % If %$P, Q \in \{1,\ldots,b\}$ and
    % $P \cap Q=\emptyset$, then t
    The minor of the response matrix of a source-synchronized network is given by
    $$\det L_P^Q = \frac{\sum_G \mathrm{mult}_P^Q(G) \cdot c(G)}{\sum_T c(T)},$$
    where the sum in the numerator is over all subgraphs $G$ that are valid relative to both $P$ and $Q$ separately, and the sum in the denominator is over all spanning trees $T$.
\end{theorem}

\begin{example} \label{ex-all-minor-matrix-tree} Consider the electric circuit shown in Fig.~\ref{fig:valid-subgraphs} to the left. Introduce a synchronized source $\omega^i:=\{e_i\}$ for each $i=1,\dots,6$.
%Each of the six edges contains an individual battery, so the synchronized sources are given by $\omega^1 = \{e_1\}$ , $\omega^2 = \{e_2\}$, $\omega^3 = \{e_3\}$, $\omega^4 = \{e_4\}$, $\omega^5 = \{e_5\}$, $\omega^6 = \{e_6\}$.
Pick a subgraph $G_1$ valid relative to both $P = \{1, 2\}$ and $Q=\{3, 4\}$ separately (Fig.~\ref{fig:valid-subgraphs}, middle). %Since each edge carries its own battery,
The Gram matrix $M$ has size $6 \times 6$:
%\mscomm{More details!}
$$M=M(G_1,\Gamma) =\begin{pmatrix}
    .& .& 0& 1&. &.\\
    .& .& 1& 0&. &.\\
    .& .& .& .&. &.\\
    .& .& .& .&. &.\\
    .& .& .&. &. &.\\
    .& .& .&. &. &.
\end{pmatrix}
\qquad \text{and} \qquad
\mathrm{mult}^{34}_{12}(G_1)=\det M^{34}_{12}=-1.
%\qquad\text{and}
$$
To compute the entries of $M$, pick a basis %for homology of the graph consisting
of two cycles $\gamma_1$ and $\gamma_2$. %To find
The entry in row one column four is %, we sum
$M_1^4=\langle \gamma_1 , \omega^1 \rangle \langle \gamma_1 , \omega^4 \rangle + \langle \gamma_2 , \omega^1 \rangle \langle \gamma_2 , \omega^4 \rangle = (+1) \cdot (+1) + 0 \cdot 0 = 1$. On the other hand, $M_1^3=0$ %the entry in row one column three is $0$ since
because neither of $\gamma_1, \gamma_2$ contains both $e_1$ and $e_3$. Analogously, $M_2^3=1$ and $M_2^4=0$.

Now, there is only one other graph $G_2$ valid relative to both $P = \{1, 2\}$ and $Q=\{3, 4\}$ separately. Computing the determinant $\det M^{34}_{12}$ for each of the two valid graphs and using notation $c_i := c(e_i)$, we get the coefficients on the right in the following equality: %\ppcomm{Perhaps it would make sense to compute and display here the $2 \times 2$ minor of $L$?}
%\mscomm{Rewrite the next two equations using notation $c_i := c(e_i)$!}
$$\det L_{12}^{34}=\frac{-c(G_1)+c(G_2)}{\sum_T c(T)}=\frac{-c_1c_2c_3c_4c_5+c_1c_2c_3c_4c_6}{\sum_T c(T)},$$
where the denominator is the sum over all spanning trees:
\begin{align*}
\sum_T c(T)&=
c_1c_2c_3
+ c_1c_2c_4
+ c_1c_2c_5
+ c_1c_2c_6
+ c_1c_3c_4
+ c_1c_3c_5
+ c_1c_4c_6
+ c_1c_5c_6
+ c_2c_3c_4
+ c_2c_3c_6
\\
&\quad
+ c_2c_4c_5
+ c_2c_5c_6
+ c_3c_4c_5
+ c_3c_4c_6
+ c_3c_5c_6
+ c_4c_5c_6.
\end{align*}

It can be checked that this indeed equals the corresponding minor of the response matrix~$L$: %Here $c_i := c(e_i)$.
%For clarity, we explicitly write the corresponding $2 \times 2$ minor of the response matrix $L$.
{%\small
\[
%\begin{aligned}
%&
L_{12}^{34}
=
\frac{1}{\sum_T c(T)}
%\\
%& \cdot
\begin{pmatrix}
-c_1c_3(
c_2c_4
+ c_2c_6
+ c_4c_6
+ c_5c_6)

&
c_1c_4(c_2c_3
+ c_2c_5
+ c_3c_5
+ c_5c_6)

\\[3pt]

c_2c_3(c_1c_4
+ c_1c_5
+ c_4c_5
+ c_5c_6)

&
-c_2c_4(c_1c_3
+ c_1c_6
+ c_3c_6
+ c_5c_6)
\end{pmatrix}.
%\\
%& /
%(
%c(e_1)c(e_2)c(e_3)
%+ c(e_1)c(e_2)c(e_4)
%+ c(e_1)c(e_2)c(e_5)
%+ c(e_1)c(e_2)c(e_6)
%+ c(e_1)c(e_3)c(e_4)
%+ c(e_1)c(e_3)c(e_5)
%\\
%&+ c(e_1)c(e_4)c(e_6)
% + c(e_1)c(e_5)c(e_6)
%+ c(e_2)c(e_3)c(e_4)
%+ c(e_2)c(e_3)c(e_6)
%+ c(e_2)c(e_4)c(e_5)
%+ c(e_2)c(e_5)c(e_6)
%\\
%&+ c(e_3)c(e_4)c(e_5)
%+ c(e_3)c(e_4)c(e_6)
%+ c(e_3)c(e_5)c(e_6)
%+ c(e_4)c(e_5)c(e_6)
%).
%\end{aligned}
\]
}

\end{example}

The multiplicities can be greater than one, already in the simplest situations; a less trivial example was provided by W.Y.~Lam et al.~\cite[Sec.~9]{Lam-25}.

\begin{example} \label{ex-two-edges}
Consider the source-synchronized network with two vertices joined by two edges $e_1$ and $e_2$, directed so that they form an oriented cycle $\gamma$, and let
$\omega^1:=\{e_1,e_2\}$ be the only synchronized source. %, so that $b=1$.
Take
$P=Q=\{1\}$. The whole network $G$ is the only valid sub-graph, the single cycle
$\Gamma=(\gamma)$ is a basis, $\langle\gamma,\omega^1\rangle=2$, and
the spanning trees are $\{e_1\}$ and $\{e_2\}$. Hence
$$M=M(G,\Gamma)=\big(4\big),\qquad \mathrm{mult}^{1}_{1}(G)=4,\qquad
\det L = L^{1}_{1}=\frac{4\,c(e_1)c(e_2)}{c(e_1)+c(e_2)}.$$
%the spanning trees being $\{e_1\}$ and $\{e_2\}$.
\end{example}

\begin{example} \label{ex-network-on-double-torus}
For the source-synchronized network $G$ shown in Fig.~\ref{fig:network-on-double-torus} to the right, we find %\mscommnew{Explain what is $\Gamma$ --- add a formula to the right from the figure.}
$$M = M(G, \Gamma) =
\begin{pmatrix}
1 & 0 & 1 & 0 \\
0 & 1 & 2 & \mscommnew{-}1 \\
1 & 2 & 5 & \mscommnew{-}2 \\
0 & \mscommnew{-}1 & \mscommnew{-}2 & 1
\end{pmatrix},
\qquad
\mathrm{mult}^{13}_{13}(G)=\det M^{13}_{13}=4,
\qquad%\text{and}
%$$
%$$
\det L_{13}^{13}=\frac{4c(G)}{\sum_T c(T)}.
%=\frac{4c(e_1)c(e_2)c(e_3)c(e_4)c(e_5)c(e_6)}{\sum_T c(T)}.
$$
\end{example}

\subsection{Cohomological networks}

There is an analogous formula for the minors of the response matrix of a cohomological network. % $X$.
%Again, let $P = (p_1, \ldots, p_k)$ and $Q = (q_1, \ldots, q_k)$ be two ordered subsets of $\{1,\ldots,b\}$,  where $b=\dim H^1(X;\mathbb{Z})$.
Again, $P,Q\subset\{1,\ldots,b\}$ are two subsets with $|P|=|Q|=k$, listed in increasing order, % as above,
where now $b=\dim H^1(S;\mathbb{Z})$.
A subgraph $G$ %containing all the vertices of the cohomological network
is \emph{valid} with respect to $P$ if it \mscommnew{is connected,} contains all vertices, has $n+|P|-1$ edges, and the elements $\imath^*\omega^{p_1},\dots,\imath^*\omega^{p_k}\in H^1(G;\mathbb{Z})\cong \mathbb{Z}^k$ are linearly independent.
For a subgraph $G$ valid with respect to both $P$ and $Q$ separately, take a basis $\Gamma =(\gamma_1,\dots,\gamma_k)$ of $H_1(G;\mathbb{Z})$ and define the %$b\times b$
\emph{Gram matrix} $M=M(G,\Gamma)$ and the \emph{multiplicity} of $G$ in %the minor
$L^Q_P$ as
$$
M_i^j:=\langle\imath^*\omega^i,\imath^*\omega^j\rangle_{G,\Gamma} =
%\langle\i^*\mathcal{S}_i,\left(\i^*\mathcal{S}_j\right)^\mathrm{T} \rangle  =
\sum_{l=1}^k
\langle\omega^i,\gamma_l\rangle
\langle\omega^j,\gamma_l\rangle
\qquad\text{and}\qquad
\mathrm{mult}_P^Q(G):=\det M^Q_P.
$$
Clearly, $\mathrm{mult}^Q_P(G)$ does not depend on the choice of basis
$\Gamma$ but may depend on the choice of basis $(\omega^{1},\dots,\omega^{b})$ of $H^1(S;\mathbb{Z})$.
The following corollary of Theorem~\ref{th-all-minor-matrix-tree} is proved in Sec.~\ref{sec-conclusion}.

\begin{corollary} \label{cor-all-minor-matrix-tree}
    % If %$P, Q \in \{1,\ldots,b\}$ and
    % $P \cap Q=\emptyset$, then t
    The minor of the response matrix of a cohomological network is given by
    $$\det L_P^Q = \frac{\sum_G \mathrm{mult}_P^Q(G) \cdot c(G)}{\sum_T c(T)},$$
    where the sum in the numerator is over all subgraphs $G$ that are valid relative to both $P$ and $Q$ separately, and the sum in the denominator is over all spanning trees $T$.
\end{corollary}

\subsection{Overview of the proof} \label{sec:over}
For the proof of Theorem~\ref{th-all-minor-matrix-tree}, %a direct application of
%Kirchhoff's formula for $L^j_i$
the formula from Proposition~\ref{prop-response} is helpless due to huge cancellations. Therefore, we employ tools from statistical physics:
%For the proof of Theorem~\ref{th-all-minor-matrix-tree}, we use tools from statistical physics:
\begin{itemize}
    \item a version of \emph{Curtis--Morrow's connection decomposition} for response matrix minors;
    \item a version of \emph{Smirnov's parafermionic observable} for the uniform spanning tree model;
    \item modified \emph{Temperley's correspondence} between spanning trees and dimer configurations.
\end{itemize}
The parafermionic observable encodes cancellations, while the Temperley correspondence translates the problem into dimer configurations where these cancellations can be analyzed explicitly.

The proof proceeds by rewriting the minor $\det L_P^Q$ in several equivalent forms %\mscommnew{remove hat?}
\[
\begin{aligned} \det L_P^Q  \cdot \sum_T c(T)&\overset{(1)}{=}
\sum_{\alpha} \mathrm{sgn}_P^Q(\alpha)c(\overline\alpha)\det K(S\setminus\overline\alpha)_{ \mathrm{V}\setminus \alpha}^{ \mathrm{V}\setminus \alpha} \\
&\overset{(2)}{=}
{F_P^Q} \\
% &=
%  \sum_{A,B}\sum_{H\text{ \textrm{valid}}} \mathrm{sgn}_P^Q (\alpha(H)) c(H)/{\det K^\mathrm{V}_\mathrm{V}}\\
&\overset{(3)}{=}
\sum_{A, B: a_i \in \omega^{p_i}, b_i \in \omega^{q_i}} \sum_{G} \mathrm{sgn}_P^Q(\alpha(\vec D)) c(G) \\
% &=\sum_G \left( \sum_A \mathrm{sgn}_P(A, \Gamma) \right) \left( \sum_B \mathrm{sgn}^Q(B, \Gamma) \right)/ {\det K^\mathrm{V}_\mathrm{V}} \\
&\overset{(4)}{=} \sum_G \mathrm{mult}_P^Q(G) \cdot c(G).
\end{aligned}
\]

In the first equality, we express the minor as a sum over all
\emph{connections} $\alpha$ between $P$ and $Q$ whose paths do not
pass through the last vertex \(v_n\), as introduced in
Sec.~\ref{sec-connection-decomposition}. Each summand %connection contributes with a \emph{sign} $\mathrm{sgn}_P^Q(\alpha)$, and the contribution
%The contribution of each connection
is the \emph{conductance} $c(\overline\alpha)$ with a
\emph{sign} $\mathrm{sgn}_P^Q(\alpha)$, multiplied by a minor of the augmented Kirchhoff matrix
$K(S\setminus\overline\alpha)$ %\mscommnew{remove hat?}
of a certain sub-network of our network $S$.
This identity is  a version of the \emph{Curtis--Morrow connection decomposition}
(Lemma~\ref{l-connections}).

In the second equality, $F_P^Q$ denotes the \emph{parafermionic observable} for the uniform spanning-tree model defined in Sec.~\ref{sec-parafermionic}. This identity is proved there in Lemma~\ref{connection_F_and_L}.

% where \(A\) and \(B\) are ordered sets of midpoints such that \(|A|=|B|=|P|=|Q|\), and \(a_i\) (respectively, \(b_i\)) denotes the edge whose midpoint is the \(i\)-th element of \(A\) (respectively, \(B\))
In the third equality, we rewrite the parafermionic observable as a sum over two ordered sets of edge midpoints $A$ and $B$ such that $|A|=|B|=|P|=|Q|$. Here, $a_i$ (respectively, $b_i$) denotes the edge whose midpoint is the $i$-th element of $A$ (respectively, $B$). The inner sum is taken over subgraphs $G$ that are valid relative to both $A$ and $B$.
Informally, $G$ can be viewed as the subgraph that supports the paths connecting %the sources in
$A$ and %to the sinks in
$B$. To the subgraph $G$, we assign a canonical \emph{double-dimer configuration} $\vec D$ on a subdivision of $G$.
%The object $\vec D$ denotes a \emph{double-dimer configuration} corresponding to $G$, obtained as the superposition of two oriented dimer configurations on a subdivision of $G$.
Such a configuration naturally encodes a connection %collection
$\alpha(\vec D)$ %of oriented paths
between $A$ and $B$. %, which we denote by $\alpha(\vec D)$.
The identity then follows from \emph{modified Temperley's correspondence} (Lemma~\ref{lemma:bijection}), established in Sec.~\ref{sec-Temperley}.

The final equality is the computation of the sign $\mathrm{sgn}_P^Q(\alpha(\vec D))$ performed in Sec.~\ref{sec-conclusion}.

% Here $\hat{K}(G\setminus\partial\alpha)$ is discussed in Section~\ref{sec-connection-decomposition}; $F_P^Q$ denotes a version of \emph{Smirnov's parafermionic observable} defined in Section~\ref{sec-parafermionic}.

\section{Connection decomposition}
\label{sec-connection-decomposition}

For the proof of Theorem~\ref{th-all-minor-matrix-tree}, we are going to derive several auxiliary results of independent interest. In this section, we construct a %n analog of
connection decomposition for the response-matrix minors of a source-synchronized network $S$. %Throughout this subsection,
We use the notation from Secs.~\ref{sec-networks}--\ref{sec-all-minor}.

First, we need to recall some standard notions.
By an (oriented) \emph{path} in the network, we mean a finite sequence $u_1,\vec e_1,u_2,\dots,\vec e_\ell,u_{\ell+1}$, where $u_1,\dots,u_{\ell+1}$ are vertices, $\vec e_1,\dots,\vec e_{\ell}$ are edges possibly with reversed direction, and $\vec e_j$ starts at $u_j$ and ends at $u_{j+1}$ for each $j=1,\dots,\ell$.

The \emph{subdivision of an edge $xy$} %of a network
means adding a new vertex $z$ (called the \emph{midpoint} of $xy$) and replacing $xy$ with two edges $xz$ and $zy$. %, each with the same conductance as the edge~$xy$. Note that this convention on the conductance is different from the one for the refined circuits in Secs~\ref{ssec-networks on surfaces}--\ref{ssec-cohomological networks}.]}
%notion of subdivision is different from the operation used in Section~\ref{ssec-networks on surfaces}, where edges were split to obtain a refined circuit.
% \mscomm{This is different from the subdivision in Section~\ref{ssec-networks on surfaces}!} %For convenience, we will continue to refer to the new vertex $z$ as a \emph{midpoint}, even though it is now a vertex in the subdivided graph.
%
The \emph{suppression} of a midpoint \( z \) is the inverse operation, that is, replacing the edges \( xz \) and \( zy \) %of the same conductance
with a single edge \( xy \) %of the same conductance
and removing the vertex~\( z \).

For a subgraph $G$, %of the network $S$,
denote by $G'$ the graph obtained from $G$ by subdivision of all edges.
Henceforth, for %the sake of
brevity, we refer to the vertices of the graph \( G \) simply as \emph{vertices}, and the remaining vertices of the graph \( G' \) are called \emph{midpoints}. The edges of~\( G' \) are called \emph{half-edges}.
%The term \emph{half-edges} refers to the edges of the graph \( G' \).

Let $A$ and $B$ be two ordered sets of midpoints such that $|A| = |B| =k$.
%Each midpoint lies on a unique edge of $G$, so we identify a set of midpoints with the corresponding set of edges. In particular, $G\setminus A$ denotes the graph obtained from $G$ by deleting the edges in $A$, and $G'\setminus A=(G\setminus A)'$.
Denote by $G\setminus A$ and $G'\setminus A$
the graphs obtained from $G$ and $G'$ by removing $A$ and all the edges and half-edges, respectively, having common points with $A$.
Denote by \( G_A^B \) the graph obtained from \( G \) by subdivision of the edges with midpoints in \( A \cup B \).

\mscommnew{For a subgraph $H\subset G_A^B$ containing $A\cup B$, its \emph{completion} is the subgraph $\overline H\subset G$ obtained by adding all missing half-edges %incident to
emanating from $A\cup B$ %, together with their endpoints,
and suppressing the midpoints in $A\cup B$.}

A \textit{connection} from $A$ to $B$ in $G$ is a collection $ \alpha = \{\alpha_1, \ldots, \alpha_{|A|} \}$ of disjoint simple oriented paths in $G_A^B$, each starting in $A$ and ending in $B$.
%\mscommnew{[Removed: and not passing through the last vertex $v_n$]}
Vertices in $A \cap B$ appear as trivial one-vertex paths. (A connection may not ``respect'' the order of midpoints in $A$ and~$B$.) See Fig.~\ref{fig:auxiliary-graphs}.

% Now we introduce a new notion of the \emph{sign} $\mathrm{sgn}_P^Q(\alpha)$ of a connection $\alpha$.

Now we introduce a new notion of the \emph{sign} $\mathrm{sgn}_A^B(\alpha)$ of a connection $\alpha$.

Let \( (a_1, \ldots, a_{|A|}) \) be the edges with midpoints in \( A \), and let \( (b_1, \ldots, b_{|B|}) \) be the edges with midpoints in \( B \). Define \( \tau \in \mathcal{S}_{|A|} \) as the unique permutation such that, for each \( i =1,\dots,|A|\), there is a path in the connection \( \alpha \) connecting the midpoint of \( a_i \) to the midpoint of \( b_{\tau(i)} \). Denote by $\vec{a_i}$ and $\vec{b}_{i}$ the edges $a_i$ and $b_{i}$ oriented consistently with their paths in \(\alpha\); if \(a_i=b_{\tau(i)}\), orient \(\vec a_i\) and \(\vec b_{\tau(i)}\) in the same direction as \(a_i\).
Define the \emph{sign} of the connection $\alpha$ by the formula
\begin{equation}\label{eq-connection-sign}
    \mathrm{sgn}_A^B(\alpha):=\mathrm{sgn}(\tau)\prod_{i=1}^{|A|}\langle \vec{a}_i, a_i \rangle \langle \vec{b_i}, b_i\rangle,
\end{equation}
where $\langle \vec{e}, e \rangle=+1$ if the directions of edges $\vec{e}$ and $e$ agree, and $\langle \vec{e}, e \rangle=-1$, otherwise; see~\eqref{eq-bracket-cycle-edge}.

An ordered set of midpoints \(A\) is \emph{consistent} with \(P\) if $|A|=|P|$ and for each \(i=1,\dots,|P|\), letting \(a_i\) denote the edge whose midpoint is the \(i\)-th element of \(A\), we have
$
a_i \in \omega^{p_i}.
$

Whenever \(A\) and \(B\) are consistent with \(P\) and \(Q\), respectively, we say that a connection from \(A\) to \(B\) is a \emph{connection} from \(P\) to \(Q\). We define
$
\mathrm{sgn}_P^Q(\alpha):=\mathrm{sgn}_A^B(\alpha).
$
% Denote by $\mathcal{C}_P^Q(S)$ the set of all connections in $S$ from $P$ to $Q$
% whose paths do not pass through the last vertex $v_n$.

For a connection \( \alpha\), %\( \alpha \in \mathcal{C}(G) \),
denote (see Fig.~\ref{fig:auxiliary-graphs}):
\begin{itemize}
    %\item $c(\alpha)$ is the product of the conductances of all edges and half-edges in $\alpha$; a trivial path contributes $1$.
    %in the connection \(\alpha\), including the ones containing the endpoints of paths in \(\alpha\) (even if the endpoints of a path coincide). %
    \item $S\setminus\overline\alpha$ is the network obtained from $S$ by deleting all edges of the completion $\overline\alpha$.
    \item $K(S\setminus\overline\alpha)$ is the augmented Kirchhoff
    matrix of the network $S\setminus\overline\alpha$.
    %obtained by deleting all edges of $\overline\alpha$.
    %        that contain the endpoints of the paths of \( \alpha \).
    %See Fig.~\ref{fig:auxiliary-graphs}.
    \item \(\mathrm{ V}\setminus \alpha\) is the set of vertices of $S$ not contained in \(\alpha\) and distinct from the last vertex $v_n$.
    %subset \(\mathrm{V}\) without the vertices contained in \(\alpha\).
\end{itemize}

\begin{figure}[ht]
\centering
\scalebox{0.85}{%
\begin{circuitikz}[x=1cm,y=1cm,
    selected/.style={circle,fill=black,draw=none,inner sep=0pt,minimum size=8pt},
    base/.style={black,line width=0.7pt},
    directed/.style={base,-{Latex[length=1.7mm]},shorten >=2pt},
    connection/.style={purple,line width=1.5pt,-{Latex[length=2mm]}}]

\begin{scope}

\coordinate (U) at (0,2.5);
\coordinate (R) at (2.4,0.8);
\coordinate (N) at (1.5,-2);
\coordinate (Vfour) at (-1.5,-2);
\coordinate (Vfive) at (-2.4,0.8);
\coordinate (Vsix) at (2.4,2.5);
\coordinate (Vseven) at (1.2,3.5);
\coordinate (C) at (0,0);
\coordinate (Atwo) at (1.2,0.4);
\coordinate (Bone) at (0.6,3);
\coordinate (Btwo) at (-0.75,-1);
\coordinate (Aone) at (-1.2,0.4);

\draw[directed] (U) -- (R);
\draw[directed] (R) -- (Vsix);
\draw[directed] (Vsix) -- (U);
\draw[directed] (Vsix) -- (Vseven);
\draw[directed] (R) -- (N);
\draw[directed] (N) -- (Vfour);
\draw[directed] (Vfour) -- (Vfive);
\draw[directed] (Vfive) -- (U);
\draw[directed] (C) -- (U);
\draw[directed] (C) -- (N);
\draw[base] (U) -- (Vseven);
\draw[base] (R) -- (C);
\draw[base] (C) -- (Vfour);
\draw[base] (C) -- (Vfive);

\draw[directed] (Atwo)--(C);
\draw[directed] (Bone)--(Vseven);
\draw[directed] (Btwo)--(Vfour);
\draw[directed] (Aone)--(Vfive);
\draw[connection] (Atwo)--(R);
\draw[connection] (R)--(U);
\draw[connection] (U)--(Bone);
\draw[connection] (Aone)--(C);
\draw[connection] (C)--(Btwo);

\fill (U) circle (2pt);
\fill (R) circle (2pt);
\fill (N) circle (2pt);
\fill (Vfour) circle (2pt);
\fill (Vfive) circle (2pt);
\fill (Vsix) circle (2pt);
\fill (Vseven) circle (2pt);
\fill (C) circle (2pt);

\foreach \p in {Aone,Atwo} \fill[red] (\p) circle (2.5pt);
\foreach \p in {Bone,Btwo} \fill[blue] (\p) circle (2.5pt);
\node[above=3pt,red] at (Aone) {$A_1$};
\node[below=3pt,red] at (Aone) {$-$};
\node[below right=2pt,red] at (Atwo) {$A_2$};
\node[above left=2pt,red] at (Atwo) {$-$};
\node[above left=2pt,blue] at (Bone) {$B_1$};
\node[right=3pt,blue] at (Bone) {$+$};
\node[below right=2pt,blue] at (Btwo) {$B_2$};
\node[above left=2pt,blue] at (Btwo) {$+$};
\node[right=4pt] at (N) {$v_n$};
\node[purple] at (1.3,2.05) {$\alpha$};
\node at (0,-2.95) {$S$};
\node at (0,-3.8) {$\mathrm{sgn}_A^B(\alpha)=-1$};
\end{scope}

\begin{scope}[shift={(6.1,0)}]

\coordinate (U) at (0,2.5);
\coordinate (R) at (2.4,0.8);
\coordinate (N) at (1.5,-2);
\coordinate (Vfour) at (-1.5,-2);
\coordinate (Vfive) at (-2.4,0.8);
\coordinate (Vsix) at (2.4,2.5);
\coordinate (Vseven) at (1.2,3.5);
\coordinate (C) at (0,0);
\coordinate (Atwo) at (1.2,0.4);
\coordinate (Bone) at (0.6,3);
\coordinate (Btwo) at (-0.75,-1);
\coordinate (Aone) at (-1.2,0.4);

\draw[base] (R) -- (Vsix);
\draw[base] (Vsix) -- (U);
\draw[base] (Vsix) -- (Vseven);
\draw[base] (R) -- (N);
\draw[base] (N) -- (Vfour);
\draw[base] (Vfour) -- (Vfive);
\draw[base] (Vfive) -- (U);
\draw[base] (C) -- (U);
\draw[base] (C) -- (N);
\fill (U) circle (2pt);
\fill (R) circle (2pt);
\fill (N) circle (2pt);
\fill (Vfour) circle (2pt);
\fill (Vfive) circle (2pt);
\fill (Vsix) circle (2pt);
\fill (Vseven) circle (2pt);
\fill (C) circle (2pt);

\foreach \p in {Vfour,Vfive,Vsix,Vseven}
    \node[selected] at (\p) {};

\node[right=4pt] at (N) {$v_n$};
\node at (0,-2.95) {$S\setminus\overline\alpha$};
\node[selected] at (-0.7,-3.8) {};
\node[right=5pt] at (-0.7,-3.8) {$\mathrm V\setminus\alpha$};
\end{scope}

\begin{scope}[shift={(12.2,0)}]

\coordinate (U) at (0,2.5);
\coordinate (R) at (2.4,0.8);
\coordinate (N) at (1.5,-2);
\coordinate (Vfour) at (-1.5,-2);
\coordinate (Vfive) at (-2.4,0.8);
\coordinate (Vsix) at (2.4,2.5);
\coordinate (Vseven) at (1.2,3.5);
\coordinate (C) at (0,0);
\coordinate (Atwo) at (1.2,0.4);
\coordinate (Bone) at (0.6,3);
\coordinate (Btwo) at (-0.75,-1);
\coordinate (Aone) at (-1.2,0.4);

\draw[base] (N) -- (Vfive);
\draw[base] (Vfour) -- (Vfive);
\draw[base] (N) -- (Vfour);
\draw[base] (Vsix) -- (Vseven);

\draw[base,bend left=30] (N) to (Vsix);
\draw[base,bend right=30] (N) to (Vsix);

\fill (N) circle (2pt);
\fill (Vfour) circle (2pt);
\fill (Vfive) circle (2pt);
\fill (Vsix) circle (2pt);
\fill (Vseven) circle (2pt);

\foreach \p in {Vfour,Vfive,Vsix,Vseven}
    \node[selected] at (\p) {};

\node[right=4pt] at (N) {$[v_n]$};
\node at (0,-2.95) {$S^\alpha=(S\setminus\overline\alpha)/(\alpha\cup\{v_n\})$};
\node[selected] at (-0.7,-3.8) {};
\node[right=5pt] at (-0.7,-3.8) {$\mathrm V\setminus\alpha$};
\end{scope}
\end{circuitikz}
}
\caption{
Left: A network $S$ (black) and a connection $\alpha$ (purple)
from $A=(A_1,A_2)$ to $B=(B_1,B_2)$.
Black and purple arrows indicate edge and path directions,
respectively; the signs at the midpoints indicate whether these
directions agree.
\mscommnew{Those signs multiply to $+1$,} while the endpoint permutation induced by $\alpha$ has sign $-1$,
%The endpoint permutation has sign $-1$, while the orientation factors have product $+1$,
so $\mathrm{sgn}_A^B(\alpha)=-1$.
Middle: $S\setminus\overline\alpha$, obtained by deleting all edges
of $\overline\alpha$; larger dots mark $\mathrm V\setminus\alpha$.
Right: The auxiliary graph $S^\alpha$.
See Sec.~\ref{sec-connection-decomposition}.}
 \label{fig:auxiliary-graphs}
\end{figure}
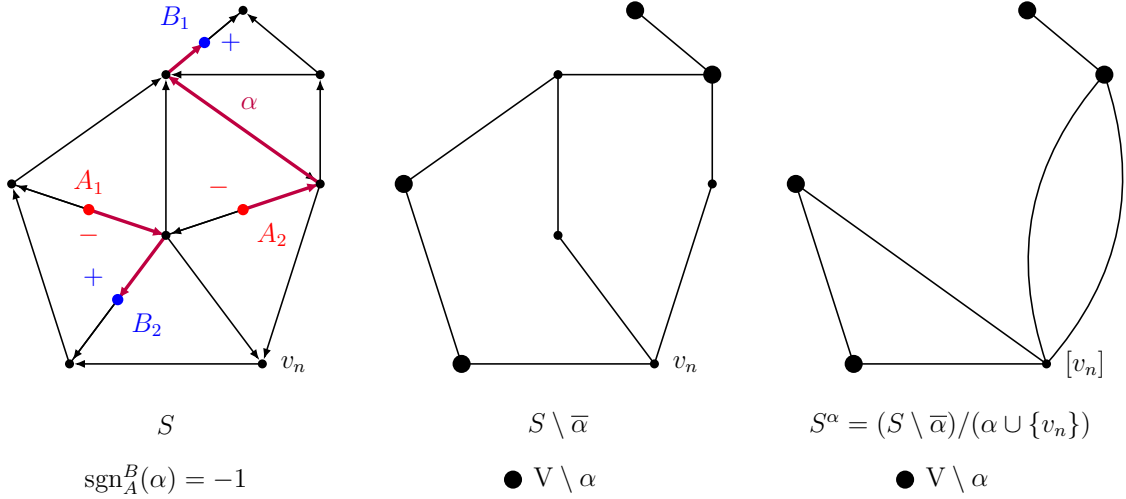

\begin{lemma}[Connections decomposition] \label{l-connections}
%Consider a source-synchronized network $S$, the response matrix \( L \), the Kirchhoff matrix \( \hat K \), and ordered subsets $P, Q \subset \{1,\ldots,b\}$. Then
For any source-synchronized network $S$ and subsets $P,Q\subset\{1,\ldots,b\}$ with $|P|=|Q|$, we have

\begin{equation}\label{eq-connections}
\det L_P^Q \cdot \det K^{ \mathrm{V}}_{ \mathrm{V}} =
\sum_{
\alpha}\mathrm{sgn}_P^Q(\alpha)c(\overline\alpha)
\det K(S\setminus\overline\alpha)_{\mathrm V\setminus\alpha}^{\mathrm V\setminus\alpha},
\end{equation}
where the sum is over all connections $\alpha$ in $S$ from $P$ to $Q$ whose paths do not pass through the last vertex $v_n$.
\end{lemma}

%Now we are ready to prove Lemma \ref{l-connections}.
The proof %closely resembles that of
follows \cite[Lemma 3.12]{CurtisMorrow}, but has an additional %combinatorial
cancellation step.

\begin{proof} %\mscommnew{[A simplified proof. Pasha, will you check?]}
The lemma follows from the following chain of equalities to be explained below:
\begin{equation*}%\label{eq-connections-beta}
\det L_P^Q \det K^{\mathrm{V}}_{\mathrm{V}}
\overset{(1)}{=} \det K^{Q \cup  \mathrm{V}}_{P \cup  \mathrm{V}}
\overset{(2)}{=} \sum_{e_{\rho^1},\dots,e_{\rho^m}}\det K(e_{\rho^1},\dots,e_{\rho^m})
\overset{(3)}{=}\sum_\alpha\mathrm{sgn}_P^Q(\alpha)c(\overline\alpha)
\det K(S\setminus\overline\alpha)_{\mathrm V\setminus\alpha}^{\mathrm V\setminus\alpha}.
\end{equation*}

Let us introduce the notation and the proof for each equality (1)--(3).

In the first equality, \( K^{Q \cup \mathrm{V}}_{P \cup \mathrm{V}} \) denotes the minor of \( K \) formed by %selecting
the rows labeled by sources \( \omega^p \) for \( p \in P \) and by trivial sources \( \delta^j \) for \( j \in \mathrm{V} \), and similarly for the columns. %Recall that the elements of $P$ and $Q$ are listed in increasing order (Section~\ref{sec-all-minor}), so that $K^{Q\cup\mathrm{V}}_{P\cup\mathrm{V}}$ is indeed a minor of~$K$.
Equality~(1) then follows from Lemma~\ref{schur_complement} and Schur's formula. %the property of Schur's complement.

In the second equality, the sum is over all collections of distinct edges $e_{\rho^1},\dots,e_{\rho^m}$, indexed by sources \( \omega^q \) for \( q \in Q \) and by trivial sources \( \delta^j \) for \( j \in \mathrm{V} \),
so that $m$ is the number of elements in $Q \cup \mathrm{V}$,
and $K(e_{\rho^1},\dots,e_{\rho^m})$ is the $m\times m$ matrix with the entries
$$
K(e_{\rho^1},\dots,e_{\rho^m})^{\rho^i}_\nu:=\langle {\rho^i},e_{\rho^i}\rangle \langle \nu,e_{\rho^i}\rangle c(e_{\rho^i}).
$$
If two of the edges $e_{\rho^1},\dots,e_{\rho^m}$ coincide, then this matrix has two proportional columns. Thus, equality~(2) follows from the multi-linearity of the determinant.

In the third equality, the sum is over all connections $\alpha$ between $P$ and $Q$, whose paths do not pass through the last vertex $v_n$. They arise naturally in the expansion of the determinant $\det K(e_{\rho^1},\dots,e_{\rho^m})$. Indeed, as we expand the determinant, we pick entries forming a generalized diagonal. Start with the row of the matrix $K(e_{\rho^1},\dots,e_{\rho^m})$ labeled by a synchronized source $\omega^p$ for some $p \in P$.
Pick an entry in this row: $$K(e_{\rho^1},\dots,e_{\rho^m})_{\omega^p}^{\rho^i}=\langle\omega^p,e_{\rho^i}\rangle\langle{\rho^i},e_{\rho^i}\rangle c(e_{\rho^i}).$$ This entry is nonzero only if the edge $e_0:=e_{\rho^i}$ belongs to $\omega^p$ and $\langle\rho^i,e_0\rangle\ne 0$.
Let us grow a path from the midpoint of $e_0$.
If $\rho^i=\omega^q$ for some $q\in Q$, so that $e_0\in \omega^q$ (this is possible if $p=q\in P\cap Q$), then the path immediately terminates at the midpoint of $e_0$.
Otherwise, $\rho^i=\delta^j$ is some trivial source, minus coboundary of a vertex $v_j$. Set $v_j$ to be the next vertex of our growing path. Continue with the row labeled by $\delta^j$, and pick an entry in that row: % that contributes to the determinant:
$$
K(e_{\rho^1},\dots,e_{\rho^m})_{\delta^j}^{\rho^l}=\langle\delta^j, e_{\rho^l}\rangle \langle {\rho^l},e_{\rho^l}\rangle c(e_{\rho^l}).
$$
The entry is nonzero only if the edge $e_1:=e_{\rho^l}$ is contained in $\delta^j$ and the next (possibly trivial) source $\rho^l$, etc. This will continue as long as we get columns labeled by trivial sources $\rho^l$, and it will stop as soon as we hit a nontrivial source $\rho^l=\omega^{q}$ for $q\in Q$.
This builds a path of a connection $\alpha$. The non-repetition of vertices follows from the generalized-diagonal condition: each row/column label is used at most once.  The non-repetition of edges follows because $e_{\rho^1},\dots,e_{\rho^m}$ are all distinct (otherwise, a path starting at the midpoint of $e_0$ could make a ``U-turn'' at the vertex $v_j$ and traverse the same edge $e_0$ again).

Let us compute the product of the selected entries. Denote by $\delta^{j_1},\dots,\delta^{j_\ell}$ and $e_0,e_1,\dots,e_\ell$ the trivial sources and edges of the path, respectively.
Let $\vec e_k$ denote
$e_k$ oriented along the path. Here, if $\ell=0$,
then the path has no intermediate trivial sources and consists of a single midpoint,
%there are no intermediate trivial sources and the path %consists of one edge $e_0=e_\ell$,
and we set $\vec e_0:=e_0$ for definiteness. The selected entries contribute the product %of pairings
$$
\langle\omega^{p},e_0\rangle\,
\langle\delta^{j_1},e_0\rangle\langle\delta^{j_1},e_1\rangle\,
\langle\delta^{j_2},e_1\rangle\langle\delta^{j_2},e_2\rangle\cdots
\langle\delta^{j_\ell},e_{\ell-1}\rangle\langle\delta^{j_\ell},e_\ell\rangle\,
\langle\omega^{q},e_\ell\rangle c(e_0)\dots c(e_\ell).
$$
At an intermediate vertex $v_{j_k}$, the path enters along $e_{k-1}$ and leaves along
$e_k$, so that %; the definition of $\delta^{j_k}$ then gives
$$
\langle\delta^{j_k},e_{k-1}\rangle=-\langle\vec e_{k-1},e_{k-1}\rangle,
\qquad
\langle\delta^{j_k},e_k\rangle=\langle\vec e_k,e_k\rangle,
$$
so $v_{j_k}$ contributes $-\langle\vec e_{k-1},e_{k-1}\rangle\langle\vec e_k,e_k\rangle$.
Since $\langle\vec e_k,e_k\rangle^2=1$ for $1\le k\le\ell-1$, in the product over
$k=1,\dots,\ell$, middle factors cancel out, and the whole pairing product equals
$$
(-1)^{\ell}\,\langle\omega^{p},e_0\rangle\,\langle\vec e_0,e_0\rangle\,
\langle\vec e_\ell,e_\ell\rangle\,\langle\omega^{q},e_\ell\rangle
=(-1)^{\ell}\,\langle\vec e_0,e_0\rangle\,\langle\vec e_\ell,e_\ell\rangle.
%=(-1)^{\ell}\,\langle\vec a,a\rangle\langle\vec b,b\rangle,
$$
%where $a=e_0$ and $b=e_\ell$ are the first and last edges of the path and
%$\langle\omega^{p},e_0\rangle=\langle\omega^{q},e_\ell\rangle=1$.
So, the product of the selected entries for all the paths is
$$
\prod_i(-1)^{\ell_i}\langle\vec a_i,a_i\rangle\langle\vec b_i,b_i\rangle c(\overline\alpha),
$$
where $\ell_i$ is the number of vertices on the $i$-th path, and $a_i$ and $b_i$ are its first and last edges.

Let us compute the sign of this product in the determinant.
%On the other hand,
The chosen generalized diagonal is a permutation $\sigma$ of
the index set $P\cup \mathrm{V}$ after identifying the $i$-th source row $\omega^{p_i}$
with the $i$-th sink column $\omega^{q_i}$ (and each $\delta^j$ with itself).
Let $\mathrm{V}\cap\alpha$ be the set of vertices used by the paths of $\alpha$. By construction,
$\sigma(P\cup (\mathrm{V}\cap\alpha))\subset P\cup (\mathrm{V}\cap\alpha)$, hence
$\sigma(\mathrm{V}\setminus\alpha)=\mathrm{V}\setminus\alpha$ and
$$
\operatorname{sgn}(\sigma)
=\operatorname{sgn}(\sigma|_{P\cup (\mathrm{V}\cap\alpha)})
 \operatorname{sgn}(\sigma|_{\mathrm{V}\setminus \alpha}).
$$
Summed over all permutations of $\mathrm{V}\setminus \alpha$, the second factor, together
with the corresponding entries, gives %the minor
$\det K(e_{\rho^1},\dots,e_{\rho^m})^{\mathrm{V}\setminus \alpha}_{\mathrm{V}\setminus \alpha}$. For the first factor, the cycles of
$\sigma|_{P\cup (\mathrm{V}\cap\alpha)}$ are obtained by gluing the paths of $\alpha$ according to
the source-to-sink permutation $\tau$: if $i_1\mapsto\cdots\mapsto i_r\mapsto i_1$ is a cycle of
$\tau$, the corresponding $r$ paths form a single cycle in $\sigma$ of length
$r+\sum_s\ell_{i_s}$ and sign $(-1)^{r-1}(-1)^{\sum_s\ell_{i_s}}$.
%where $\ell_i$ denotes the number of vertices on the $i$-th path.
Then
%$$
\begin{equation*}%\label{eq-tmp}
\operatorname{sgn}(\sigma|_{P\cup (\mathrm{V}\cap\alpha)})
=\operatorname{sgn}(\tau)\prod_i(-1)^{\ell_i}.
\end{equation*}
The factors $(-1)^{\ell_i}$ cancel the factors from the telescoping computation
above, so the paths of $\alpha$ contribute the sign
$$
\operatorname{sgn}(\tau)\prod_i\langle\vec a_i,a_i\rangle\langle\vec b_i,b_i\rangle
=\mathrm{sgn}_P^Q(\alpha)
$$
by%the definition of the sign of $\alpha$. %
~\eqref{eq-connection-sign}. %, where $a_i$ is the first edge of $i$-th path, and $b_i$ is its last edge.
Thus, the generalized diagonals yielding the same connection $\alpha$ contribute %the value
$$\mathrm{sgn}_P^Q(\alpha)\,c(\overline\alpha)\det K(e_{\rho^1},\dots,e_{\rho^m})^{\mathrm{V}\setminus \alpha}_{\mathrm{V}\setminus \alpha}.
$$

The connection $\alpha$ determines the edges of $\overline\alpha$ among $e_{\rho^1},\dots,e_{\rho^m}$. We still need to sum over the remaining edges. By multi-linearity, the resulting sum of minors $\det K(e_{\rho^1},\dots,e_{\rho^m})^{\mathrm{V}\setminus \alpha}_{\mathrm{V}\setminus \alpha}$
gives exactly $\det K(S\setminus\overline\alpha)^{\mathrm V\setminus\alpha}_{\mathrm V\setminus\alpha}$, because
%none of the remaining edges can contain the points of $\partial \alpha$, the edges in $\alpha$ do not contribute to those minors anyway, and allowing repeated edges does not change the sum, since
any repeated-edge term has two proportional columns and vanishes. This proves~(3).
\end{proof}

\section{Parafermionic observable}
\label{sec-parafermionic}

In this section, we continue the proof of Theorem~\ref{th-all-minor-matrix-tree} and introduce a new key tool: a parafermionic observable for the (weighted) uniform spanning-tree model. %\mscomm{Illustrate with a figure!}

\begin{figure}[ht]
\centering
\begin{circuitikz}[
    x=1cm,y=1cm,
    forest/.style={purple,line width=1.5pt},
    context/.style={black,line width=0.8pt}]
\begin{scope}[shift={(0,3.5)}]

\coordinate (N) at (0,2);
\coordinate (L) at (0,0);
\coordinate (TL) at (2,2);
\coordinate (J) at (2,0);
\coordinate (TR) at (4,2);
\coordinate (BR) at (4,0);
\coordinate (Aone) at (2,1);
\coordinate (Btwo) at (1,2);
\coordinate (AB) at (3,0);
\draw[black,line width=0.8pt,-{Latex[length=2mm]}] (L) -- (N);
\draw[black,line width=0.8pt,-{Latex[length=2mm]}] (N) -- (TL);
\draw[black,line width=0.8pt,-{Latex[length=2mm]}] (L) -- (J);
\draw[black,line width=0.8pt,-{Latex[length=2mm]}] (J) -- (TL);
\draw[black,line width=0.8pt,-{Latex[length=2mm]}] (J) -- (BR);
\draw[black,line width=0.8pt,-{Latex[length=2mm]}] (BR) -- (TR);
\draw[black,line width=0.8pt,-{Latex[length=2mm]}] (TL) -- (TR);
\foreach \p in {N,L,TL,J,TR,BR} \fill[black] (\p) circle (2pt);\node[left=4pt] at (N) {$v_n$};
\node at (2,-0.4) {$G$};
\end{scope}

\begin{scope}[shift={(6,3.5)}]

\coordinate (N) at (0,2);
\coordinate (L) at (0,0);
\coordinate (TL) at (2,2);
\coordinate (J) at (2,0);
\coordinate (TR) at (4,2);
\coordinate (BR) at (4,0);
\coordinate (Aone) at (2,1);
\coordinate (Btwo) at (1,2);
\coordinate (AB) at (3,0);
\draw[context] (L) -- (N);
\draw[context] (N) -- (TL);
\draw[context] (L) -- (J);
\draw[context] (J) -- (TL);
\draw[context] (J) -- (BR);
\draw[context] (BR) -- (TR);
\draw[context] (TL) -- (TR);
\draw[forest] (N)--(L)--(J);
\draw[forest] (Aone)--(TL)--(Btwo);
\draw[forest] (TL)--(TR)--(BR);
\foreach \p in {N,L,TL,J,TR,BR} \fill[black] (\p) circle (2pt);
\fill[red] (Aone) circle (2.5pt);
\node[right=3pt,red] at (Aone) {$A_1$};
\fill[blue] (Btwo) circle (2.5pt);
\node[above=3pt,blue] at (Btwo) {$B_2$};
\begin{scope}[shift={(AB)}]
\fill[red] (0,0) -- (90:2.5pt) arc (90:270:2.5pt) -- cycle;
\fill[blue] (0,0) -- (-90:2.5pt) arc (-90:90:2.5pt) -- cycle;
\end{scope}
\node[above=4pt] at (AB) {$\textcolor{red}{A_2}=\textcolor{blue}{B_1}$};
\node[left=4pt] at (N) {$v_n$};
\node at (2,-0.4) {$H$};
\end{scope}

\begin{scope}[shift={(0,0)}]

\coordinate (N) at (0,2);
\coordinate (L) at (0,0);
\coordinate (TL) at (2,2);
\coordinate (J) at (2,0);
\coordinate (TR) at (4,2);
\coordinate (BR) at (4,0);
\coordinate (Aone) at (2,1);
\coordinate (Btwo) at (1,2);
\coordinate (AB) at (3,0);
\draw[context] (L) -- (N);
\draw[context] (N) -- (TL);
\draw[context] (L) -- (J);
\draw[context] (J) -- (TL);
\draw[context] (J) -- (BR);
\draw[context] (BR) -- (TR);
\draw[context] (TL) -- (TR);
\foreach \p in {N,L,TL,J,TR,BR} \fill[black] (\p) circle (2pt);
\draw[forest,-{Latex[length=2.4mm]}] (Aone)--(TL)--(Btwo);
\fill (TL) circle (2pt);
\node[left=4pt,red] at (Aone) {$+$};
\node[below=4pt,blue] at (Btwo) {$-$};

\fill[red] (Aone) circle (2.5pt);
\node[right=3pt,red] at (Aone) {$A_1$};
\fill[blue] (Btwo) circle (2.5pt);
\node[above=3pt,blue] at (Btwo) {$B_2$};
\begin{scope}[shift={(AB)}]
\fill[red] (0,0) -- (90:2.5pt) arc (90:270:2.5pt) -- cycle;
\fill[blue] (0,0) -- (-90:2.5pt) arc (-90:90:2.5pt) -- cycle;
\end{scope}
\node[above=4pt] at (AB) {$\textcolor{red}{A_2}=\textcolor{blue}{B_1}$};
\node[left=4pt] at (N) {$v_n$};
\node at (2,-0.4) {$\alpha(H)$};
\end{scope}

\begin{scope}[shift={(6,0)}]

\coordinate (N) at (0,2);
\coordinate (L) at (0,0);
\coordinate (TL) at (2,2);
\coordinate (J) at (2,0);
\coordinate (TR) at (4,2);
\coordinate (BR) at (4,0);
\coordinate (Aone) at (2,1);
\coordinate (Btwo) at (1,2);
\coordinate (AB) at (3,0);
\draw[forest] (L) -- (N);
\draw[forest] (N) -- (TL);
\draw[forest] (L) -- (J);
\draw[forest] (J) -- (TL);
\draw[forest] (J) -- (BR);
\draw[forest] (BR) -- (TR);
\draw[forest] (TL) -- (TR);
\foreach \p in {N,L,TL,J,TR,BR} \fill[black] (\p) circle (2pt);\node[left=4pt] at (N) {$v_n$};
\node at (2,-0.4) {$\overline H$};
\end{scope}
\end{circuitikz}
\caption{
A directed graph $G$ with a marked vertex $v_n$, a %valid
spanning forest $H$ with sources \(A=(A_1,A_2)\) and sinks \(B=(B_1,B_2)\), its connection
$\alpha(H)$, and its completion $\overline H=G$.
Black and purple arrows indicate edge and path directions,
respectively; the signs at the midpoints indicate whether these
directions agree. The signs multiply to $-1$, and
%Here $A=(A_1,A_2)$, $B=(B_1,B_2)$, and $A_2=B_1$.
the endpoint permutation induced by $\alpha(H)$ has sign $-1$, %and the orientation factors also have product $-1$,
so $\mathrm{sgn}_A^B(\alpha(H))=+1$. For %the
unit edge weights, the forest weight $c(\overline H)$ is also $1$.
Since $H$ is the unique %valid
forest with sources \(A\) and sinks \(B\), the parafermionic observable equals $F(A\leftrightarrow B)=\sum_{H}\mathrm{sgn}_A^B (\alpha(H)) c(\overline H)=1$. %=c(\overline H)$.
See Definition~\ref{def-spanning-forest}.
}
\label{fig:spanning-forest}
\end{figure}
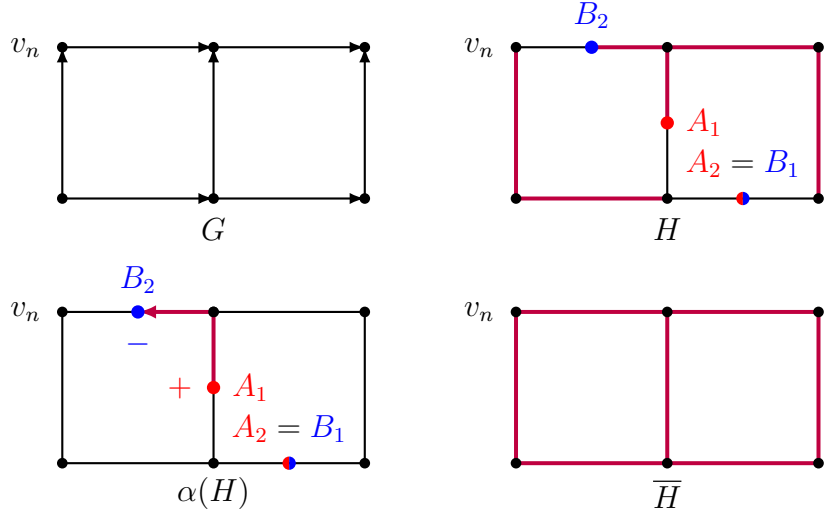

\begin{definition}\label{def-spanning-forest} (See Fig.~\ref{fig:spanning-forest}.) Let $G$ be a directed graph with a marked vertex $v_n$ called the \emph{last vertex}. Let $A$ and $B$ be two ordered sets of midpoints in %a directed graph
$G$ with $|A|=|B|$, called the sets of \emph{sources} and \emph{sinks}, respectively. Recall that  \( G_A^B \) is the directed graph obtained from \( G \)
by subdivision of edges with the midpoints in $A\cup B$. A \emph{spanning forest with sources \( A \) and sinks \( B \)} is a spanning forest in the graph \( G_A^B \) satisfying the following conditions:
\begin{itemize}
    \item Each component, except for one, contains exactly one source and one sink.
    \item One component contains the last vertex $v_n$ %(the one not in \( I \))
    and no sources or sinks.
    \item Each source or sink in \( A \cap B \) has degree 0.
    \item Each source or sink not in \( A \cap B \) has degree 1.
\end{itemize}
For a spanning forest \( H \) with sources \( A \) and sinks \( B \), denote by \( \alpha(H) \) the unique connection from $A$ to $B$ in $H$. Recall definition~\eqref{eq-connection-sign} of the sign $\mathrm{sgn}_A^B (\alpha(H))$
and the notation $\overline H$ for completion (the subgraph of $G$ obtained from $H$ by adding all missing half-edges emanating from $A\cup B$).

The \emph{parafermionic observable} (for the uniform spanning-tree model) is defined as
\begin{equation}
    \label{eq-observable}
     F(A\leftrightarrow B):= \sum_{H} \mathrm{sgn}_A^B (\alpha(H)) c(\overline H),
\end{equation}
 where the sum is %taken
 over all the spanning forests \( H \) with sources \( A \) and sinks \( B \).
\end{definition}

\begin{example} \label{rem:empty} If $A=B=\emptyset$, then $H$ is a spanning tree and $F(\emptyset\leftrightarrow \emptyset)=\sum_{H} c(H)$ is the partition function for the (weighted) uniform spanning-tree model. It would be more natural to define the observable as $F(A\leftrightarrow B)/F(\emptyset\leftrightarrow \emptyset)$ (so that in Lemma~\ref{connection_F_and_L} below, the factor $\det K^{\mathrm{V}}_{\mathrm{V}}$ disappears), but we prefer not to normalize it in this way.
\end{example}

\begin{example} If %$A\cap B=\emptyset$,
$G$ is a plane graph and the points of $A$, $B$, and $v_n$ %are all distinct and
lie on one face boundary in order (the face is bounded by a simple cycle), % and $A\cap B=\emptyset$),
then $H$ is a \emph{banded spanning forest} in $G_A^B$ (with two coincident distinguished vertices $v_n$) as in \cite{Byun-Ciucu-24}.
\end{example}

\begin{remark} The parafermionic observable is invariant under interchanging the sets $A$ and $B$ of sources and sinks. In particular, for $|A|=1$, it is not important which of the points of $A\cup B$ is the source and which is the sink. For $|A|\ge2$, it becomes crucial. The parafermionic observable depends on the ordering of the sets $A$ and $B$; it may change sign if the ordering is changed. It also depends on the direction of edges containing the sources and sinks, and changes sign if the direction of an edge with midpoint in $(A\cup B)\setminus (A\cap B)$ is reversed. However, it does not depend on the direction of the other edges; we consider a directed graph just to fit our general setup.
\end{remark}

This construction has a natural extension to a source-synchronized network.
For two subsets $P,Q\subset\{1,\dots,b\}$ with $|P|=|Q|$,
%$P = (p_1, \ldots, p_k)$ and $Q = (q_1, \ldots, q_k)$ of $\{1,\ldots,b\}$,
define the \emph{total parafermionic observable} as
\begin{equation}\label{eq-total-observable}
 F_P^Q := \sum_{A,B: a_i \in \omega^{p_i},b_i \in \omega^{q_i}} F(A\leftrightarrow B),
\end{equation}
where the sum is over ordered sets \(A\) and \(B\) of midpoints consistent with $P$ and $Q$, respectively.
% \begin{remark}\label{no_valid_a_b}
% If the set of edges with the midpoints in $A$ is not valid relative to $P$ or the set of edges with the midpoints in $B$ is not valid relative to $Q$, then $F_{P}^Q(A, B) = 0$.
% \end{remark}

Let us show how this notion arises naturally in the computation of response-matrix minors.

\begin{lemma} \label{connection_F_and_L}
For any subsets $P,Q\subset\{1,\dots,b\}$ with $|P|=|Q|$,
we have \(\det L_P^Q \cdot \det K^\mathrm{V}_\mathrm{V} = F_P^Q.\)
\end{lemma}

\begin{proof}
Let us %apply Lemma~\ref{l-connections} and
derive a combinatorial formula for the determinant on the right side of~\eqref{eq-connections}. Fix ordered sets $A$ and $B$ of midpoints consistent with $P$ and $Q$, respectively, and a connection $\alpha$ between $A$ and $B$, whose paths do not pass through %the last vertex
$v_n$. Let $S^\alpha:=(S\setminus\overline\alpha)/(\alpha\cup\{v_n\})$ be the graph obtained from the network $S$ by deleting all the edges of $\overline\alpha$, then identifying all the vertices in $\alpha$ (except for the endpoints) and the last vertex~$v_n$, and deleting all the loops in the resulting graph. See Fig.~\ref{fig:auxiliary-graphs} to the right. Denote by $K(S^\alpha)_{\mathrm{V}\setminus \alpha}^{\mathrm{V}\setminus \alpha}$ the Kirchhoff matrix of $S^\alpha$ %the resulting graph
without the row and the column labeled by the vertex obtained from $v_n$.

Let us prove that
$$
\det K(S\setminus\overline\alpha)
_{\mathrm{V}\setminus\alpha}^{\mathrm{V}\setminus\alpha}
= \det {K}(S^\alpha)_{\mathrm{V}\setminus \alpha}^{\mathrm{V}\setminus \alpha}
= \sum_{H^\alpha}c(H^\alpha)
=\sum_{H:\alpha(H)=\alpha}c(\overline H)/c(\overline\alpha),
$$
where the first sum is over all the spanning trees $H^\alpha$ in $S^\alpha$ and the second sum is over all the spanning forests \( H \) in $S^B_A$ with sources in \( A \) and sinks in \( B \) such that $\alpha(H)=\alpha$. Indeed, the first equality follows from $K(S\setminus\overline\alpha)_{\mathrm{V}\setminus \alpha}^{\mathrm{V}\setminus \alpha}={K}(S^\alpha)_{\mathrm{V}\setminus \alpha}^{\mathrm{V}\setminus \alpha}$ because identifying the vertices in $\alpha$ and $v_n$ does not affect these sub-matrices. The second one is the Kirchhoff matrix-tree theorem for $S^\alpha$. The third one holds because
$H\mapsto H^\alpha$, %=(H\setminus\partial\alpha)/(\alpha\cup\{v_n\})
where $H^\alpha$ is constructed from $\overline H$ in the same way as $S^\alpha$ was constructed from $S$, is a bijection between the %spanning
forests and %spanning
trees in question.

Multiplication by $\mathrm{sgn}_P^Q(\alpha)c(\overline\alpha)$, summation over all connections $\alpha$ in $S$ from $P$ to $Q$ whose paths do not pass through %the last vertex
$v_n$, %then over ordered sets $A$ and $B$ consistent with $P$ and $Q$, respectively,
and application of Lemma~\ref{l-connections} lead to the desired result.
\end{proof}

The following lemma shows that in the sum defining $F(A\leftrightarrow B)$, the contributions of some forests cancel out, so that we can restrict ourselves to so-called valid forests, defined now.

A subgraph \(G\) of $S$ is \textit{valid relative to an ordered set $A$ of midpoints} if it contains all edges with the midpoints in $A$ and becomes a spanning tree after deleting these edges.
%
% A spanning forest with sources \( A \) and sinks \( B \) is called \textit{valid} (and \textit{invalid} otherwise) if it becomes a valid subgraph relative to both $A$ and $B$ when we include back the edges with midpoints in $A$ and $B$ into the subgraph and remove the newly introduced vertices $A$ and $B$.
%
A spanning forest $H$ with sources $A$ and sinks $B$ is
\emph{valid} if the completion $\overline H$ is valid relative to both $A$ and $B$. See Fig.~\ref{fig:spanning-forest}.
Otherwise, $H$ is \emph{invalid}.
See Fig.~\ref{fig:bijection}.

%The following lemma shows that in the sum defining $F(A\leftrightarrow B)$, the contributions of invalid forests cancel out, so that we can restrict ourselves to valid spanning forests.

\begin{lemma} \label{lemma:valid_forests}
The parafermionic observable equals
$$F(A\leftrightarrow B) = \sum_{H\text{ \textrm{valid}}} \mathrm{sgn}_A^B (\alpha(H)) c(\overline H),$$
 where the sum is over all valid spanning forests \( H \) with sources \( A \) and sinks \( B \).
\end{lemma}

The proof is analogous to the proof of \cite[Lemma 7.8]{PSS}.

\begin{proof}
It suffices to prove that the sum over invalid spanning forests vanishes, i.e.,
\[
\sum_{H\text{ invalid}} \mathrm{sgn}_A^B(\alpha(H)) c(\overline H) = 0.
\]

Assume \( H \) is an invalid spanning forest with sources \(A\) and sinks \(B\).
Then, by definition, adding the half-edges emanating from \(A \setminus B\) or from \(B \setminus A\)
%(and suppressing the resulting midpoints), the obtained graph is not valid relative to the corresponding set. Since every component of a spanning forest is a tree, this can only happen through the creation of a cycle.
creates a cycle.
Without loss of generality, assume that a simple cycle arises after adding the half-edges emanating from \(A \setminus B\).

Since the vertices in our network are numbered, we may represent each cycle as a string of vertex numbers in the order they appear on the cycle,
%, written in the order in which the cycle is traversed,
starting with the vertex with the minimal number in the cycle. Choose the simple cycle %that corresponds to
with the lexicographically smallest string.

The added half-edges split the cycle into a disjoint collection of paths
\[
s_1 \to \cdots \to t_1, \quad s_2 \to \cdots \to t_2, \quad \ldots, \quad s_m \to \cdots \to t_m,
\]
contained in $H$ and appearing in order along the cycle.

Let us examine how the half-edges reconnect these paths to form a cycle. The half-edge \( s_1 t_m \) has an endpoint in \( A \setminus B \), thus either \( s_1  \) or \( t_m \) belongs to \( A \setminus B \). Suppose \( s_1 \in A \setminus B \) (the other case is symmetric). Then, by the definition of a spanning forest with sources \( A \) and sinks \( B \),  the vertex $s_1$ has degree $1$ in $H$. In particular, $t_1\ne s_1$, and thus
\( t_1 \notin A \),
because each component contains exactly one source or no sources.
Next, the half-edge \( t_1 s_2 \) has again an endpoint in \( A \setminus B \). Since \( t_1 \notin A \), we conclude \( s_2 \in A \setminus B \). Continuing this argument inductively, we deduce that all \( s_i \in A \setminus B \), and the path \( s_i \to \cdots \to t_i \) is a part of the component of %vertex
\( s_i \in A \setminus B \) in %the forest
\( H \).

We now define a sign-reversing involution on the set of invalid forests; see Fig.~\ref{fig:bijection}. Along the identified cycle, we swap the included and excluded half-edges emanating from~$A \setminus B$: that is, at each midpoint \( s_i \), we remove the currently included half-edge and instead include the missing one. This results in a new forest \( \widetilde{H} \), where each %vertex
\( s_i \in A\setminus B \) is reconnected to the component that previously contained \( s_{i-1} \). Then \(\widetilde{H}\) is a spanning forest with sources $A$ and sinks $B$ because:

\begin{itemize}
    \item Each component but one still connects one source %from \( A \)
    to one sink: %in \( B \):
    the involution rearranges which source is connected to which sink, but does not break or merge components.
    \item The component containing \( v_n \) is unchanged, because the cycle does not intersect it.
    \item Exactly one half-edge is still emanating from each midpoint in \( A \setminus B\) and in \( B \setminus A\).
    \item The isolated vertices from \( A\cap B \) remain such, because the cycle does not contain them.
\end{itemize}

Let us show that \( \mathrm{sgn}_A^B(\alpha(\widetilde{H})) \) is opposite to \( \mathrm{sgn}_A^B(\alpha(H)) \). Along the cycle, we replace \( m \) included half-edges with the excluded %opposite
ones, reversing their orientations in the connection. This flips the sign of each factor %corresponding algebraic intersection number
$\langle \vec{a_i}, a_i\rangle$ in~\eqref{eq-connection-sign},
resulting in overall multiplication by \( (-1)^m \).
%, where \( \vec{a_i} \) is the half-edge oriented along the connection. Since there are \( t \) such replacements, this contributes a factor of \( (-1)^t \) to the total sign.

Recall that \(\tau\) is the unique permutation such that for each \( i = 1, \dots, |A| \), the connection \(\alpha\) contains a path from the midpoint of \( a_i \) to the midpoint of \( b_{\tau(i)} \).
Here, the permutation \(\tau\)
%\mscomm{recall its definition here!}
is composed with a cycle of length \(m\) because the involution cyclically reassigns each source to a different sink. %It contributes a factor of
This results in multiplication by \( (-1)^{m-1} \).

Thus, the total sign change is
\(
(-1)^m \cdot (-1)^{m - 1} = -1.
\)
Since $H$ and $\widetilde H$ have the same completion, %$\overline H=\overline{\widetilde H}$, %are preserved under this involution,
it follows that their contributions %of \( H \) and \( \widetilde{H} \) c
ancel out in the sum.
%
%Therefore, the sum over all invalid spanning forests vanishes, and we conclude that $F_{P}^Q(A, B) = (-1)^k \sum_{H} \mathrm{sgn} (\alpha(H)) w(H)$ can be computed by summing only over valid spanning forests $H$ with sources $A$ and sinks $B$.
\end{proof}

\begin{figure}[ht]
\centering
\begin{circuitikz}

\begin{scope}[shift={(0,0)}]

  % ===== координаты =====

  % левый квадрат 2x2
  \coordinate (N)   at (0,2);   % v_n -- верхняя левая
  \coordinate (L)   at (0,0);   % нижняя левая
  \coordinate (TL)  at (2,2);   % общая верхняя вершина
  \coordinate (J)   at (2,0);   % общая нижняя вершина

  % правый квадрат 2x2
  \coordinate (TR)  at (4,2);
  \coordinate (BR)  at (4,0);

  % отмеченные точки
  \coordinate (B2)  at (1,2);
  \coordinate (B1)  at (1,0);

  \coordinate (A1)  at (2,1);
  \coordinate (A2)  at (3,0);

  % ===== рёбра =====

  \draw
    (N) -- (TL) -- (TR) -- (BR) --
    (J) -- (L) -- cycle;

  % общая сторона квадратов
  \draw (TL) -- (J);

  % ===== фиолетовые участки =====

  % левая сторона
  \draw[purple, line width=2pt] (L) -- (N);

  % верх левого квадрата
  \draw[purple, line width=2pt] (TL) -- (B2);

  % низ левого квадрата
  \draw[purple, line width=2pt] (J) -- (B1);

  % A1: теперь нижнее полуребро
  \draw[purple, line width=2pt] (A1) -- (J);

  % A2: теперь правое полуребро
  \draw[purple, line width=2pt] (A2) -- (BR);

  % верх правого квадрата
  \draw[purple, line width=2pt] (TR) -- (TL);

  % правая сторона
  \draw[purple, line width=2pt] (BR) -- (TR);

  % ===== подписи =====

  \node[left] at (N) {$v_n$};

  \node[right, red] at (A1) {$A_1$};
  \node[above, red] at (A2) {$A_2$};

  \node[above, blue] at (B1) {$B_1$};
  \node[above, blue] at (B2) {$B_2$};

  % ===== вершины =====

  \foreach \p in {N,L,TL,J,TR,BR}
    \fill (\p) circle (2pt);

  \foreach \p in {A1,A2}
    \fill[red] (\p) circle (2pt);

  \foreach \p in {B1,B2}
    \fill[blue] (\p) circle (2pt);
  \draw[<->, thick, <->] (5,1) -- (7,1) node[midway, above] {};
\end{scope}

% === Правая копия: вершины + рёбра ===
\begin{scope}[shift={(8,0)}]

  % ===== координаты =====

  % левый квадрат 2x2
  \coordinate (N)   at (0,2);   % v_n -- верхняя левая
  \coordinate (L)   at (0,0);   % нижняя левая
  \coordinate (TL)  at (2,2);   % общая верхняя вершина
  \coordinate (J)   at (2,0);   % общая нижняя вершина

  % правый квадрат 2x2
  \coordinate (TR)  at (4,2);
  \coordinate (BR)  at (4,0);

  % отмеченные точки
  \coordinate (B2)  at (1,2);
  \coordinate (B1)  at (1,0);

  \coordinate (A1)  at (2,1);
  \coordinate (A2)  at (3,0);

  % ===== рёбра =====

  % два одинаковых квадрата
  \draw
    (N) -- (TL) -- (TR) -- (BR) --
    (J) -- (L) -- cycle;

  % общая сторона квадратов
  \draw (TL) -- (J);

  % ===== фиолетовые участки =====

  % левая сторона
  \draw[purple, line width=2pt]
    (L) -- (N);

  % верх левого квадрата
  \draw[purple, line width=2pt]
    (TL) -- (B2);

  % низ левого квадрата
  \draw[purple, line width=2pt]
    (J) -- (B1);

  % общая вертикальная сторона
  \draw[purple, line width=2pt]
    (A1) -- (TL);

  % низ правого квадрата
  \draw[purple, line width=2pt]
    (A2) -- (J);

  % верх правого квадрата
  \draw[purple, line width=2pt]
    (TR) -- (TL);

  % правая сторона
  \draw[purple, line width=2pt]
    (BR) -- (TR);

  % ===== подписи =====

  \node[left] at (N) {$v_n$};

  \node[right, red] at (A1) {$A_1$};
  \node[above, red] at (A2) {$A_2$};

  \node[above, blue] at (B1) {$B_1$};
  \node[above, blue] at (B2) {$B_2$};

  % ===== вершины =====

  % чёрные вершины только в углах
  \foreach \p in {N,L,TL,J,TR,BR}
    \fill (\p) circle (2pt);

  % красные точки
  \foreach \p in {A1,A2}
    \fill[red] (\p) circle (2pt);

  % синие точки
  \foreach \p in {B1,B2}
    \fill[blue] (\p) circle (2pt);

\end{scope}

\end{circuitikz}
\caption{A sign-reversing involution on the set of invalid forests.}
\label{fig:bijection}
\end{figure}
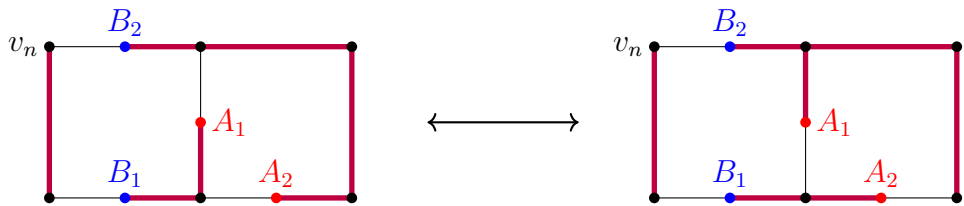

\section{Modified Temperley correspondence}
\label{sec-Temperley}

In this section, we continue the proof of Theorem~\ref{th-all-minor-matrix-tree}, namely, rewrite the parafermionic observable as a sum over valid subgraphs of $S$. For that, we use %dimers and double-dimers, and
a modification of the Temperley correspondence between dimers and spanning trees to show that $H\mapsto\overline{H}$ is a bijection between valid spanning forests in $S^B_A$ and subgraphs of $S$ valid relative to both $A$ and $B$.
In the setup of Definition~\ref{def-spanning-forest}, recall that $G'$ is the graph obtained from $G$ by subdivision of all edges.

\emph{A dimer configuration} (or \emph{a perfect matching}) on $G'\setminus (\{ v_n\} \cup A)$ %a graph
is a set of half-edges such that each vertex or midpoint of $G'\setminus (\{ v_n\} \cup A)$ is the endpoint of exactly one half-edge.

An \emph{oriented dimer configuration directed towards midpoints} is %an oriented
a perfect matching on $G'\setminus (\{ v_n\} \cup A)$ %such that each half-edge is
with each half-edge oriented from a vertex to a midpoint. See Fig.~\ref{fig:double_dimer}, bottom-left. %Similarly,
An \emph{oriented dimer configuration directed from midpoints} is a perfect matching on $G'\setminus (\{ v_n\} \cup B)$ %such that
with half-edges %are
oriented from midpoints to vertices.  See Fig.~\ref{fig:double_dimer}, bottom-middle.

 An \emph{oriented double-dimer configuration} $\vec D$ on $G' \setminus \{v_n\}$ with sources $A$ and sinks $B$ is a set of oriented half-edges in $G' \setminus (\{ v_n\} \cup (A \cap B))$ such that each source in $A \setminus B$ has exactly one outgoing and no incoming half-edge, each sink in $B \setminus A$ has exactly one incoming and no outgoing half-edge, and every other vertex in $G' \setminus (\{ v_n\} \cup (A \cap B))$ has exactly one incoming and one outgoing half-edge.  See Fig.~\ref{fig:double_dimer}, bottom-right.

Consider the collection of oriented paths from sources to sinks in \( \vec D \). After suppression of all midpoints except those in \( A \) and \( B \), and adding a trivial one-vertex path for each vertex in \(A \cap B \), we obtain a connection from \( A \) to \( B \). We denote this connection by \( \alpha(\vec D) \).

Given a half-edge $\vec{yx}$ directed from a midpoint~$y$ to a vertex~$x$, its \emph{half-edge flip} is the half-edge $\vec{yz}$, where $z \ne x$ is the other endpoint of the edge containing %the midpoint
$y$. See Fig.~\ref{fig:half_edge_flip}.
 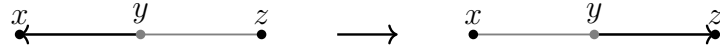
\begin{figure}[ht]
 \centering
\begin{circuitikz}

\begin{scope}[shift={(0,0)}]
  \coordinate (Xl) at (0,0);
  \coordinate (Yl) at (1.6,0);
  \coordinate (Zl) at (3.2,0);
  \draw[gray, line width=0.7pt] (Xl) -- (Zl);
  \draw[<- , line width=0.9pt] (Xl) -- (Yl);
  \fill[gray] (Yl) circle (1.8pt);
  \fill[black]  (Xl) circle (1.8pt);
  \fill[black]  (Zl) circle (1.8pt);
  \node[above] at (Xl) {$x$};
  \node[above] at (Yl) {$y$};
  \node[above] at (Zl) {$z$};
\end{scope}

\draw[->, line width=0.9pt] (4.2,0) -- (5.0,0);
\begin{scope}[shift={(6,0)}]
  \coordinate (Xr) at (0,0);
  \coordinate (Yr) at (1.6,0);
  \coordinate (Zr) at (3.2,0);
  \draw[gray, line width=0.7pt] (Xr) -- (Zr);
  \draw[-> , line width=0.9pt] (Yr) -- (Zr);
  \fill[gray] (Yr) circle (1.8pt);
  \fill[black]  (Xr) circle (1.8pt);
  \fill[black]  (Zr) circle (1.8pt);
  \node[above] at (Xr) {$x$};
  \node[above] at (Yr) {$y$};
  \node[above] at (Zr) {$z$};
\end{scope}

\end{circuitikz}
 \caption{A half-edge \(\vec{yx}\) and its half-edge flip \(\vec{yz}\).}
 \label{fig:half_edge_flip}
 \end{figure}

 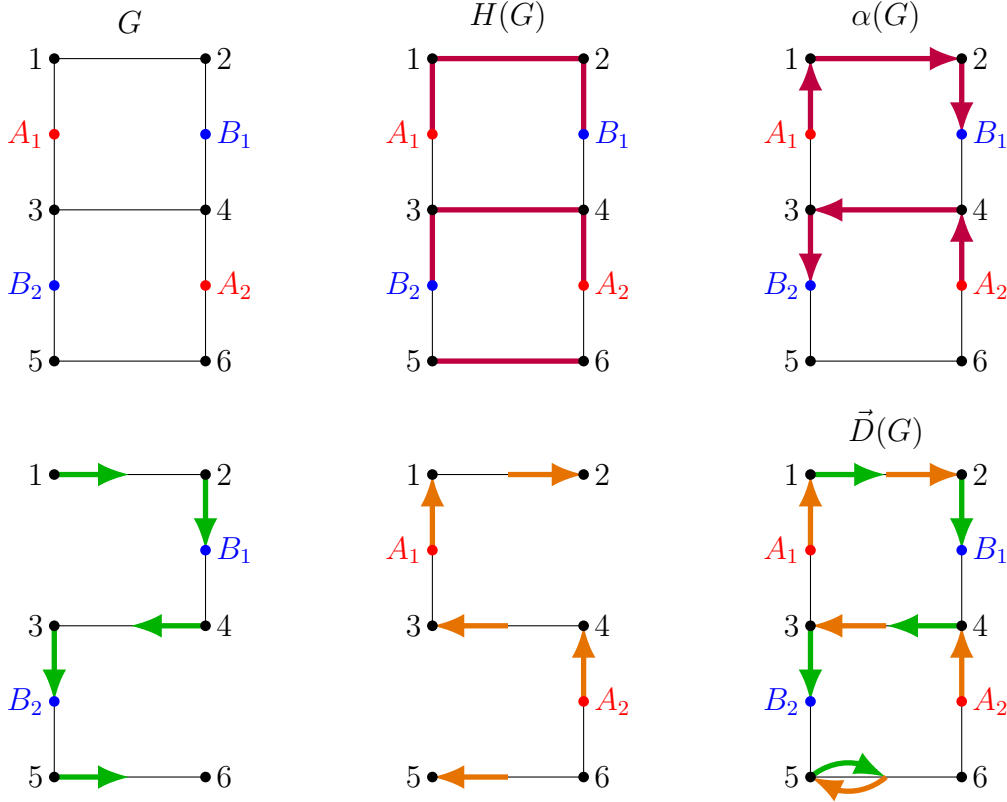
\begin{figure}[ht]
 \centering
\begin{circuitikz}

% 1
\begin{scope}[shift={(0,0)}]
\node[above] at (1,4.2) {$G$};
\coordinate (A)   at (0,0);
\coordinate (A1)   at (0,3);
\coordinate (A2)   at (2,1);
\coordinate (B2)   at (0,1);
\coordinate (B1)   at (2,3);
\coordinate (B)    at (2,0);
\coordinate (C)    at (0,2);
\coordinate (D)    at (2,2);
\coordinate (E)    at (0,4);
\coordinate (F)    at (2,4);
\node[left] at (A) {$5$};
\node[right] at (B) {$6$};
\node[right] at (D) {$4$};
\node[right] at (F) {$2$};
\node[left] at (E) {$1$};
\node[left] at (C) {$3$};
\node[left, red] at (A1) {$A_1$};
\node[right, red] at (A2) {$A_2$};
\node[left, blue] at (B2) {$B_2$};
\node[right, blue] at (B1) {$B_1$};
\draw (A) -- (B) -- (F) -- (E) -- (A);
\draw (D) -- (C);
\foreach \p in {A, B, C, D, E, F} \fill (\p) circle (2pt);
\foreach \p in {A1, A2} \fill[red] (\p) circle (2pt);
\foreach \p in {B1, B2} \fill[blue] (\p) circle (2pt);
\end{scope}

% 2
\begin{scope}[shift={(5,0)}]
\node[above] at (1,4.2) {$H(G)$};
\coordinate (A)   at (0,0);
\coordinate (A1)   at (0,3);
\coordinate (A2)   at (2,1);
\coordinate (B2)   at (0,1);
\coordinate (B1)   at (2,3);
\coordinate (B)    at (2,0);
\coordinate (C)    at (0,2);
\coordinate (D)    at (2,2);
\coordinate (E)    at (0,4);
\coordinate (F)    at (2,4);
\draw (A) -- (B) -- (F) -- (E) -- (A);
\draw (D) -- (C);
\draw[purple, line width=2pt] (A1) -- (E) -- (F) -- (B1);
\draw[purple, line width=2pt] (B2) -- (C) -- (D) -- (A2);
\draw[purple, line width=2pt] (A) -- (B);
\node[left] at (A) {$5$};
\node[right] at (B) {$6$};
\node[right] at (D) {$4$};
\node[right] at (F) {$2$};
\node[left] at (E) {$1$};
\node[left] at (C) {$3$};
\node[left, red] at (A1) {$A_1$};
\node[right, red] at (A2) {$A_2$};
\node[left, blue] at (B2) {$B_2$};
\node[right, blue] at (B1) {$B_1$};
\foreach \p in {A, B, C, D, E, F} \fill (\p) circle (2pt);
\foreach \p in {A1, A2} \fill[red] (\p) circle (2pt);
\foreach \p in {B1, B2} \fill[blue] (\p) circle (2pt);
\end{scope}

% 3
\begin{scope}[shift={(10,0)}]
\node[above] at (1,4.2) {$\alpha(G)$};
\coordinate (A)   at (0,0);
\coordinate (A1)   at (0,3);
\coordinate (A2)   at (2,1);
\coordinate (B2)   at (0,1);
\coordinate (B1)   at (2,3);
\coordinate (B)    at (2,0);
\coordinate (C)    at (0,2);
\coordinate (D)    at (2,2);
\coordinate (E)    at (0,4);
\coordinate (F)    at (2,4);
\draw (A) -- (B) -- (F) -- (E) -- (A);
\draw (D) -- (C);
\draw[{Latex}-, purple, line width=2pt] (B1) -- (F);
\draw[{Latex}-, purple, line width=2pt] (F) -- (E);
\draw[-{Latex}, purple, line width=2pt] (A1) -- (E);
\draw[{Latex}-, purple, line width=2pt] (B2) -- (C);
\draw[{Latex}-, purple, line width=2pt] (C) -- (D);
\draw[{Latex}-, purple, line width=2pt] (D) -- (A2);
\node[left] at (A) {$5$};
\node[right] at (B) {$6$};
\node[right] at (D) {$4$};
\node[right] at (F) {$2$};
\node[left] at (E) {$1$};
\node[left] at (C) {$3$};
\node[left, red] at (A1) {$A_1$};
\node[right, red] at (A2) {$A_2$};
\node[left, blue] at (B2) {$B_2$};
\node[right, blue] at (B1) {$B_1$};
\foreach \p in {A, B, C, D, E, F} \fill (\p) circle (2pt);
\foreach \p in {A1, A2} \fill[red] (\p) circle (2pt);
\foreach \p in {B1, B2} \fill[blue] (\p) circle (2pt);
\end{scope}

% 4
\begin{scope}[shift={(0,-5.5)}]
\coordinate (A)   at (0,0);
\coordinate (B2)  at (0,1);
\coordinate (B1)  at (2,3);
\coordinate (B)   at (2,0);
\coordinate (C)   at (0,2);
\coordinate (D)   at (2,2);
\coordinate (E)   at (0,4);
\coordinate (F)   at (2,4);
\node[left] at (A) {$5$};
\node[right] at (B) {$6$};
\node[right] at (D) {$4$};
\node[right] at (F) {$2$};
\node[left] at (E) {$1$};
\node[left] at (C) {$3$};
\node[left, blue] at (B2) {$B_2$};
\node[right, blue] at (B1) {$B_1$};
\draw (B) -- (A) -- (C) -- (D) -- (F) -- (E);
\draw[-{Latex}, green!70!black, line width=2pt] (E) -- ($(E)!0.5!(F)$);
\draw[{-Latex}, green!70!black, line width=2pt] (F) -- ($(F)!0.5!(D)$);
\draw[-{Latex}, green!70!black, line width=2pt] (D) -- ($(D)!0.5!(C)$);
\draw[-{Latex}, green!70!black, line width=2pt] (C) -- ($(C)!0.5!(A)$);
\draw[-{Latex}, green!70!black, line width=2pt] (A) -- ($(A)!0.5!(B)$);
\foreach \p in {A, B, C, D, E, F} \fill (\p) circle (2pt);
\foreach \p in {B1, B2} \fill[blue] (\p) circle (2pt);
\end{scope}

% 5
\begin{scope}[shift={(5,-5.5)}]
\coordinate (A)   at (0,0);
\coordinate (A1)  at (0,3);
\coordinate (A2)  at (2,1);
\coordinate (B)   at (2,0);
\coordinate (C)   at (0,2);
\coordinate (D)   at (2,2);
\coordinate (E)   at (0,4);
\coordinate (F)   at (2,4);
\node[left] at (A) {$5$};
\node[right] at (B) {$6$};
\node[right] at (D) {$4$};
\node[right] at (F) {$2$};
\node[left] at (E) {$1$};
\node[left] at (C) {$3$};
\node[left, red] at (A1) {$A_1$};
\node[right, red] at (A2) {$A_2$};
\draw (A) -- (B) -- (D) -- (C) -- (E) -- (F);
\draw[-{Latex}, orange!90!black, line width=2pt] ($(E)!0.5!(F)$) -- (F);
\draw[-{Latex}, orange!90!black, line width=2pt] ($(E)!0.5!(C)$) -- (E);
\draw[-{Latex}, orange!90!black, line width=2pt] ($(C)!0.5!(D)$) -- (C);
\draw[-{Latex}, orange!90!black, line width=2pt] ($(D)!0.5!(B)$) -- (D);
\draw[-{Latex}, orange!90!black, line width=2pt] ($(A)!0.5!(B)$) -- (A);
\foreach \p in {A, B, C, D, E, F} \fill (\p) circle (2pt);
\foreach \p in {A1, A2} \fill[red] (\p) circle (2pt);
\end{scope}

% 6
\begin{scope}[shift={(10,-5.5)}]
\node[above] at (1,4.2) {$\vec D(G)$};
\coordinate (A)   at (0,0);
\coordinate (A1)  at (0,3);
\coordinate (A2)  at (2,1);
\coordinate (B2)  at (0,1);
\coordinate (B1)  at (2,3);
\coordinate (B)   at (2,0);
\coordinate (C)   at (0,2);
\coordinate (D)   at (2,2);
\coordinate (E)   at (0,4);
\coordinate (F)   at (2,4);
\node[left] at (A) {$5$};
\node[right] at (B) {$6$};
\node[right] at (D) {$4$};
\node[right] at (F) {$2$};
\node[left] at (E) {$1$};
\node[left] at (C) {$3$};
\node[left, red] at (A1) {$A_1$};
\node[right, red] at (A2) {$A_2$};
\node[left, blue] at (B2) {$B_2$};
\node[right, blue] at (B1) {$B_1$};
\draw (A) -- (B) -- (F) -- (E) -- (A);
\draw (D) -- (C);
\draw[-{Latex}, orange!90!black, line width=2pt] ($(E)!0.5!(F)$) -- (F);
\draw[-{Latex}, orange!90!black, line width=2pt] ($(E)!0.5!(C)$) -- (E);
\draw[-{Latex}, orange!90!black, line width=2pt] ($(C)!0.5!(D)$) -- (C);
\draw[-{Latex}, orange!90!black, line width=2pt] ($(D)!0.5!(B)$) -- (D);

\draw[-{Latex}, green!70!black, line width=2pt] (E) -- ($(E)!0.5!(F)$);
\draw[{-Latex}, green!70!black, line width=2pt] (F) -- ($(F)!0.5!(D)$);
\draw[-{Latex}, green!70!black, line width=2pt] (D) -- ($(D)!0.5!(C)$);
\draw[-{Latex}, green!70!black, line width=2pt] (C) -- ($(C)!0.5!(A)$);

\draw[-{Latex}, green!70!black, line width=2pt] (A) to[out=40, in=140] ($(A)!0.5!(B)$);
\draw[-{Latex}, orange!90!black, line width=2pt] ($(A)!0.5!(B)$) to[out=220, in=-40] (A);
\foreach \p in {A, B, C, D, E, F} \fill (\p) circle (2pt);
\foreach \p in {A1, A2} \fill[red] (\p) circle (2pt);
\foreach \p in {B1, B2} \fill[blue] (\p) circle (2pt);
\end{scope}

\end{circuitikz}
 \caption{Top: A valid graph $G$, the forest $H(G)$, and the connection $\alpha(G)$. Bottom: The construction of an oriented double-dimer configuration $\vec D(G)$ with sources $A$ and sinks $B$ on $G' \setminus \{ 6 \}$. See the proof of Lemma~\ref{lemma:bijection}.
 }
 \label{fig:double_dimer}
 \end{figure}

\begin{lemma}[Modified Temperley correspondence] \label{lemma:bijection}
(See
Fig.~\ref{fig:double_dimer}.) For any subgraph $G$ valid relative to both
$A$ and $B$ separately, there %exists
is a unique oriented double-dimer configuration $\vec D(G)$ on $G' \setminus \{ v_n \}$ with sources~$A$ and sinks $B$ and a unique valid spanning forest $H(G)$ on $G_A^B$ with sources $A$ and sinks $B$.
Moreover, $\alpha(\vec D(G)) = \alpha(H(G))$, $\overline{H(G)}=G$, and \mscommnew{each %midpoint
$B_i\in B\!\setminus\! A$
is the unique common point of the path of $\alpha(\vec D(G))$ ending at $B_i$ and the path from $B_i$ to $v_n$ in~%the tree
$G'\!\setminus\! A$.}
%%%
%each path of $\alpha(\vec D(G))$ joining the midpoints of two distinct edges $a$ and $b$, followed by the path from the midpoint of $b$ to $v_n$ in the spanning tree $G'\setminus A$, is again a simple path.
\end{lemma}

\begin{proof}

 Since \(G\) is valid relative to \(A\), the subgraph \(G \setminus A\) is a spanning tree in the network. There is a unique oriented dimer configuration directed towards midpoints on \(G' \setminus (A \cup \{v_n\})\).
 It is constructed by choosing $v_n$ as the root in the tree  $G' \setminus A$ and connecting each vertex in \(G' \setminus (A \cup \{v_n\})\) to a midpoint in the direction towards the root. The uniqueness holds because the dimer configuration is uniquely reconstructed, starting from the edges emanating from the root.
 Similarly, we construct a unique oriented dimer configuration directed from midpoints on \(G' \setminus (B \cup \{v_n\})\).
 The union $\vec D(G)$ of these oriented dimer configurations is the desired oriented double-dimer configuration $\vec D(G)$ on $G' \setminus \{ v_n \}$ with sources $A$ and sinks $B$ (see Fig.~\ref{fig:double_dimer}).

Clearly, any oriented double-dimer configuration on $G' \setminus \{ v_n \}$ with sources $A$ and sinks $B$ splits into two oriented dimer configurations directed towards and from midpoints on \(G' \setminus (A \cup \{v_n\})\) and \(G' \setminus (B \cup \{v_n\})\), respectively.
Hence, the oriented double-dimer configuration is unique.
%because both oriented dimer configurations are unique.

Now, we construct a valid spanning forest $H(G)$ with sources \(A\) and sinks \(B\).
Denote $\alpha' = \alpha'(\vec D(G))$ the set of oriented paths from sources to sinks in $\vec D(G)$, so that suppression of all midpoints in $\alpha'$ except for $A$ and $B$ and adding trivial paths for vertices in \( A \cap B \) gives $\alpha(\vec D(G))$.

Denote by \(\vec{H}(G)\) the directed subgraph obtained from \(\vec D(G)\) by adding vertices in \( A \cap B \) and by replacing each oriented half-edge, which is not contained in \(\alpha'\) and is directed from a midpoint to a vertex, with its half-edge flip. See Fig.~\ref{fig:half_edge_flip}.

We then obtain $H(G)$ by forgetting the orientation of all half-edges in $\vec{H}(G)$ and suppressing all midpoints %in $\alpha'$
except for those in \(A\) and \(B\).

The subgraph $H(G)$ is well-defined, that is, both half-edges emanating from each midpoint except for those in \(A\) and \(B\) are contained in \(\vec{H}(G)\) with some orientation. Indeed, since $G$ is valid relative to \(A\), it follows that \(G' \setminus \alpha'\) is acyclic, hence the half-edges in \(\vec D(G) \setminus \alpha'\) come in pairs: one half-edge from a vertex \(x\) to a midpoint \(y\), and the other from the midpoint \(y\) to the vertex \(x\). The half-edge flip takes each pair to a complete undirected edge in $\vec{H}(G)$.
%\(H(G)\), making the suppression of midpoints not in \(A \cup B\) possible.
%, which we replace with its half-edge flip in $\vec{H}(G)$. Thus, each such pair corresponds to a complete undirected edge in \(H(G)\), making the suppression of midpoints not in \(A \cup B\) possible.

Let us show that \(H(G)\) is a valid spanning forest with sources \(A\) and sinks \(B\).

In $\vec{H}(G)$, each vertex or midpoint, except for those in \(B \cup \{v_n\}\), has exactly one outgoing half-edge, while vertices or midpoints in
$B \cup \{v_n\}$ have none. Hence each component either contains a vertex of $B \cup \{v_n\}$, so all directed paths in it terminate there, or else contains a directed cycle. The latter is impossible, since such a cycle would consist entirely of edges outside $\alpha'$, and therefore lie in
$G' \setminus \alpha'% \subseteq G' \setminus (A \cup B)
$, which is acyclic.

Therefore every component contains exactly one vertex of $B \cup \{v_n\}$ and is a tree directed toward it. By construction, each component containing a sink in $B$ also contains the corresponding source in $A$.
Thus, \(H(G)\) is a valid spanning forest with sources \(A\) and sinks
\(B\), and by construction, $\alpha(\vec D(G))=\alpha(H(G))$ and $\overline{H(G)}=G$.

Now, we prove that a valid spanning forest \(H\) with sources $A$ and sinks $B$ is unique, by constructing an oriented double-dimer configuration with sources $A$ and sinks $B$ from \(H\).

Consider the edges in \(H\) that are not contained in \(\alpha(H)\). They form a forest, and the roots of its trees can be chosen among the vertices in \(\alpha(H)\) and the vertex \(v_n\). For each such tree \(T\) with a root \(r\), we construct two oriented dimer configurations on \(T' \setminus \{r\}\): one directed towards midpoints and one directed from midpoints.

Then, after subdividing all edges in \(\alpha(H)\) and orienting the resulting half-edges from sources to sinks, we %combine these oriented edges with the previously constructed dimer configurations to
obtain an oriented double-dimer configuration with sources \(A\) and sinks \(B\) on \(G' \setminus \{v_n\}\). Since the latter is unique, it follows that \(H\) is unique.

It remains to prove the `moreover' part. %last assertion.
Orient the tree $G\setminus A$ towards $v_n$, and let $\alpha_i$ be the path of $\alpha(\vec D(G))$ ending at $B_i$.
%joining the midpoints of two distinct edges $a$ and $b$.
By the above construction, $\alpha_i$ leaves each of its vertices along the edge directed towards $v_n$ in $G\setminus A$; being simple, it does not turn back at a midpoint. Hence each half-edge of $\alpha_i$, except the first one, is directed towards $v_n$ in $G'\setminus A$, and the same is true of the path from $B_i$ %the midpoint of $b$
to $v_n$. Along such half-edges the graph distance to $v_n$ decreases, so
$B_i$ is the unique common point of the two paths.
% the concatenation of the two paths visits each point at most once, the midpoint of $a$ lying outside $G'\setminus A$. The concatenation is therefore a simple path.
\end{proof}

\begin{remark}
The construction of $\vec D(G)$ in the proof of Lemma \ref{lemma:bijection} is a modification of the classical Temperley correspondence \cite{Temperley1974, KPW} and its surface analogs \cite{sun2016, ber2024}. Recall that, in the classical setting, a spanning tree rooted at $v_n$ is oriented toward the root, and each non-root vertex is matched to the midpoint of the edge pointing toward its parent. This produces an oriented dimer configuration on the subdivided graph, %In our setting, since $G$ is valid relative to $A$, it contains all edges with midpoints in $A$, and $G\setminus A$ is a spanning tree. Applying the same rule to this tree gives an oriented dimer configuration directed toward midpoints. Similarly, since $G$ is valid relative to $B$, the graph $G\setminus B$ is a spanning tree, and the analogous rule gives an oriented dimer configuration directed from midpoints. The union of these two configurations is the oriented double-dimer configuration $\vec D(G)$. Its oriented paths from the sources $A$ to the sinks $B$ determine the connection $\alpha(G)$.
%
%Note that the classical Temperley correspondence makes use both of the graph $G$ and its dual $G^*$.
and then the same construction is applied to the dual graph. Our construction does not use the dual graph, instead we use two rooted trees $G\setminus A$ and $G\setminus B$ to produce an oriented double dimer configuration on the same graph. This is a non-planar/non-surface version of Temperley's trick and %a generalization of
the ``gliding'' from \cite{KPW,Byun-Ciucu-24}.

%\begin{remark} \label{remark:bijection}
Beware that the relation to Temperley's bijection is the method of the proof; the lemma itself does not generalize to a nontrivial bijection.
For a graph $G$ that is non-valid relative to $A$ or $B$,
the oriented double-dimer configurations on $G' \setminus \{ v_n \}$ with sources $A$ and sinks $B$ are not necessarily in bijection with
valid spanning forests on $G_A^B$ with sources $A$ and sinks $B$.
\end{remark}

\begin{notation*}
    Set $\alpha(G):=\alpha(H(G)) = \alpha(\vec D(G))$ under the notation of Lemma~\ref{lemma:bijection}.
\end{notation*}

The following result follows directly from %the combination of
Lemmas~\ref{lemma:valid_forests} and~\ref{lemma:bijection}.

\begin{corollary}\label{valid_graphs}
We have
$$F(A\leftrightarrow B) = \sum_{G} \mathrm{sgn}_A^B(\alpha(G)) c(G),$$
where the sum is over subgraphs $G$ that are valid relative to both \(A\) and \(B\) separately.
\end{corollary}

Interestingly, despite the statement of Corollary~\ref{valid_graphs} does not involve any dimers, the only proof we know uses modified Temperley's correspondence.

\section{Total parafermionic observable}
\label{sec-conclusion}

In this section, we conclude the proof of Theorem~\ref{th-all-minor-matrix-tree} by expressing the total parafermionic observable in terms of the multiplicities introduced before the theorem.

We need an auxiliary notion.
Let $\Gamma = (\gamma_1, \ldots, \gamma_k)$ be a basis of simple cycles in \( G \), let \( A\) be an ordered set of midpoints in $G$, let \( (a_1, \ldots, a_{k}) \) be the edges with midpoints in \( A \) in order.

 If $\pi\in \mathcal{S}_{k}$ is a permutation such that $a_j\subset \gamma_{\pi(j)}$ for $j=1,\dots,k$, then denote by
\(\vec{a}_j\) the edge \( a_j \) oriented consistently with the orientation of the cycle \( \gamma_{\pi(j)} \).
Then define
\[
\mathrm{sgn}_A(\Gamma) := \sum_{\pi\in \mathcal{S}_{k}:a_j\subset \gamma_{\pi(j)}\text{ for all }j}
    \mathrm{sgn}(\pi) \prod_{j=1}^k \langle \vec{a}_j, a_j \rangle.
\]

We shall see that $\mathrm{sgn}_A(\Gamma)\in \{0,\pm 1\}$, once $G$ is valid with respect to $A$.

\begin{lemma}\label{sign_lemma}
For any subgraph $G$ valid relative to both $A$ and $B$ separately, %\subset\{1,\dots, b\}$,
there is a
%any
basis $\Gamma$ of simple cycles of $G$ such that %we have
$\mathrm{sgn}_A^B(\alpha(G))$ = $\mathrm{sgn}_A( \Gamma)\mathrm{sgn}_B(\Gamma)$.
\end{lemma}

\begin{proof}

With the notation introduced before~\eqref{eq-connection-sign},
we are going to construct
a basis \(\Gamma = (\gamma_1, \ldots, \gamma_k)\) such that the path from the midpoint of \(a_i\) to the midpoint of \(b_{\tau(i)}\) is contained in the simple cycle \(\gamma_i\) and is oriented in the same direction.

For each \(i=1,\dots,k\), denote by \(\alpha_i\) the path of \(\alpha(G)\) from the
midpoint of \(a_i\) to the midpoint of \(b_{\tau(i)}\); if \(a_i=b_{\tau(i)}\),
this path consists of a single midpoint. %is the one-vertex path at the common midpoint.

Since \(G\) is valid relative to \(A\) and to \(B\), the graphs \(G'\setminus A\) and \(G'\setminus B\)
%\(T_A:=G\setminus A\) and \(T_B:=G\setminus B\)
are %spanning
trees containing all the vertices of $G$.
If \(a_i\ne b_{\tau(i)}\), then let
\(\beta_{\tau(i)}\) be the simple path in \(G'\setminus A\)  %\(T_A'\)
from the midpoint of \(b_{\tau(i)}\)
to \(v_n\), and \(\eta_i\) be the simple path in \(G'\setminus B\) %\(T_B'\)
from the midpoint of \(a_i\) to \(v_n\). If \(a_i=b_{\tau(i)}\), %this midpoint lies in neither tree; in this case let \(\beta_{\tau(i)}\) and \(\eta_i\) begin with the half-edges \(\vec b_{\tau(i)}\) and against \(\vec a_i\),
then let \(\beta_{\tau(i)}\) and \(\eta_i\) begin with the half-edges in the direction of $b_{\tau(i)}$ and opposite to it, respectively,
and continue with the simple paths in
%\(T_A'\) and \(T_B'\)
\(G'\setminus A\) and \(G'\setminus B\)
to \(v_n\). %Both paths are simple, being paths in trees. Besides,
The path \(\beta_{\tau(i)}\) contains no midpoint of \(A\) and \(\eta_i\) contains no midpoint of \(B\), except possibly their
starting points.

By Lemma~\ref{lemma:bijection}, the concatenation of \(\alpha_i\) with
\(\beta_{\tau(i)}\) is a simple path. After interchanging \(A\) and \(B\), which
reverses the paths of \(\alpha(G)\), the same lemma shows that the concatenation
of \(\eta_i\) reversed with \(\alpha_i\) is again a simple path. % as well.
Thus
\(\beta_{\tau(i)}\) and~\(\eta_i\) meet \(\alpha_i\) only at the midpoints of~\(b_{\tau(i)}\) and \(a_i\), respectively. For \(a_i=b_{\tau(i)}\), the same holds as the common midpoint lies in neither tree.

Both \(\beta_{\tau(i)}\) and \(\eta_i\) end at \(v_n\), hence they have a common
point distinct from their starting points; let \(w\) be the first such point, a
vertex or a midpoint, along \(\beta_{\tau(i)}\). (The starting points have to be
excluded because they coincide when \(a_i=b_{\tau(i)}\).)
Concatenating \(\alpha_i\), the arc of \(\beta_{\tau(i)}\) from the midpoint of
\(b_{\tau(i)}\) to \(w\), and the arc of \(\eta_i\) from \(w\) back to the
midpoint of \(a_i\), we obtain a cycle \(\gamma_i\), oriented so as to traverse
\(\alpha_i\) in its given direction. By the choice of \(w\), these three pieces meet
only at their endpoints; thus \(\gamma_i\) is simple. %Moreover, \(\beta_{\tau(i)}\) and \(\eta_i\) leave those midpoints through the half-edges not contained in \(\alpha_i\), since otherwise they would meet \(\alpha_i\) also at the endpoints of these half-edges; for \(a_i=b_{\tau(i)}\) this holds by construction. Hence
In particular, \(\gamma_i\) traverses the edges \(a_i\) and \(b_{\tau(i)}\)
in full, in the directions of \(\vec a_i\) and \(\vec b_{\tau(i)}\).

Now, we prove that the cycles \(\gamma_1, \ldots, \gamma_k\) form a basis of simple cycles in $G$.
It suffices to show that they are linearly independent. Indeed, since $G\setminus A$ is a spanning tree, $\dim H_1(G;\mathbb Z)=|A|=k$, so
any $k$ linearly independent cycles automatically form a basis.

Introduce a partial order on \( \{b_1, \ldots, b_{k}\}\): let \(b_i > b_j\) if the path $\beta_{i}$ contains \(b_j\). This relation has no cycles: if $b_i>b_j$, then after reaching $b_j$, the path
$\beta_i$ follows $\beta_j$, so a cyclic chain $b_{i_1}>\dots>b_{i_r}>b_{i_1}$
would force $\beta_{i_1}$ to return to $b_{i_1}$, contradicting its simplicity.
Hence a maximal element exists.

Take a maximal element $b_i=b_{\tau(j)}$. Then $\gamma_j$ is the only cycle
among $\gamma_1, \ldots, \gamma_k$ that contains
%of our collection containing
the edge \(b_{\tau(j)}\). Indeed, take any $m\neq j$. Then the path $\beta_{\tau(m)}$ does not contain $b_{\tau(j)}$; otherwise,
$b_{\tau(m)}>b_{\tau(j)}$, contradicting maximality. The path $\eta_m$ contains
no midpoint of $B$ except possibly its starting point. The path \(\alpha_m\)
contains no midpoint of $B$ except that of \(b_{\tau(m)}\) because the paths of a connection are disjoint. Thus, $\gamma_m$ does not contain the edge
$b_{\tau(j)}$.

Now assume that \(\sum_{m}\lambda_m\gamma_m=0\) for some integers \(\lambda_m\). %Since \(\gamma=\sum_{e}\langle\gamma,e\rangle\,e\) for every cycle \(\gamma\), this means that \(\sum_{m}c_m\langle\gamma_m,e\rangle=0\) for each edge \(e\). Take
%\(e=b_{\tau(j)}\):
Then \(0=\sum_{m}\lambda_m\langle\gamma_m,b_{\tau(j)}\rangle=\pm \lambda_j\) because
the summands with \(m\ne j\) vanish, while
\(\langle\gamma_j,b_{\tau(j)}\rangle=\pm1\). Hence, \(\lambda_j=0\) and
\(\sum_{m\ne j}\lambda_m\gamma_m=0\). Discarding \(\gamma_j\) and \(b_{\tau(j)}\) and
repeating the argument with a maximal element of the remaining edges, we get
\(\lambda_m=0\) for all~\(m\).

Now we prove that there exists a unique permutation \(\pi \in \mathcal{S}_k\) such that \(b_r \subset \gamma_{\pi(r)}\) for each \(r=1,\dots,k\). As shown above, the maximal edge \(b_{\tau(j)}\) appears in exactly one cycle \(\gamma_j\). Therefore, \(\pi(\tau(j)) = j\). After discarding \(b_{\tau(j)}\), pick a maximal edge among the remaining ones. By the same reasoning, it appears in exactly one of the remaining cycles. Repeating this process, we reconstruct the permutation $\pi= \tau^{-1}$.

Similarly, for \((a_1, \dots, a_k)\), there exists a unique such permutation, which is the identity.

Then
$$
\mathrm{sgn}_A(\Gamma) = \prod_{j=1}^k \langle \vec{a_j}, a_j \rangle \qquad\text{ and }\qquad
\mathrm{sgn}_B(\Gamma) = \mathrm{sgn}(\tau^{-1}) \prod_{j=1}^k \langle \vec{b}_j, b_j \rangle .
$$
Since \(\mathrm{sgn}(\tau^{-1}) = \mathrm{sgn}(\tau)\), we arrive at the desired identity $$\mathrm{sgn}_A(\Gamma)\mathrm{sgn}_B( \Gamma) = \mathrm{sgn}(\tau) \prod_{j=1}^k \langle \vec{a_j}, a_j \rangle \langle \vec{b}_j, b_j \rangle= \mathrm{sgn}_A^B(\alpha(G)).$$
%\vspace{-0.6cm}
\end{proof}

%Let \((a_1, \ldots, a_{|P|})\) be the edges with the midpoints in ordered set \( A \). Introduce the \emph{intersection matrix} $X_A(\Gamma)$ with the entries $X_A(\Gamma)_{i}^{j}:=\langle a_i,\gamma_j\rangle$.

Introduce the \emph{pairing matrix} $X_A(\Gamma)$ with the entries $X_A(\Gamma)_{i}^{j}:=\langle \gamma_j,a_i\rangle$.

\begin{lemma}\label{l-pi-exists-and-unique}
For any subgraph $G$ valid relative to $A$ and any basis $\Gamma$ of simple cycles in $G$, we have $\mathrm{sgn}_A( \Gamma) = \det X_A(\Gamma).$
%$\det X_A(\Gamma)=\mathrm{sgn}_A( \Gamma)$.
\end{lemma}

\begin{proof}
This is a direct computation:
\begin{multline*}
    \det X_A(\Gamma)
    =\sum_{\pi\in \mathcal{S}_{k}}
    \mathrm{sgn}(\pi)
    \prod_{j=1}^{k} \langle  \gamma_{\pi(j)},a_j\rangle
    %=\sum_{\pi\in S_{k}: u_{\pi(j)}\supset a_j}
    %\mathrm{sgn}(\pi)
    %\prod_{j=1}^{k} \langle {a}_{j}^\circ, \omega_{p_j} \rangle
    %\langle  a_j^\circ,u_{\pi(j)}\rangle
    %\\
    =\sum_{\pi\in \mathcal{S}_{k}:a_j\subset \gamma_{\pi(j)}\text{ for all }j}
    \mathrm{sgn}(\pi)
    \prod_{j=1}^{k}\langle \vec{a}_{j}, a_j \rangle
    = \mathrm{sgn}_A(\Gamma).
\end{multline*}
%
% This implies that $\mathrm{sgn}_P(A, U)$ is invariant under elementary transformations of the basis $U$, namely, adding an integer multiple of one cycle to another one. Changing the direction of a cycle in $U$ or swapping two cycles in $U$ implies changing both signs $\mathrm{sgn}_P(A, U)$ and $\mathrm{sgn}^Q(B, U)$, so that their product does not depend on $U$.
\end{proof}

\begin{lemma} \label{new_lemma}
Let $A$ be an ordered set of midpoints consistent with $P$.
For any subgraph $G$ valid relative to $P$ but not valid relative to $A$ and any basis $\Gamma$ of cycles in $G$, we have $\mathrm{sgn}_A(\Gamma) = 0$.
\end{lemma}

\begin{proof}
Changing the basis $\Gamma$ of $H_1(G;\mathbb{Z})$ to a new one $\tilde \Gamma$ (not necessarily consisting of simple cycles), we can bring the matrix $X_A(\Gamma)$ to the Smith normal form so that $X_A(\tilde \Gamma)^j_i=\lambda_i\delta^j_i$ for some $\lambda_1,\dots,\lambda_k\in\mathbb{Z}$.

Since $G$ is valid relative to $P$ but not valid relative to $A$, it has a simple cycle $\gamma$ that is disjoint from $A$. Express $\gamma$ on the basis $\tilde \Gamma=(\tilde \gamma_1,\dots,\tilde \gamma_k)$ as $\gamma=\mu_1\tilde \gamma_1+\dots+\mu_k\tilde \gamma_k$ for some $\mu_1,\dots,\mu_k\in\mathbb{Z}$. Here, at least one $\mu_i\ne 0$ because $\gamma\ne\emptyset$. We get
$$
0=%\langle a_i,\omega^{p_i}\rangle
\langle \gamma,a_i\rangle
= \sum_j \mu_j%\langle a_i,\omega^{p_i}\rangle
\langle \tilde \gamma_j,a_i\rangle
=\sum_j \mu_j X_A(\tilde \Gamma)^j_i
=\sum_j \mu_j\lambda_i\delta^j_i
=\mu_i\lambda_i.
$$
Since $\mu_i\ne 0$, it follows that $\lambda_i=0$. %We may assume that the transition matrix has a positive determinant, so that
Thus by Lemma~\ref{l-pi-exists-and-unique}, we get $\mathrm{sgn}_A(\Gamma) = \det X_A(\Gamma)=\pm\det X_A(\tilde \Gamma)=\pm\lambda_1\dots\lambda_k=0$.
\end{proof}

\begin{lemma}\label{det_lemma}
For any subgraph $G$ valid relative to both $P$ and $Q$ separately and any basis $\Gamma$ of simple cycles in $G$, we have
\[\mathrm{mult}_P^Q(G) = \left(\sum_{A: a_i \in \omega^{p_i}} \mathrm{sgn}_A( \Gamma)\right)\left(\sum_{B: b_i \in \omega^{q_i}} \mathrm{sgn}_B(\Gamma)\right),\]
where the sums are over all ordered sets \(A\) and \(B\) consistent with $P$ and $Q$, respectively.
\end{lemma}

\begin{proof}
 %\mscommnew{Refer to the definition above, no need to repeat.}
Introduce the \emph{pairing matrix} $X_P(\Gamma)$ with the entries $X_P(\Gamma)_{i}^{j}:=\langle \gamma_j,\omega^{p_i}\rangle$. Then $M^Q_P=X_P(\Gamma)X_Q(\Gamma)^\mathrm{T}$ so that
$$
\mathrm{mult}_P^Q(G)
    = \det M_P^Q
    = \det X_Q(\Gamma) \det X_P(\Gamma).
$$
Now the lemma follows from the computation
\begin{multline*}
    \det X_P(\Gamma)
    =\sum_{\pi\in \mathcal{S}_{k}}
    \mathrm{sgn}(\pi)
    \prod_{j=1}^{k} \langle \gamma_{\pi(j)}, \omega^{p_j} \rangle
    =\sum_{\pi\in \mathcal{S}_{k}}
    \mathrm{sgn}(\pi)
    \prod_{j=1}^{k} \sum_{a_j\subset \gamma_{\pi(j)}}\langle \vec{a}_{j}, \omega^{p_j} \rangle
    \\
    =\sum_{a_1,\dots,a_k}\sum_{\pi\in \mathcal{S}_{k}:a_j\subset \gamma_{\pi(j)}}
    \mathrm{sgn}(\pi)
    \prod_{j=1}^{k}\langle \vec{a}_{j}, \omega^{p_j} \rangle
    =\sum_{A: a_i \in \omega^{p_i}}  \mathrm{sgn}_A(\Gamma).
\end{multline*}
Here, on the right side of the second equality, the inner sum %$\sum_{a_j\subset \gamma_{\pi(j)}}$ \mscommnew{which one? this is not true for all inner sums here}
is over all edges $a_j$ of the cycle $\gamma_{\pi(j)}$, and $\vec{a}_{j}$ denotes the edge $a_j$ oriented in the direction of $\gamma_{\pi(j)}$.
On the left side of the \mscommnew{last} equality, the outer sum is over all ordered sets of edges $(a_1,\dots,a_k)$. We may assume that the edges are distinct and ordered so that $a_j \in \omega^{p_j}$, otherwise, the product vanishes. Then the summation over $(a_1,\dots,a_k)$ is equivalent to the summation over ordered sets of midpoints $A$ consistent with $P$,
and we arrive at the desired expression.
%
%By Lemma~\ref{l-pi-exists-and-unique}, the permutation $\pi$ is uniquely determined by conditions $a_j\subset u_{\pi(j)}$, %summation over $\pi$ can be omitted,
\end{proof}

\begin{proposition}\label{main_proposition}
For any %subsets
$P,Q\subset\{1,\ldots,b\}$ with $|P|=|Q|$,
we have
$F_P^Q = \sum_G \mathrm{mult}_P^Q(G) \cdot c(G),$
where the sum is over all subgraphs $G$ valid relative to both $P$ and $Q$ separately.
\end{proposition}

\begin{proof}
By Lemma~\ref{det_lemma}, for any basis $\Gamma$ of simple cycles in $G$, we have
\[
\sum_G \mathrm{mult}_P^Q(G) \cdot c(G) =
\sum_G \left(\sum_{A: a_i \in \omega^{p_i}} \mathrm{sgn}_A(\Gamma)\right)\left(\sum_{B: b_i \in \omega^{q_i}} \mathrm{sgn}_B(\Gamma)\right)c(G),
\]
where the outer sum is %taken
over all subgraphs $G$ that are valid relative to both $P$ and $Q$ separately, and the inner sums are %taken
over ordered sets $A$ and $B$ that are consistent with $P$ and $Q$, respectively.

By Lemma~\ref{new_lemma}, we can restrict the inner sums to %those
$A$ and $B$ such that $G$ is valid relative to $A$ and $B$, respectively. Then the expression becomes
\[
\sum_{A, B: a_i \in \omega^{p_i}, b_i \in \omega^{q_i}} \sum_G \mathrm{sgn}_A(\Gamma) \cdot \mathrm{sgn}_B( \Gamma) \cdot c(G),
\]
where $A$ and $B$ are ordered sets consistent with $P$ and $Q$, respectively, and the inner sum is %taken
over all subgraphs $G$ valid relative to both $A$ and $B$ separately.
% where the sum is over all subgraphs $G$ valid relative to both $A$ and $B$ separately.

By Lemma~\ref{l-pi-exists-and-unique}, we have
$\mathrm{sgn}_A(\Gamma) = \det X_A(\Gamma)$. This implies that $\mathrm{sgn}_A(\Gamma)$ is invariant under elementary transformations of the basis $\Gamma$, namely, adding an integer multiple of one cycle to another one. Changing the direction of a cycle in $\Gamma$ or swapping two cycles in $\Gamma$ implies changing both signs $\mathrm{sgn}_A(\Gamma)$ and $\mathrm{sgn}_B(\Gamma)$, so that their product does not depend on $\Gamma$.

By Lemma~\ref{sign_lemma}, for every subgraph $G$ valid relative to both $A$ and $B$ separately, there exists a basis $\Gamma$ such that
\[
\mathrm{sgn}_A(\Gamma) \cdot \mathrm{sgn}_B(\Gamma) = \mathrm{sgn}_A^B(\alpha(G)).
\]
Since the left-hand side is independent of $\Gamma$, this identity holds for any basis $\Gamma$.

Therefore,
\[
\sum_G \mathrm{mult}_P^Q(G) \cdot c(G) = \sum_{A, B: a_i \in \omega^{p_i}, b_i \in \omega^{q_i}} \sum_G \mathrm{sgn}_A^B(\alpha(G)) \cdot c(G),
\]
where $A$ and $B$ are ordered sets consistent with $P$ and $Q$, respectively, and the inner sum is %taken
over all subgraphs $G$ valid relative to both $A$ and $B$ separately.

By Corollary~\ref{valid_graphs} and~\eqref{eq-total-observable}, this expression equals $F_P^Q$, completing the proof.
\end{proof}

%Now we are ready to prove Theorem~\ref{th-all-minor-matrix-tree}.
Let us summarize the proof of the all-minor matrix-tree theorem.

\begin{proof}[Proof of Theorem~\ref{th-all-minor-matrix-tree}]
    By Lemma~\ref{connection_F_and_L}, we have
    $
        \det L_P^Q = {F_P^Q}/{\det K^\mathrm{V}_\mathrm{V}}.
    $
    By the classical matrix-tree theorem, the denominator equals
    $
        \det K^\mathrm{V}_\mathrm{V} = \sum_T c(T),
    $
    where the sum is over all spanning trees~$T$ of the network.
    By Proposition~\ref{main_proposition}, the numerator equals
    $
        F_P^Q = \sum_G \mathrm{mult}_P^Q(G) \cdot c(G),
    $
    where the sum is over all subgraphs~$G$ that are valid relative to both~$P$ and~$Q$ separately.
    %
    %Substituting these expressions into the formula for~$\det L_P^Q$, we get the desired result.
\end{proof}

\begin{remark}\label{rem:intersection-matrix} %\mscommnew{Svetlana, do you remember what the issue was in this argument?}
In Theorem~\ref{th-all-minor-matrix-tree}, the sum in the numerator may equivalently be taken over all connected spanning subgraphs with $n+|P|-1$ edges, where $n$ is the number of vertices in the network. %containing all the vertices.
For such subgraphs, we define multiplicities $\operatorname{mult}^Q_P(G)$ literally as above. However, if a subgraph $G$ is not valid relative to $P$ or $Q$, its multiplicity vanishes and it does not contribute.

Indeed, the multiplicity can be equivalently expressed in terms of the pairing matrix $X_P(\Gamma)$ (defined in the proof of Lemma~\ref{det_lemma}) %, so that $M_P^Q = X_P(\Gamma)X_Q(\Gamma)^\mathrm{T}$ and
as $\mathrm{mult}_P^Q(G)=\det X_P(\Gamma)\det X_Q(\Gamma)$.
We claim that if $G$ is not valid relative to $P$, then $\det X_P(\Gamma)=0$. %the matrix $X_P(\Gamma)$ is singular.
The row of $X_P(\Gamma)$ corresponding to $p\in P$ is
$$
\sum_{e\in \omega^p\cap G}
\big(\langle\gamma_1,e\rangle,\dots,\langle\gamma_k,e\rangle\big).
$$
Expanding $\det X_P(\Gamma)$ by multilinearity in the rows, each term corresponds
to choosing one edge $e_p\in\omega^p\cap G$ for each $p\in P$. Since $G$ has $n+|P|-1$ edges, deleting the chosen edges leaves a
subgraph with $n-1$ edges; if this subgraph is not a tree, it must
contain a cycle $\gamma$. %If such a cycle is denoted by $\gamma$, then
The vector of
coordinates of $\gamma$ in the basis $\Gamma$ lies in the kernel of the
corresponding matrix, because $\langle\gamma,e_p\rangle=0$ for all $p\in P$. Therefore the corresponding
determinant is zero. If $G$ is not valid relative to $P$, this happens for every
choice of the edges $e_p$, and therefore $\det X_P(\Gamma)=0$. The same argument
applies to $Q$, so $\operatorname{mult}^Q_P(G)=0$ unless $G$ is valid relative to
both $P$ and $Q$.
\end{remark}

\begin{proof}[Proof of Corollary~\ref{cor-all-minor-matrix-tree}]
In Corollary~\ref{cor-all-minor-matrix-tree}, the sum in the numerator may equivalently be taken over all connected spanning subgraphs $G$ with $n+|P|-1$ edges, where $n$ is the number of vertices in the network. This is analogous to Remark~\ref{rem:intersection-matrix}, but the proof is even simpler: if $G$ is not valid relative to both $P$ and $Q$ separately, then $\imath^*\omega^{p_1},\dots,\imath^*\omega^{p_k}\in H^1(G;\mathbb{Z})$ are linearly dependent, immediately implying vanishing multiplicity of $G$ in $L_P^Q$.

Let us show that the right side in Corollary~\ref{cor-all-minor-matrix-tree} is invariant under circuit refinement. Subdivide an edge $e$ into $e_1,\dots,e_\ell$ with
$
\frac{1}{c(e)}=\sum_{i=1}^{\ell}\frac{1}{c(e_i)}.
$
As shown in the proof of Corollary \ref{cor-cohomological-current}, then the denominator $\sum_T c(T)$ is multiplied by $\prod_i c(e_i)/c(e)$.
Analogously, %to Corollary \ref{cor-cohomological-current},
a %spanning tree or
spanning subgraph $G$ with $n+|P|-1$ edges containing $e$ refines to one with $n+|P|+\ell-2$ edges containing all of $e_1,\dots,e_\ell$,
%multiplying its weight by
thus having the weight
$c(G)\cdot\prod_i c(e_i)/c(e)$. A %spanning tree or
spanning subgraph $G$ with $n+|P|-1$ edges not containing $e$ gives rise to $\ell$ ones with $n+|P|+\ell-2$ edges obtained by omitting exactly one of the edges $e_i$. The total weight of those $\ell$ subgraphs is again $c(G)\cdot\prod_i c(e_i)/c(e)$.
%$
%\prod_i c(e_i)\sum_i\frac{1}{c(e_i)}=\frac{\prod_i c(e_i)}{c(e)}.
%$
%Thus every contribution to the numerator and to the denominator is multiplied by
%the same factor $\prod_i c(e_i)/c(e)$.
The cycles in those $\ell$ subgraphs are obviously identified with cycles in the original subgraph $G$, so the Gram matrix $M(G,\Gamma)$ and the multiplicity $\operatorname{mult}^Q_P(G)$ are unchanged.
Therefore, the numerator and denominator are multiplied by the same factor, and
the quotient is invariant under refinement.

%By Example \ref{ex:reduc}, %after finitely many refinements,
The response matrix, hence $\det L^Q_P$, is invariant under
refinement, and a cohomological circuit can be realized as a source-synchronized circuit after suitable refinement, by Example \ref{ex:reduc}. % $\widetilde S$; %by the preceding paragraph
In the refined circuit, the basis cocycle %representing
$\omega^i$ is the indicator function of %the corresponding
a synchronized source. Thus, for any $G$ and any basis $\Gamma$ %=(\gamma_1,\dots,\gamma_k)$
 of $H_1(G;\mathbb Z)$, the matrices $M(G,\Gamma)$ and the multiplicities $\operatorname{mult}^Q_P(G)$ for both circuits are the same. By Theorem~\ref{th-all-minor-matrix-tree} and Remark~\ref{rem:intersection-matrix}, the corollary follows.
\end{proof}

\begin{remark} After we presented Theorem~\ref{th-all-minor-matrix-tree} at the Third Joint SIAM/CAIMS Annual Meetings in 2025, we learned that its particular case for networks on surfaces had been obtained slightly earlier by W.-Y.~Lam et al.~\cite{Lam-25} by a different and very elegant method. Those authors used a `homotopy argument' to eliminate intersections of certain graphs on the surface (which made the argument strongly tied to that particular case). Despite our primary motivation being the application of parafermionic observables, it would be interesting if this homotopy argument could be adopted to prove the theorem in full generality without the observables.
\end{remark}

\section{Research program on multi-point parafermionic observables}

\label{sec-program}

\subsection{Comparison to known parafermionic observables}

Let us show that for a single source and a single sink, %, i.e., for $|A| = |B| = 1$,
our observable~\eqref{eq-observable} is equivalent to a well-known parafermionic observable in the uniform spanning tree model and places the latter within the general unified framework for constructing such observables.

The most famous %favorite %parafermionic
observable in the %uniform spanning tree
model is defined as follows. Fix a pair of directed edges $e\ne e'$. A spanning tree has a unique oriented path $\alpha$ from the endpoint to the starting point of $e'$. The %desired
observable is the difference between the probability that $\alpha$ passes the edge $e$ in the direction of $e$ and the probability that $\alpha$ passes $e$ in the opposite direction. We %have already
encountered the observable in~\eqref{eq-current-tree} in a form.
% was implicitly introduced in the numerator of~\eqref{eq-current-tree} above. It is defined as an algebraic sum over certain spanning trees, depending on two edges $e_k$ and $e$. %two vertices $v_{i'}$ and $v_{j}$. At the same time,
By Proposition~\ref{prop-current}, it equals %is proportional to
the current through $e$, hence satisfies Kirchhoff's laws, % (I) and~(V),
when all edges have unit conductance and a unit voltage is applied to $e'$. % and $c(e')=1$,

Furthermore, it equals the observable $F(A\leftrightarrow B)$ given by~\eqref{eq-observable} divided by the partition function $\sum_Tc(T)$ if %$k=1$ and
$A$ and $B$ consist of the midpoints of $e$ and $e'$, respectively. This follows from Proposition~\ref{prop-response}, Lemma~\ref{connection_F_and_L}, and the matrix-tree theorem, because the current is exactly an entry of the response matrix. This is not straightforward to see directly: the sums in the numerator of~\eqref{eq-current-tree} and~\eqref{eq-observable} have different numbers of summands, and %nontrivial
cancellations are present.

Let us show that the observable $F(A\leftrightarrow B)$ matches the general construction %of parafermionic observables
from~\cite{Duminil-Copin-13}; cf.~\cite{Ozhegov-25}.
%for an elementary introduction).
Informally, general Smirnov's parafermionic observable is constructed as a sum
\begin{equation}\label{eq-general-observable}
    F(a\to b)=\sum_H \exp(-i\sigma W(\alpha))c(H)
\end{equation}
 over all configurations $H$ with ``disorders'' at two fixed edge midpoints $a$ and $b$, where $c(H)$ is a positive configuration weight, $W(\alpha)$ is the winding (measured in radians) of the specially-chosen path $\alpha$ joining %the two midpoints
 $a$ and $b$ in $H$, and $\sigma$ is a real parameter called spin.

In our case, a configuration $H$ is a spanning forest with a source at the only midpoint of $A$ and a sink at the only midpoint of $B$. The path $\alpha$ is the unique path %(connection)
between the midpoints in the forest. Now, if the graph is planar and the edges $a_1=e$ and $b_1=e'$ containing the midpoints are parallel and co-oriented, then the winding number $W(\alpha)$ of the path is %$0$ or $\pm \pi$
an even or odd multiple of $\pi$ depending on whether
$\langle \vec{a}_1, a_1 \rangle \langle \vec{b}_1, b_1\rangle$ is plus or minus one. In other words, by~\eqref{eq-connection-sign}, we have $\exp(-i\sigma W(\alpha))=\mathrm{sgn}_A^B(\alpha)$ with $\sigma=1$. A similar equation holds for nonparallel edges; only a constant is added to all winding numbers $W(\alpha)$, resulting in an irrelevant overall factor in the observable. Thus, for $\sigma=1$ and $|A|=|B|=1$, sign~\eqref{eq-connection-sign} can be viewed as a natural generalization of the factor $\exp(-i\sigma W(\alpha))$ to the case where the graph is no longer embedded into the plane or a surface, so that the winding is no longer defined. The choice of the orientation of the edges $a_1$ and $b_1$ plays the role of ``local charts'' around the midpoints. So, %observable
$F(A\leftrightarrow B)$ can be viewed as a natural version of% general construction
~\eqref{eq-general-observable} for the uniform spanning-tree model.

Furthermore, the value $\sigma=1$ is consistent with the general relation $\sigma=3/\kappa-1/2$ between the spin $\sigma$ and the parameter $\kappa=2$ of the Schramm--Loewner evolution $SLE_2$ describing the continuum limit of the path \cite[Theorem 3]{Werness-12}. (This should not be confused with the dual $SLE_8$ describing a different path arising in the uniform spanning-tree model, and one needs to be careful with different notational conventions in the literature.) This is also consistent with the relation $n=-2\cos\frac{4\pi(\sigma-1)}{3}$ between the spin $\sigma$ and the dimension $n$ of a loop $O(n)$-model \cite[Sec.~12.5]{Duminil-Copin-13}, as the latter is believed to belong to the same universality class as the uniform spanning tree for $n=-2$. Note that there is a combinatorial interpretation of negative-dimensional tensors, with dimension $n=-2$ playing a special role \cite{Penrose-71}.
%The value $\sigma=1$ also supports viewing an electric network as a discrete analog of an electromagnetic field, which has spin $1$.
%This explains why we prefer to set $\sigma=-1$ while $\exp(-i\sigma W(\alpha))=\mathrm{sgn}_A^B(\alpha)$ remains true for any other odd integer $\sigma$.

Next, the observable $F(A\leftrightarrow B)$ can be viewed as a \emph{multi-point} generalization of the famous 2-point observable we started with. Instead of one path $\alpha$, it involves a connection $\alpha$, that is, a disjoint union of several paths, and ``disorders'' are now decomposed into sources and sinks. Interestingly, although our observable has integer spin $1$, its multipoint version involves the \emph{fermionic sign} $\mathrm{sgn}({\pi})$ in~\eqref{eq-connection-sign}, analogously to \emph{symplectic fermions}. %Such a wrong relation between spin and statistics is not uncommon in quantum field theory.

Our spin-$1$ multi-point observable can be further viewed as a natural analog of spin-$1/2$ observables in the Ising model \cite{chelkak2017revisiting}. Not just are the definitions similar, but the determinant formulae in Lemma~\ref{connection_F_and_L} above are parallel to Pfaffian formulae in \cite{chelkak2017revisiting}. A new feature of our observable is the decomposition of ``disorders'' into sources and sinks, the extension beyond 2-dimensional models (embedded in the plane or on a surface), and the links to cohomology.

%Even when multi-point correlation functions were studied %in the Ising model
%\cite{Chelkak-etal-15}, the underlying observable was usually still a 2-point one, in the sense that the connection $\alpha$ consisted of a single path (and \emph{multi-point fermionic observables} in~\cite{Chelkak-etal-21} did not match general construction~\eqref{eq-general-observable}).
%\mscommnew{[add ref to Chelkak et al]}.

\subsection{Research program on higher-dimensional $n$-point parafermionic observables}

Our results open several new directions for future work, which we now sketch informally. We suggest decomposing further development into the following steps of increasing complexity:
\begin{center}
\begin{tabular}{|c|}
\hline
     other models \\
\hline
\end{tabular}
$\to$
\begin{tabular}{|c|}
\hline
     multi-point observables \\
\hline
\end{tabular}
$\to$
\begin{tabular}{|c|}
\hline
     higher dimensions \\
\hline
\end{tabular}
\end{center}

\textbf{Extension to other models.} Despite having already demonstrated an application of a higher-dimensional multi-point parafermionic observable, we suggest starting with the known 2-dimensional 2-point case on a square or honeycomb lattice and rethinking it from the new viewpoint %known observables
in the other models. The most straightforward is the \emph{loop-erased random walk}, having the same law as the above path $\alpha$ in the uniform spanning-tree model. The next accessible one is the \emph{double-dimer model} on the subdivided graph, %with sources and sinks,
the transition to which we have already established in Sec.~\ref{sec-Temperley} (via a modified Temperley correspondence, which also works for an arbitrary non-planar graph). In the future work by the last author and M.~Dmitriev, we shall show that, for a special choice of complex weights, the resulting 2-point observable reproduces %the one from
the free massive quantum field of spin 0 in 1+1 dimensions in the continuum limit. We shall also link the observable to one in the \emph{six-vertex model}, again with %special
complex weights.

There are also %less precise and
less understood analogies deserving further investigation. Lemma~\ref{det_lemma} above suggests that our spin-$1$ observable in the double-dimer model could be a ``square'' of a yet unknown spin-$1/2$ observable in the \emph{dimer model} (at least in the case of unit edge conductances). Both can be related to multi-point spin-$1/2$ observables in the \emph{Ising model} \cite{chelkak2017revisiting}.
%and Dubedat's $\tau$-functions in the double-dimer model.
%Similarly, one can expect \cite{Masbaum-Vaintrob} a relation between determinant formulae for our spin-$1$ observable in Lemma~\ref{connection_F_and_L} above and Pfaffian formulae for spin-$1/2$ observables in the \emph{Ising model} \cite{Chelkak-etal-21}.
%Our observable for the double-dimer model could have relationship with Dubedat's $\tau$-functions.
The formula $F^Q_P=\sum_G\det X_P(\Gamma)\det X_Q(\Gamma)c(G)$ in Sec.~\ref{sec-conclusion} deserves a probabilistic interpretation. Note that the matrix $X_P(\Gamma)$ naturally connects the cohomology of the network and its subgraph $G$. It would be interesting to find analogs of \emph{cohomological boundary conditions} (see Definition~\ref{def-homological}) for the other models and find their tiling interpretation similar to \cite{PS,S}; cf.~\cite{Dmitriev-Ozhegov,Novikov}.
%Next, it would be intriguing to establish a connection to the \emph{loop $O(n)$ model}, where multi-point generalizations of observable~\eqref{eq-general-observable} have appeared recently \cite{KS-20,KSS-25}. The relation $n=-2\cos\frac{4\pi(\sigma-1)}{3}$ between the spin $\sigma$ and the dimension $n$ \cite[Sec.~12.5]{Duminil-Copin-13} suggests a possible relation between our observable~\eqref{eq-observable} and the (properly interpreted) loop $O(n)$-model for $n=-2$. Note that there is a combinatorial interpretation of negative-dimensional tensors, with dimension $-2$ playing a special role \cite{Penrose-71}.

\textbf{Extension to multi-point observables.} For the multi-point case, one could start with establishing basic properties of observable~\eqref{eq-observable}, such as discrete harmonicity and boundary conditions. As we have already observed, the 2-point observable is discrete harmonic in the sense of satisfying Kirchhoff's laws, % (I) and (V),
and it would be useful to have a direct combinatorial proof of that. Next, one can compute the continuum limit of the multi-point observable; there are well-developed techniques to prove the convergence of discrete harmonic functions to their continuum analogs \cite{Skopenkov-13,BobenkoSkopenkov+2016+217+250,Skopenkov+2023}, and the methods of the present paper support various non-standard boundary conditions. As in \cite{KSS-25}, one can apply this to the computation of certain archetypal probabilities in the continuum limit, for instance, the probability that the above path $\alpha$ lies to the right of one or two given points, studied by R.~Kenyon \cite{Kenyon-11}. Some illuminating techniques of dealing with multi-point observables have already been developed in the Ising \cite{chelkak2017revisiting} and loop $O(1)$-models (M.~Khristoforov and V.~Kleptsyn, private communication). Similar computations could guide the construction of multi-point observables in the double-dimer and other models. %It would be particularly interesting to represent known multi-point observables in the Ising model \cite{Chelkak-etal-21} as a multi-point version of general construction~\eqref{eq-general-observable} involving a collection $\alpha$ of disjoint paths.

\textbf{Extension to higher dimensions.} One could proceed with a 3-dimensional cubic lattice and use the discrete harmonicity and boundary conditions for a rigorous computation of the continuum limit of observable~\eqref{eq-observable}. Discrete harmonicity, in contrast to discrete analiticity, makes perfect sense in any dimension. The main open question is whether this is still useful for the computation of continuum limits; cf.~\cite{kozma2007scaling}. %, for instance, for the loop-erased random walk.
For instance, given four vertices at the boundary of a $n\times n\times n$ cube, one can try to find (the continuum limit of) the probability that the path joining the first two vertices in the uniform spanning tree passes to the left of the path joining the other two vertices. The particular target probabilities, say, in the loop-erased random walk, can be ``reverse-engineered'' from the construction of the observable itself, as in~\cite{KSS-25}. Analyzing the analogy between observable~\eqref{eq-observable} and the one in the Ising model \cite{chelkak2017revisiting}, one can try to extend the latter to 3 dimensions; cf.~\cite{Masbaum-Vaintrob}. This could pave a path to other models and higher dimensions.

%\section{Thompson's principle}
%
%\section{Rayleigh monotonicity}

%\mscomm{------------------------------------------------------------------------------}

%\show\subsection

\subsection*{Acknowledgements}

We are grateful to Vassily Gorbounov, Anton Kazakov, Mikhail Khristoforov, Richard Kenyon, Graeme Milton, and Stanislav Smirnov for useful discussions.

\bibliographystyle{plainurl}
\bibliography{kirchoff}
\end{document}